%% file: Paper_arxiv.tex
\documentclass[a4paper,11pt]{article}

\usepackage{amsfonts}
\usepackage{bbm} 
\usepackage{amsmath}
\usepackage{amssymb}
\usepackage{amsthm} 
\usepackage{mathtools}
\usepackage[T1]{fontenc}
\usepackage[utf8]{inputenc}
\usepackage{pdfpages}
\usepackage{comment}
\usepackage{algorithmic}
\usepackage{algorithm2e}
\usepackage{tikz}
\usetikzlibrary{calc,arrows.meta,patterns}
\usepackage{graphicx}
\usepackage{authblk}
\usepackage{MnSymbol}

\graphicspath{{figures/}} 

\usepackage[left=27mm, right=27mm, top=27mm, bottom=27mm]{geometry} 

\usepackage{hyperref}
\usepackage{cleveref}
\usepackage{caption}
\usepackage{subcaption}

\newenvironment{@abs}[1]{%
  \vspace{4pt}
  \parindent 15pt {\bfseries #1. }\ignorespaces
     }
     {\par\vspace{7pt}}
     
\renewenvironment{abstract}{\begin{@abs}{{Abstract}}}{\end{@abs}}
\newenvironment{keywords}{\begin{@abs}{Key words}}{\end{@abs}}
\newenvironment{AMS}{\begin{@abs}{AMS subject classifications}}{\end{@abs}}

\newtheorem{theorem}{Theorem}
\newtheorem{corollary}{Corollary}

\newtheorem{lemma}[theorem]{Lemma}

\newtheorem{algo}[theorem]{Algorithm}
\newtheorem{testexample}{Test model specification}

\theoremstyle{remark}
\newtheorem{remark}{Remark}

\input{def}

\title{Multi-material structural optimization for additive manufacturing based on a phase field approach}
\author[1]{Luise Blank}
\author[1]{Maximilian Urmann}
\affil[1]{Fakultät für Mathematik, Universität Regensburg, 93040 Regensburg, Germany, email: \href{mailto:luise.blank@ur.de}{luise.blank@ur.de}, \href{mailto:maximilian.urmann@ur.de}{maximilian.urmann@ur.de}}
\date{}

\begin{document}

\maketitle



\begin{abstract}
\input{abstract.tex}
\end{abstract}


\section{Introduction}
\label{sec:Intro}
\input{literature_introduction.tex}
\section{The optimization problem}
\label{sec:optProb}
\input{intro.tex}

\newpage
\section{Analytical results for the optimization problem}
\label{sec:AnalysisOCP}
\input{analysis.tex}

\section{Numerical method for the optimization problem}
\label{sec:VMPT}
\input{VMPT.tex}
\input{PDAS.tex}
\input{Nesting.tex}

\section{Numerical Experiments}
\label{sec:NumExp}
\input{Numerics.tex}

\section{Appendix}
\label{sec:Appendix}
\input{KornLui.tex}

\paragraph{Acknowledgments}
The last author is supported by the Graduiertenkolleg 2339 IntComSin of the Deutsche Forschungsgemeinschaft (DFG, German Research Foundation) - Project-ID 321821685. The support is gratefully acknowledged.

  
\newpage 
\bibliographystyle{plain}   
\bibliography{literature}

\end{document}

%% file: def.tex
\newcommand{\phik}{\varphi^{(k)}}
\newcommand{\Bk}{B^{(k)}}
\newcommand{\Bcjk}{B^{c,j,(k)}}
\newcommand{\Fk}{F^{(k)}}
\newcommand{\Fcjk}{F^{c,j,(k)}}
\newcommand{\deltax}{\delta_H}

\newcommand{\dx}{\text{ d}x}
\newcommand{\dy}{\text{ d}y}
\renewcommand{\dh}{\text{ d}h}
\newcommand{\dH}{\text{ d}\mathcal{H}^{d-1}}
\newcommand{\intO}{\int_{\Omega}}

\newcommand{\intG}{\int_{\Gamma_{N}}}
\newcommand{\dphi}{\dfrac{\text{d}}{\text{d}\varphi}}
\newcommand{\dphitwo}{\dfrac{\text{d}^2}{\text{d}\varphi^2}}

\newcommand{\Pad}{\Phi_{\text{ad}}}

\newcommand{\Hzero}{H^1_{0}(\Omega,\mathbb{R}^N)}

\newcommand{\R}{\mathbb{R}}
\newcommand{\N}{\mathbb{N}}

\newcommand{\HON}{H^1(\Omega)^N }
\newcommand{\LinfN}{L^\infty(\Omega)^N}
\newcommand{\mass}{\mathfrak{m}}
\newcommand{\G}{\mathbf{G}}

\def\fint#1{\strokedint#1}

%% file: abstract.tex
A topology optimization problem in a phase field setting is considered to obtain rigid structures, which are resilient to external forces and constructable with additive manufacturing. 
Hence, large deformations of unsupported overhangs due to gravity shall be avoided during construction. 
The deformations depend on the stage of the construction and are modelled by linear elasticity equations on growing domains with height-dependent stress tensors and forces. 
Herewith, possible hardening effects can be included. 
Analytical results concerning the existence of minimizers and the differentiability of the reduced cost functional are presented in case of a finite number of construction layers. 
By proving Korn's inequality with a constant independent of the height, it is shown that the cost functional, formulated continuous in height, is well-defined. 
The problem is numerically solved using a projected gradient type method in function space, for which applicability is shown. 
Second order information can be included by adapting the underlying inner product in every iteration. Additional adjustments enhancing the solver's performance, such as a nested procedure and subsystem solver specifications, are stated. 
Numerical evidence is provided that for all discretization levels and also for any number of construction layers, the iteration numbers stay roughly constant. 
The benefits of the nested procedure as well as of the inclusion of second order information are illustrated. 
Furthermore, the choice of weights for the penalization of deformation during the construction is discussed. 
For various problem settings, results are presented for one or two materials and void in two as well as in three dimensions.

%% file: literature_introduction.tex
Structural shape and topology optimization aims to find an optimal material distribution in a fixed domain, such that a given cost functional is minimized while satisfying physical or geometrical constraints. 
Multiple different mathematical frameworks have been developed to approach these kinds of problems.
We refer to \cite{SigmundReview} for a general overview of this topic.
SIMP and RAMP methods, as described e.g. in \cite{BenSig03,RAMP}, use a homogenization approach introduced in \cite{bendsoe1988HomogenizationApproach} to alter a density variable to a 0--1 solution. 
In level set methods, the boundary of the material is tracked by a sharp interface given by the level set of a function, see e.g. \cite{Burger2003LevelSet,Wang2003LevelSet}. 
Topology changes are possible by including the topological derivative as performed in \cite{AllaireTopologicalDeriv,Burger2004TopologicalDerivative}. 
Instead of sharp transitions from one material to another, phase field approaches use diffuse interfaces.
This results in an optimization problem posed in a vector space.
Consequently, the phase field ansatz is capable of topology changes and of nucleation of new holes.
Limitations given for shape or topological derivatives are absent.
Also remeshing is not required.
Originally, phase field models are introduced to describe the phase evolution in binary alloys in \cite{AllenCahn,CahnHilliard}. 
For structural optimization, this approach is proposed in \cite{BourdinChambolle} and is widely used and extended to multi-material structural topology optimization, see e.g. \cite{ABCHRR,BFGS14,BurgerStainko,Dede2012,Garcke4DPrinting,Takezawa2010,Wang2004PhaseField} and references therein.
For numerical simulations, pseudo time stepping (discretized gradient flow) is typically employed, see e.g. \cite{CHEN2026SemiImpl,WANG2026Uncert}.
Nevertheless, for this method no convergence result is known for these kinds of problems.

In this article, we aim to optimize the shape and topology of rigid structures, which are resilient to external forces and also constructable with additive manufacturing techniques.
Additive manufacturing, shortly AM, refers to a set of construction methods that fabricate a component by decomposing it into layers which are produced one after the other.
For a general overview of different techniques, we refer to \cite{ReviewAM}.
A wide range of applications can be found in \cite{ReviewAMNgo} and in the references therein.
These techniques enable cost-effective production and the construction of topologies with higher complexity than it is possible with traditional manufacturing methods.
Hence, greater design flexibility and improved product quality is gained.
With AM, novel constraints arise, e.g. the producibility of overhangs.
A collection of challenges for the shape and topology optimization concerning AM is gathered in \cite{TrendsAM}.
A common strategy to ensure constructability is to penalize large overhang angles by including a suitable anisotropy. 
Here, side effects, such as the dripping effect or stretched interfaces, are observed, e.g. in \cite{AllaireOverview2,GLNS}.
A strategy to counter the dripping effect is proposed in \cite{PreventDripping} using additional filter techniques.
In \cite{gaynor2016topology}, a combination of RAMP and overhang projection is studied to avoid overhang angles exceeding a pre-set limit. 
A different approach that ensures stability during the manufacturing process is to construct removable supports alongside the main structure, e.g. see \cite{ReviewSupports}.
Removable supports are successfully incorporated in different manners:
In \cite{AllaireBogoselSupports1,Kuo2018support}, for a given desired structure, the optimal removable support is requested, while in \cite{AllaireBogoselSupports2,langelaar2018combined,MIRZENDEHDEL20161}, main structure and supports are optimized simultaneously.
One drawback of supports is the amount of additional material, which increases production costs.
Furthermore, the removal process is labour-intensive and carries the risk of damaging the fabricated component.

Although the work in this paper also includes the possibility of external supports, we aim to obtain structures that support their own weight throughout the production phase.
Rather than penalizing the geometry of overhangs itself, we penalize the deformations caused by gravity during construction.
That means
  we assess constructability by examining the
  deformations that occur while building the structure.
This method is introduced in a sharp interface setting in \cite{AllaireOverview1}, and numerically solved with a level set approach in \cite{AllaireOverview2}.
The ansatz is adapted to a diffuse interface framework and solved via discretized gradient flow, see \cite{WIAS}.

We generalize both approaches in a diffuse interface setting by considering multiple materials and by allowing height-dependent elasticity tensors.
Furthermore, we discuss possible choices of the employed weights in the cost functional.
As solver, a projected gradient type method is applied, for which a convergence result is given in \cite{VMPT}.

\medskip
This paper is organized as follows:
In the following section, we introduce the structural optimization problem including additive manufacturing.
The phase field ansatz and the corresponding linear elasticity equations for the displacement fields in a multiphase setting are presented.
To include possible hardening of the structure, the elasticity tensors may depend on the construction level and also on the spatial location.
The construction of the structure is considered with continuous growth in height, while the layer-by-layer ansatz is viewed as its discretization. 
For both cases, the optimization problem is formulated.
In Section \ref{sec:AnalysisOCP}, we provide analytical results for the optimization problem.
First, we show that the cost functional, which is reduced to the phase field variable, is well-defined for the layer-by-layer approach, and, under some mild additional assumptions, also for the formulation that is continuous in height.
Therefor, the necessary version of Korn's inequality is proved in the appendix. 
This also provides the boundedness of the cost functional independent of the number of employed layers.
Fréchet differentiability of the control-to-state operators for any construction level, and thereby of the reduced cost functional for the layer-by-layer approach, is shown.
The derivatives are listed for application in the numerical solver.
The numerical solution method is presented and discussed in Section \ref{sec:VMPT}.
More precisely, we cite the variable metric projection type method, shortly VMPT method, given in \cite{VMPT} and verify the assumptions on the optimization problem and the applied metrics for its convergence.
Further computational aspects enhancing the solver's performance, e.g. the use of second order information, the solver for the quadratic subproblem, and a nested algorithm, are discussed. Also the spatial discretization is given.
In Section \ref{sec:NumExp}, we present various numerical experiments.
The capabilities of the VMPT method are demonstrated with the independence of the number of layers and of the discretization level, as well as with the use of second order information.
Furthermore, the drastic speed-up of utilizing the nested algorithm is presented.
A parameter study investigates the impact of the choice of weights on the obtained solution. 
Additional variations of the model, such as building plate locations and hardening materials, are considered.
Structural optimization with three materials is performed.
Using this multiphase setting, also the inclusion of removable supports into the optimization process is studied.
Finally, solutions in three dimensions are presented and discussed.

%% file: intro.tex
The structure is modelled using a phase field approach with diffuse interfaces between the materials, as this allows for topological changes during the optimization process.
The use of phase fields in topology optimization was introduced in \cite{BourdinChambolle}.
In a design domain $\Omega\subseteq \R^d$, $d\in \{2,3\}$, with a piecewise Lipschitz boundary $\Gamma = \partial\Omega $, a phase field $\varphi:\Omega\to\R^N$ with $\varphi\geq 0$ and $\sum_{i=1}^N \varphi_i = 1$ is employed to describe the distribution of $N$ materials. 
The $i-$th component $\varphi_i(x)$ denotes the fraction in which the $i-$th material is present at $x\in\Omega$. Here, the $N-$th component corresponds to the void region.
The amount of each material is prescribed by a vector $\mass \in \R^N$.

The interface thickness is controlled by a parameter $\varepsilon>0$.
Moreover, the perimeter of the interface can be approximated by the Ginzburg-Landau energy 
\begin{align}
\label{Ginz}
	E(\varphi):=\int_\Omega\left\{ \frac{\varepsilon}{2}|\nabla \varphi|^2 + \frac{1}{\varepsilon}\Psi(\varphi)\right\}\dx
\end{align}
with an appropriate potential $\Psi$.
Consequently, the phase field lies in the set
\begin{align}
\label{eq:PhiAD}
  \Phi_{ad} = \Big\{ \varphi\in \HON \mid \varphi(x) \in \mathbf{G} \text{ a.e.; }
  \fint_{\Omega} \varphi \dx = \mathfrak{m} \Big\},
\end{align}
where $\mathbf{G} = \{ v\in\mathbb{R}^N \mid \sum_{j=1}^N v_j = 1,\ v_j\geq 0  \ \forall j=1,\ldots , N \}$ is given by the Gibbs simplex.
Based on the results in \cite{BFGRS14}, we choose an obstacle potential by setting $\Psi(x) =\tfrac12 x^T Ax$ for $x\in G$ and $\Psi(x) =\infty$ elsewhere for some symmetric matrix $A$ having at least one negative eigenvalue, cf. \cite{BFGS14,EL91}.
In our numerical experiments, we choose
\begin{align}\label{eq:multipot2}
	A = \begin{pmatrix}0 & 0.1 & \cdots & 0.1 & 1 \\ 0.1 & 0 & \cdots & 0.1 & 1\\ \vdots & & \ddots & & \vdots \\ 0.1 & 0.1 & \cdots & 0 & 1\\ 1 & 1 & \cdots & 1 & 0\end{pmatrix}\in \R^{N\times N}.
\end{align}
The minima of $\Psi$ are attained at the unit vectors, corresponding to pure materials. To only penalize impure regions, the diagonal entries are chosen as $0$, while $A_{ij}$ for $i\neq j$ controls the penalization of mixtures of material $i$ and $j$, where mixtures with the void phase get penalized more. This choice ensures a straightened boundary between material and void and prevents 120° angles, which may lead to the formation of cracks, see \cite{BFGS14} and references therein.%
An internal force $f\in L^2(\G\times\Omega)^d$ may act on the finally constructed structure, e.g. gravitational effects and tension forces within the material, and an external force $g\in L^2(\G\times\Gamma_N)^d$ on $\Gamma_N\subseteq \Gamma$.
With the dependence of $f$ and $g$ on values in $\G $ it is possible to prevent e.g. the influence of gravity and tension forces on void regions of the structure.
The forces lead to a displacement field $u$ of the structure.
Given that the structure shall be fixed on $\Gamma_D\subseteq\partial\Omega$, we have $u \in H^1_{D}(\Omega)^d:= \{v \in H^1(\Omega)^d \mid v|_{\Gamma_D}=0 \}$, and $u$ is given as the solution of the linear elasticity equations (also called {\em mechanical system} \eqref{eq:WeakFormLinElast}, see e.g. \cite{AllaireOverview2,WIAS}):
\begin{align}
\label{eq:WeakFormLinElast}
  \intO \mathcal{C}(\varphi)\mathcal{E}(u):\mathcal{E}(\xi) \dx =& \intO f(\varphi,.) \cdot \xi \dx 
  + 
  \intG g(\varphi,.) \cdot \xi\dH\ 
 \forall\xi\in H^1_{D}(\Omega)^d, \tag{MS}
\end{align}
where $\mathcal{E}(u) := \frac 1 2 (\nabla u + \nabla u^T)$ is the linearized strain tensor and the tensor valued mapping $\mathcal{C}$ is a suitable interpolation, given in \eqref{eq:StressTensor}, of the stiffness tensors $\mathcal{C}( e_i)=\mathcal{C}^i$, $i=1,\ldots , N$,  of the different materials. The material is assumed to be isotropic, i.e., the material properties are uniquely determined by the Lamé parameters $\lambda_i,\mu_i$ by $\mathcal{C}^iM = \lambda_i \text{tr}(M) + 2\mu_i M$ for all matrices $M \in\mathbb{R}^{d\times d}$.
Here, we model the void, in accordance to a result in \cite{BFGS14}, as a very soft ersatz material and hence set
$\mathcal{C}^N  = \varepsilon^2 \tilde{\mathcal{C}}^N$, where $\tilde{\mathcal{C}}^N$
is a fixed elasticity tensor.
To make impure regions less favourable in the optimization process, we choose in particular as in \cite{BFGRS14} 
\begin{align}
\label{eq:StressTensor}
\mathcal{C}(\varphi) = \sum_{i,j=1}^N \varphi_i\varphi_j\mathcal{C}^{max(i,j)},
\end{align}
where the indices of the materials are ordered from stiffest to weakest.
A similar kind of interpolation is used in case of two phases in \cite{BC06,PRW12,WR12} and for the SIMP approach for one material and void in \cite{BenSig03}. 
The additive manufacturing (construction) process  starts
on the so-called building plate
$\Gamma_B\subset\Gamma$.
The building direction shall be $e_d = (0,\dots, 0,1)^T$.
Let $H$ be the maximal height of $\Omega$ in building direction, and
$h$ with $0<h\leq H$ denote the height of the
currently manufactured structure.
This intermediate structure is then contained within $\overline{\Omega_h}$, where
\begin{align}
\label{eq:InterStruc}
\Omega_h &:= \Omega \cap 
\{x\in\mathbb{R}^d\ \mid 0 < x_d < h \},
\end{align}
and $x_d$ denotes the $d-$th component of $x$. 
On the intermediate structure --given by $\varphi|_{\Omega_h}$-- only gravity $f^c(h, \varphi, .)$ acts as force.
The resulting displacement fields
$u^c(h,.)\in H^1_B(\Omega_{h})^d$ with $H^1_B(\Omega_{h})^d:= \{v \in H^1(\Omega_h)^d \mid v|_{\Gamma_B}=0 \}$ change
with the current construction height $h$
(see e.g. Figure \ref{fig:Displacements})
and
fulfil the linear elasticity equations (also called {\em additive manufacturing system} (AMS), see, e.g., \cite{AllaireOverview2,WIAS}) 
\begin{align}
\label{eq:WeakFormAMCS}
\int_{\Omega_{h}}\mathcal{C}^c(h,\varphi,.)
\mathcal{E}(u^c(h,.)):\mathcal{E}(\xi) \dx
= \int_{\Omega_{h}}f^c(h,\varphi,.) \cdot \xi \dx 
\quad \forall \xi\in H^1_B(\Omega_{h})^d \; . \tag{AMS}
\end{align}
This formulation extends the existing models by allowing the stress tensor of (MS) to differ from (AMS). Moreover, we allow $\mathcal{C}^c$ to depend on $h$ and also on the space variable. This can be used e.g. to model the hardening of the material during the construction process.
The {\em self-weight} of the structure up to height $h$ denotes, in accordance to \cite{AllaireOverview2}, the mean compliance with respect to gravity $f^c$, i.e.
\begin{align}
\label{selfweight}
F^c(h,\varphi, u^c(h,.)) :=
\int_{\Omega_h} f^c(h,\varphi,.) \cdot u^c(h,.) \dx.
\end{align}
To avoid collapsing overhangs or large deformations during the construction, one penalizes the self-weights with
\begin{align}
\label{eq:LayerComp1}
W(\varphi, u^c) :=
\int_0^H
\omega(h) 
F^c(h,\varphi, u^c(h,.)) 
\dh,
\end{align}
where $\omega \in C((0,H])$ is a nonnegative weight function.
Then, using the regularization with the Ginzburg-Landau energy as an approximation of the perimeter, the problem of distributing $N$ materials in such a way that, on one hand, the mean compliance of the mechanical system
\begin{align}
\label{eq:MeanComp}
F(\varphi, u):=\intO f(\varphi,.) \cdot u \dx + \intG g(\varphi,. ) \cdot u\dH
\end{align}
is as low as possible, while on the other hand, large deformations during the construction are avoided, can be formulated by
\begin{align}
\tag{$P$}
\label{eq:ContCostFunc}
\min_{\varphi \in \Pad} & J(\varphi, u, u^c) =
F(\varphi, u) 
+ \beta_{1}  W(\varphi,u^c)
+ \beta_{2}E(\varphi) \\ 
\text{s.t.} \quad
& u \text{ is the solution to } \eqref{eq:WeakFormLinElast} \text{, } \nonumber  \\
& u^c(h,.) \text{ is the solution to }
\eqref{eq:WeakFormAMCS}
\text{ for all } h \in (0,H] , \nonumber
\end{align}
with parameters $\beta_1 \geq 0$ and $\beta_2>0$.
For the layer-by-layer manufacturing with $M$ layers of thickness $\delta:=H/M$, it is assumed that each new layer is produced instantaneously. 
Discretizing the integral over $(0,H)$ using a Riemann sum yields the penalization
of deformations during construction
for the layer-by-layer approach by
\begin{align}
\label{eq:LayerCompDisc}
W_\Delta (\varphi, u^c_\Delta) :=
\delta
\sum_{k=1}^M \omega(k\delta) F^c(k\delta,\varphi, u^c_k) .
\end{align}
Here we use the notations $u^c_k:=u^c(k\delta,.)$ and $u^c_\Delta:=(u^c_1,\ldots,u^c_M)^T$.
Consequently, the optimization problem for the layer-by-layer manufacturing reads as
\begin{align}
\tag{$P_\Delta$}
\label{eq:OptimProb}
\min_{\varphi \in \Pad} & J_\Delta(\varphi, u, u^c_\Delta)=
F(\varphi, u) +
\beta_{1}  W_\Delta(\varphi,u^c_\Delta) +
\beta_{2}E(\varphi)
\\ 
\text{s.t.} \quad
& u \text{ is the solution to } \eqref{eq:WeakFormLinElast} \text{, } \nonumber  \\
& u^c_k \text{ is the solution to }
\eqref{eq:WeakFormAMCS}
\text{ with } h=k\delta \quad \forall k=1,\ldots , M .
\nonumber
\end{align}

Let us mention that the above approach incorporates the ansatz of \cite{Allaire2,AllaireOverview2} in their practical applications by setting $\omega(h) \equiv 1$ and using independently of $h$ for $ f^c$ the gravity force $f_{grav}(\varphi, .):= -(1-\varphi_N) e_d$.
It also includes the approach of \cite{WIAS} for the layer-by-layer model by setting $\omega (h) = 1/h$ and $f^c(h,\varphi ,.):= f_{grav}(\varphi) \chi_{\Omega_h/ \Omega_{h-\delta}}$ where $\chi$ denotes the indicator function.
We favour the inclusion of gravity on the whole currently constructed domain $\Omega_h$ and the scaling of the self-weight by $ \omega(h):= \tfrac{1}{|\Omega_h|}$ in order to penalize deformations more evenly throughout the construction process.
The different choices are discussed in the following sections.

When only two phases are present, e.g. one material and void, the system can be reduced to one phase field variable $\tilde \varphi := \varphi_1-\varphi_2 \in [-1,1]$ (and vice versa, it holds $\varphi=\tfrac12(1+\tilde \varphi,1-\tilde \varphi)^T$).
Then the Gibbs simplex is substituted by $[-1,1]$ and the mass of $\tilde \varphi $ is given by
$\tilde{\mathfrak{m}}:= \fint_\Omega \tilde \varphi \dx = \mathfrak{m}_1-\mathfrak{m}_2$.
While $\tilde f(\tilde \varphi) := f(\varphi) $, etc. is employed for the reformulated problem, the Ginzburg-Landau energy is not changed, i.e. 
$\tilde{E}(\tilde{\varphi}) = \int_{\Omega} \frac{\varepsilon}{2} |\nabla \tilde{\varphi}|^2 + \frac{1}{\varepsilon} \tilde{\psi}(\tilde{\varphi}) dx
= 2E(\varphi)$
with $\tilde{\psi}(x)=\tfrac12 (1-x^2) $ on $[-1,1]$ and $\infty$ otherwise.
Thus it holds $\tilde\beta_2:=\tfrac12\beta_2$. 

%% file: analysis.tex
Previous results for the phase field approach to structural topology optimization without the construction phase can be found, e.g., in \cite{BFGS14,VMPT,BurgerStainko,Dede2012,PRW12}. 
In the thesis \cite{DissRupprecht}, Chapter 6, this approach is extended, such that the cost functional and the right-hand side of the linear elasticity equation may depend, in addition to the space variable, also on $\varphi$.
Therein, the existence of global minimizers, differentiability of first and second order, as well as $\Gamma-$convergence results are shown.
In our paper, the model is extended by the construction phase, i.e. by $W$ and $u^c$.
In the numerical experiments, we also allow that the elasticity tensors depend, in addition to $\varphi$, on the spatial variable $x$ and on the current construction state $h$.
However, for the following analysis, which is based on \cite{DissRupprecht}, we assume
$C^c(h,\varphi,.)\equiv C^c(\varphi)$.
We refer to this literature for the precise assumptions on the involved functions and operators, which are fulfilled for our class of examples and omitted for shortness.We like to emphasize that
the coercivity constant $\theta $
of
 $\int_{\Omega}\mathcal{C}(\varphi)
 \mathcal{E}(\eta):\mathcal{E}(\xi) \dx$
 and the  coercivity constants
 $\theta^c(\Omega_h)$
of $\int_{\Omega_{h}}\mathcal{C}^c(\varphi)
 \mathcal{E}(\eta):\mathcal{E}(\xi) \dx$,
 are independent of 
 $\varphi$.
We obtain the following results concerning the control-to-state operators corresponding to the equations \eqref{eq:WeakFormLinElast} and \eqref{eq:WeakFormAMCS}.
\begin{theorem}
The control-to-state operators
 $S:\HON \cap \LinfN \rightarrow H^1_D(\Omega)^d $,
 where \\ 
 $S(\varphi):=u$ is the solution of \eqref{eq:WeakFormLinElast},
 and
 $S^c(h,.):\HON \cap \LinfN \rightarrow  H^1_B(\Omega_{h})^d$,
 where $S^c(h,\varphi):=u^c(h,.)$ is the solution of \eqref{eq:WeakFormAMCS},
are well defined and twice continuously
Fréchet differentiable
for all $h\in (0,H]$. 
Furthermore, the a priori bounds 
\begin{align}
\label{apriori1}
  \| u \| _{ H^1_D(\Omega)^d} &\leq
                                c(\Omega) (\| f(\varphi,.)  \| _{L^2(\Omega)} + \| g(\varphi,.)  \| _{L^2(\Gamma_N)}, \\
  \label{apriori}
\| u^c(h,.) \| _{ H^1_B(\Omega_{h})^d} &\leq c^c(\Omega_h) \| f^c(h,\varphi ,.)  \| _{L^2(\Omega_h)},
\end{align}
hold with constants $c$ and $c^c$ independent of $\varphi$.
\end{theorem}
\begin{proof}
The application of the Lax-Milgram Theorem provides the well posedness
of the control-to-state operators and the a priori bounds. Details can be found in \cite{BFGS14}.
Theorem 6.11 in \cite{DissRupprecht} provides the continuous Fr\'echet differentiability of second order of the operator\\
 $S:\HON \cap \LinfN \rightarrow H^1_D(\Omega)^d $ as well as of each map $\varphi|_{\Omega_h} \mapsto u^c(h,.)$ given by \eqref{eq:WeakFormAMCS} from $ H^1(\Omega_h)^N\cap L^\infty(\Omega_h)^N$ to $H^1_B(\Omega_{h})^d$.
Considering the restriction map $\varphi \mapsto \varphi|_{\Omega_h}$ yields immediately that $S^c(h,.):\HON \cap \LinfN \rightarrow  H^1_B(\Omega_{h})^d $
has continuous Fr\'echet derivatives of second order for all $h\in (0,H]$.
\end{proof}
Consequently,
  we obtain the following result for the mean compliance and the self-weights.
\begin{corollary}
The functionals $\bar F, 
\bar F^c(h),
\bar W_\Delta: \HON \cap \LinfN \rightarrow \R$ with 
\begin{align}
\label{FcDelta}
&\bar F(\varphi):=
F(\varphi, u(\varphi))
,                    \; 
\bar F^c(h) (\varphi):=
F^c(h, \varphi, u^c(h,\varphi))
,                  \;
\bar W_\Delta(\varphi):=
W_\Delta(\varphi, u^c_\Delta(\varphi))
\end{align}
are well-defined and twice
continuously Fréchet differentiable.
\end{corollary}

Further results can be obtained for rectangular design domains $\Omega $. 

\begin{theorem}\label{Wc}
 Let $\Omega=\Omega_B\times (0,H)$ with $\Omega_B=(0,l_1)\times \ldots \times (0,l_{d-1})$
  and $ \Gamma_B=\Omega_B\times \{ 0 \}$.
Moreover, let
$\int_0^H\omega(h)|\Omega_h| \dh $ exist,
and $\| f^c(h,\varphi ) \|_{ L^\infty(\Omega_h)}$
be uniformly bounded for 
$h\in (0,H]$ and $\varphi \in \Pad$.
Then, $\bar W: \Pad \rightarrow \R $ with $ \bar W(\varphi):= W(\varphi, u^c(\varphi)) $ is well-defined,
and $\bar W_\Delta (\varphi)$ is bounded independently of the number of layers $M$
for $\varphi \in \Pad$.
\end{theorem}
\begin{proof}
  Due to \eqref{apriori}, it holds for all $h\in (0,H]$
\begin{align}
\label{Wapriori}
| F^c(h,\varphi, u^c(h))|
&\leq c^c(\Omega_h)
\| f^c(h,\varphi )  \|^2 _{L^2(\Omega_h)} .
\end{align}
For the given $\Omega$, the
constants in the Poincar\'e inequality and in Korn's inequality can be chosen independently of $h$ according to Lemma \ref{PoinIE} and Theorem \ref{kornIE} in the Appendix. 
Then the above mentioned coercivity 
constants $\theta^c(\Omega_h)$ are
independent of $h$.
Hence, there exists $C>0$ with $c^c(\Omega_h) \leq C<\infty $
for all $0<h\leq H$.
With $\| f^c(h,\varphi )  \|^2 _{L^2(\Omega_h)} \leq |\Omega_h| \| f^c(h,\varphi ) \|^2_{ L^\infty(\Omega_h)}$, this yields the remaining assertions under the additional assumptions on $f^c$ and $\omega$.
\end{proof}
Typically $f^c$ is the gravitational force, i.e. $f^c(h, \varphi,.)=-(1-\varphi_N)c_{grav}e_d$, and therefore satisfies the condition on $f^c$ for $\varphi \in \Pad$. 
The assumption on $\omega$
  in Theorem \ref{Wc}
  motivates our choice of $\omega(h)=1/|\Omega_h|$.
In addition, this choice intends to
balance
the dependence of  the self-weights
$F^c(h,\varphi, u^c(h))$
on the size of the domain $\Omega_h$
in 
$ W(\varphi, u^c(\varphi))  $, such that the
deformations on smaller regions are penalized
the same as deformations on larger regions.
With $\omega(h)\equiv 1$ as in \cite{AllaireOverview1,Allaire2}, the deformations at lower construction stages are penalized less than for later stages.
For the choice $\omega(h)=1/|\Omega_h|$
with $f^c(h,\varphi ,.):= -(1-\varphi_N)c_{grav} \chi_{\Omega_h/ \Omega_{h-\delta}}e_d$ as in \cite{WIAS}
and $\Omega$ as in Theorem \ref{Wc},
we obtain
$\lim_{M\to \infty} \max_{k=1,\ldots,M}\| u^c_k \| _{ H^1_B(\Omega_{h})^d} =0 $
using the estimate \eqref{apriori}.
Hence, this model assumes that using thinner layers
leads to less deformations during the construction.
Also the penalization of the self-weights converges to zero when the number of layers $M$ tends to infinity,
i.e. $\lim_{M\to\infty}\bar W_\Delta (\varphi)=0$.
Numerical evidence is given in Subsection \ref{sec:SolverSpec}.

\begin{theorem}
  \label{Analysis}
There exists a global minimizer for the given additive manufacturing structural optimization problem \eqref{eq:OptimProb}.
Moreover, the cost functional of the reduced problem formulation
\begin{align}
\label{eq:RedCostFunc}
j_\Delta(\varphi) &:= J_\Delta(\varphi, u(\varphi), u^c_\Delta(\varphi )) 
\end{align}
is twice continuously  Fr\'echet differentiable in $\HON \cap \LinfN $ and $j_\Delta \in C^{1,1}(\Pad) $.
\end{theorem}
These results also hold true for the case of two phases reduced to one phase variable $\tilde \varphi $. 

In the following, the derivatives are stated
as they are needed for the numerics in Section \ref{sec:VMPT}.
The existence and uniqueness of the adjoints $p$ to \eqref{eq:WeakFormLinElast} and of $p^c(h,.)$ to \eqref{eq:WeakFormAMCS} for each $h\in (0,H]$ follows directly due to the symmetry of the left hand side of the state equations and the linearity of the cost functional with respect to $u$ and $u^c(h,.)$, respectively.
For $p$, the right-hand side of the corresponding state equation coincides with $F$
in the cost functional,
and hence $p=u$, while for $p^c(h,.)$ it scales with $\beta_1 \omega(h)$ which results in $p^c(h,.)= \beta_1 \omega(h) u^c(h,.)$, respectively for $p^c_k$ we have $p^c_k= \delta \beta_1\omega(k\delta ) u^c_k$ for $k=1,\ldots ,M$
(see e.g. \cite{BFGS14,DissRupprecht}
  for $p$).
This simplifies the expressions of the derivatives. We obtain the first order Fréchet derivatives for $\zeta\in\HON\cap\LinfN$:
\begin{align}
&\dphi F(\varphi,S(\varphi))[\zeta] 
 =
 \intO 2 f_\varphi (\varphi,\cdot)[\zeta] \cdot u - \mathcal{C}'(\varphi)[\zeta]\mathcal{E}(u):\mathcal{E}(u) \dx
 \label{Fdiff}\\
& \qquad  + \intG 2 g_\varphi(\varphi,\cdot)[\zeta] \cdot u\dH \ ,
\nonumber\\
&\dphi E(\varphi)[\zeta]  =
\intO \varepsilon \nabla\varphi:\nabla\zeta + \tfrac{1}{\varepsilon} \nabla \Psi(\varphi) \cdot \zeta\dx \ ,
\label{Ediff}
\end{align}
and the derivative of the self-weight $F^c$ with respect to $\varphi $ has the same structure as for $F$ given in \eqref{Fdiff}. 
Positive definite parts of the second derivative of the reduced cost functional, which induces an inner product, are employed in the numerics. 
Hence, we list the formulae here, again omitting $F^c$ due to the same structure as $F$.
To compute the second order derivatives, the linearized state equations for
\eqref{eq:WeakFormLinElast}
for a $\zeta\in\HON\cap\LinfN$
\begin{align}
\label{eq:LinMS}
\intO  \mathcal{C}(\varphi)\mathcal{E}(
w_{\zeta}):\mathcal{E}(\xi) \dx =
&- \intO \mathcal{C}'(\varphi)[\zeta]\mathcal{E}(u):\mathcal{E}(\xi) \dx\\
 &+ \intO f_{\varphi} (\varphi,.)[\zeta] \cdot \xi \dx + \intG g_{\varphi}(\varphi, .)[\zeta] \cdot \xi\dH
\; 
\nonumber
\end{align}
and the corresponding linearized state equation for
\eqref{eq:WeakFormAMCS}
have to be solved providing \\
$w_\zeta= S'(\varphi)[\zeta]$ and $w^c_\zeta(h)= (S^c(h))'(\varphi)[\zeta]$ (and with them, the linearized adjoint states $\hat p =w_\zeta$, $\hat p^c(h)= \beta_1\omega(h) w^c_\zeta(h)$, $\hat p^c_k= \delta \beta_1\omega(k\delta) w^c_\zeta(k\delta)$.)
Existence and a priori estimates are given in Theorem 6.9 in \cite{DissRupprecht}, while in Theorem 6.44 the following representations of the second derivatives for $\zeta,\eta\in\HON\cap\LinfN$ are rigorously shown:
\begin{align}
&\dphitwo F(\varphi, S(\varphi))[\zeta,\eta] 
= \intO 2 f_{\varphi,\varphi} (\varphi,x)[\zeta,\eta] \cdot u  \dx
+ \intG 2 g_{\varphi,\varphi}(\varphi,x)[\zeta,\eta] \cdot u \dH
\label{Fsec} \\
&- \intO \mathcal{C}''(\varphi)[\zeta,\eta]\mathcal{E}(u):\mathcal{E}(u) \dx
+ 2 \intO \mathcal{C}(\varphi)\mathcal{E}(w_\zeta):\mathcal{E}(w_\eta) \dx \; , \nonumber \\
&\dphitwo E(\varphi)[\zeta,\eta] 
= \intO \varepsilon \nabla\zeta : \nabla\eta + \tfrac{1}{\varepsilon} \nabla^2_{\varphi,\varphi} \Psi(\varphi) [\zeta,\eta]\dx  \; . \label{Esec}
\end{align}

%% file: VMPT.tex
In a sharp interface setting, similar shape and topology optimization problems concerning additive manufacturing have been solved using the topological derivative numerically via a level set approach in \cite{AllaireOverview1,AllaireOverview2,Allaire2}. 
In a scalar phase field setting, pseudo time stepping with a fixed number of time steps has been applied in \cite{WIAS} to obtain a numerical solution.
However, there exists no convergence analysis for this method.
Given that the reduced cost functional $j_\Delta$ is only Fréchet differentiable in the Banach space $\HON\cap\LinfN$, we apply the variable metric projection type (VMPT) method introduced and analyzed in \cite{VMPT}. 
This is an extension of the classical projected gradient method in Hilbert spaces as presented, e.g. in \cite{Bertsekas,GruverSachs1981,KelleySachs1992}.
We briefly review the method and fix the notation used in the following:
\begin{algo}[VMPT method]\label{algorithm1}
\quad 
\begin{algorithmic}[1]
\STATE Choose $0<\tau < 1$, $0<\sigma<1$, $\varphi^{(0)} \in \Pad$ and set $k := 0$
\WHILE{$k \leq k_{\textrm{max}}$}
\STATE Choose$\rho_k> 0 $.
        \STATE Choose a symmetric positive definite bilinear form $a_k$ on
$\HON$.
	\STATE \label{choiceofoverlinevarphi} Calculate the minimizer $\hat \varphi ^{(k)} $
        of the subproblem
\begin{align}
\label{eq:ProjProb}
\min_{y\in\Pad} q_k(y) := \dfrac{1}{2} \Vert y-\phik \Vert^2_{a_k} + \rho_k j_\Delta'(\phik) [ y-\phik ]
\end{align}
	\STATE 
        $v^{(k)} := \hat \varphi ^{(k)} - \phik$
\STATE if  $\| v^{(k)} \|_{\HON} \leq \textrm{tol}$ then return
	\STATE Determine  $\alpha_k:= \tau^{m_k}$ with minimal $m_k\in\N _0$ such that 
\begin{align}
\label{eq:armijo}
j_\Delta(\phik + \alpha_k v^{(k)}) \leq j_\Delta(\phik) + \alpha_k \sigma  j_\Delta'(\phik)[v^{(k)}]  .
\end{align}
	\STATE
        $\varphi^{(k+1)} := \phik + \alpha_k v^{(k)}$; 
	$k:=k+1$
\ENDWHILE
\end{algorithmic}
\end{algo}
Let us mention that the convergence of the algorithm in function space indicates that the iteration numbers of the discretized algorithm are bounded independently of the discretization level.
In the following, we discuss possible choices of $a_k$ and $\rho_k$
and verify the assumptions for convergence given in \cite{VMPT}. We start with the assumptions on the given optimization problem:
The mass constraint is included in the space $\HON $ by considering the shift
$\varphi -\mathfrak{m}$.
We denote by $\Hzero$ the functions in $\HON$ with zero mass.
That the space $\Hzero \cap\LinfN$, as well as the set $\Pad$, fulfil the assumptions on the Banach space and the admissible set is shown in \cite{VMPT}. 
Theorem \ref{Analysis} ensures that also the requirements on $j_\Delta $ are fulfilled.
\newpage
\begin{remark}[Choice of the bilinear forms $a_k$.]
\label{choiceAK}
{\rm
\hfill \\
Choosing for all $k$ the scalar product
\begin{align}
a^1_k(p,y):= \beta_2\varepsilon \intO \nabla p:\nabla y \dx
\end{align}
on the Hilbert space $\Hzero$, the assumptions on the bilinear form listed in \cite{VMPT} are fulfilled.
Here, $a^1_k$ includes the scaling with $\beta_2\varepsilon$, which appears in the second derivative of the Ginzburg-Landau energy.
However, to be able to include more second order information and thereby speed up the error reduction, the bilinear forms $a_k$ are allowed to vary with the iterations:
\begin{align}
\label{eq:SOI1} a^2_k(p,y) :=&
\beta_2\varepsilon \intO \nabla p:\nabla y \dx + 2\intO \mathcal{C}(\phik) \mathcal{E}(w_p) : \mathcal{E}(w_y) \dx \\
= &\beta_2\varepsilon \intO \nabla p:\nabla y \dx
- 2\intO \mathcal{C}' (\phik)[p] \mathcal{E}(u(\phik)) : \mathcal{E}(w_y) \dx
\label{eq:MetricTransformation}
\\
&+ 2\intO (f)_\varphi(\phik,.)[p] \cdot w_y \dx
+ 2\int_{\Gamma_N}(g)_\varphi(\phik,.)[p] \cdot w_y \dx
, \nonumber
\end{align}
where $w_\zeta:=S'(\phik)[\zeta]$ is given by \eqref{eq:LinMS}.
In \eqref{eq:SOI1} one can see that
  $a_k^2$ is symmetric positive definite on $\HON$.
  By testing \eqref{eq:LinMS} for $w_p$ with $w_y$, the calculation of $w_p$ is dispensable. Hence,
 \eqref{eq:MetricTransformation} is used for numerical implementation.
The bilinear forms $a^2_k$ include symmetric positive parts of $j_\Delta''(\phik)$, namely those of $F''(\phik)$ and $E''(\phik)$ (see \eqref{Fsec} and \eqref{Esec}).

Along this line, the inclusion of corresponding parts of $W_\Delta$ also meets the requirements, which can be shown by extending the proof for $F$. 
Therefore, in the numerical experiments, we also test
\begin{align}
\label{eq:SOI2} a^3_k(p,y) :=& a^2_k(p,y) +
2\beta_1 \delta \sum_{j=1}^M \omega(j\delta)
\left\{
\int_{\Omega_{j\delta}} (f^c)_{\varphi} (j\delta,\phik,\cdot)[p] \cdot w^c_{j\delta,y} \dx \right. \\
&\left. - \int_{\Omega_{j\delta}} \mathcal{C}^c_\varphi(j\delta,\phik,\cdot )[p]\mathcal{E}(u_j^c(\phik )):\mathcal{E}(w^c_{j\delta ,y} ) \dx
\right\}, \nonumber
\end{align}
where $w^c_{h,\zeta} :=(S^c(h))'(\phik)[\zeta]$.

Although fewer iterations until termination of the method are expected when $a_k$ includes more second order information, the computational cost of determining the corresponding search direction is increased since the projection type subproblem \eqref{eq:ProjProb} has multiple linearized state equations as constraints in these cases. 
The solver for the subproblem is discussed in Remark \ref{DiscPdas}.
}\end{remark}
\begin{remark}[Choice of the scaling $\rho_k$ of $j'(\varphi^{(k)})$.]
\label{choiceLAM}
{\rm \hfill \\
  Although the parameter $\rho_k$ can be included in $a_k$, we treat it as a separate parameter.
Herewith, we can incorporate the idea of the search along the
projection arc.
Namely, a full step ($\alpha_k=1$) is done, and the scaling $\rho_k$
of the gradient is chosen such that an Armijo-type inequality holds \cite{Bertsekas,KelleySachs1992}.
This would be an expensive algorithm, given that the subproblem \eqref{eq:ProjProb} has to be solved in each Armijo step. 
Furthermore, the existence of such a $\rho_k$
is not guaranteed.
However, to be close to the boundary of $\Pad$ in each iteration, where the solution is expected in general, we apply the following rule suggested in \cite{BFGRS14,VMPT}:
For a given $0<c<1$ set
\begin{align}\label{lambdaRule}
\rho_{k+1} :=
\left\{
\begin{array}{l c l }
\max (c \rho_k, \rho_{\min})&\text { if } &
\alpha_k<1,\\
\min (\frac{1}{c} \rho_k, \rho_{\max})&\text { if } & \alpha_k=1.
\end{array}
\right. 
\end{align}
(As in \cite{VMPT}, we choose $\rho_0=0.005$, $\rho_{\min}=10^{-10}$,
$\rho_{\max}=10^{10}$ and $c=0.75$ in the numerical experiments.)
Herewith, the requirement $0<\rho_{\min}\leq \rho_k\leq \rho_{\max}\ \forall k\in\N_0$ for convergence of the method is fulfilled.
In the numerical experiments, $\rho_k$
increased as long as no clear interface emerged in $\phik$. 
After the formation of the interface, one can observe that $\rho_k$ stabilizes, as observed in Figure 1 in \cite{BFGRS14}.
}\end{remark}

\medskip
Given that with these choices, the assumptions of Theorem 2.2 in \cite{VMPT} are fulfilled, we obtain the following result:
\begin{theorem}
\label{VMPTprop}
Let $\{\phik \}_k \in \HON \cap \LinfN$ be the sequence generated by the VMPT
Algorithm \ref{algorithm1}
applied to the optimization problem
\eqref{eq:OptimProb} while
using $\rho_k$
and $a_k$ as described in Remark \ref{choiceAK} and \ref{choiceLAM}.
Then $j_\Delta(\phik )$ converges with $k\to \infty$ and every accumulation point of $\{\phik \}_k$ is a stationary point.
Furthermore, for the search directions $v^{(k)}$ it holds $j^\prime_\Delta(\phik )[v^{(k)}] \rightarrow 0$ and $\| v^{(k)}\|_{\HON}\rightarrow 0$ for $k \to \infty$.
\end{theorem}
The convergence of the whole sequence $|| v^{(k)}||_{\HON}$ to zero reasons the stopping criteria used in the Algorithm \ref{algorithm1}.
\medskip

%% file: PDAS.tex
\medskip

For the numerical efficiency of the VMPT method, a fast solution of the quadratic subproblem \eqref{eq:ProjProb} is essential.
We employ the primal-dual active set (PDAS) method (see, e.g., \cite{HIK02}).
This method is discussed with respect to the constraints $\Pad$ e.g. in \cite{BGSS13b,BSS12}.
In the literature, it is shown that the method is applicable only to the fully discretized problem.
Hence, next we provide details of the employed spatial discretization.
\begin{remark}[Spatial discretization]\label{Discretization}{\rm  \hfill \\
We consider rectangular domains $\Omega:= \Omega_B \times (0,H)$ and choose a Friedrich-Keller triangulation $\mathcal{T}_{\deltax}$ with fixed mesh sizes in each direction. 
In particular, the mesh size $\deltax$ in direction $e_d$ is chosen such that $\bar K :=\delta / \deltax \in \N$, reminding the definition $\delta=H/M$.
This guarantees that the mesh resolves each layer in the same way.
Furthermore, to resolve the interface appropriately at least 8 mesh points should lie across the interface, i.e. $\delta \leq \varepsilon/4$ is required.
Let us mention that adaptive meshes, as performed in \cite{AdaptiveMeshStructOptim}, influence possible changes in the topology.
Also, uniform meshes avoid mesh updates in each iteration and allow possible reuse of assembled matrices.
The same mesh is used for all involved functions.
All functions are approximated using standard piecewise linear $P1$ finite elements.}
\end{remark}
Since the PDAS method is performed in every VMPT step, we give some details that have a large computational impact and enhance numerical efficiency in the following. 
\begin{remark}[Some details on the solver for the quadratic subproblem]
  \label{DiscPdas}{\rm  \hfill \\
The PDAS method can be viewed as a semismooth Newton method applied to the discretized first-order condition formulated as $G(x)=0$ \cite{HIK02}.
We globalize this method by applying the classical damped Newton method to $\min \frac{1}{2}\|G(x)\|^2$. That means, the Newton step provides the descent direction $d$  and the step-length $t$ is determined with backtracking using the factor $0.75$ in the Armijo condition,
which reads in this case as $\| G(x+td)\|^2 \leq (1 - t/4 ) \| G(x) \|^2 $.
As an initial guess for the PDAS algorithm in the $k$-th VMPT iteration, we choose the solution (including the Lagrange multipliers) of the PDAS algorithm of the $(k-1)$-th VMPT iteration. 
This approach significantly reduced the necessary mean PDAS iteration.

While the remaining main steps of the PDAS algorithm can be found in the given references, we now like to give details on how the arising linear equations are solved.
Let $\bar \varphi$ denote the vector consisting of the coefficients of the numerical finite element solution for $\varphi$. For the other functions, the notation is used respectively. Then the discretized elasticity equations
  \eqref{eq:WeakFormLinElast} and
  \eqref{eq:WeakFormAMCS}
  for a given $\phik$ result into linear systems 
\begin{align}
  \Bk \bar u^{(k)} = \bar r^{(k)} , \qquad
  \Bcjk \bar u^{c,j,(k)} = \bar r^{c,j,(k)}  \text{ for } j=1,\ldots , M ,
\end{align}
with symmetric positive definite matrices $\Bk$ and
$\Bcjk$.
The corresponding discretized linearized state equations \eqref{eq:LinMS} together with the corresponding version for \eqref{eq:WeakFormAMCS} take the form
\begin{align}
  \Bk \bar w_\zeta^{(k)} = \Fk \bar \zeta ,  \qquad
  \Bcjk \bar w^{c,j,(k)} = \Fcjk  \bar \zeta  \text{ for } j=1,\ldots , M .
\end{align}
The cost functional $q$ of the subproblem \eqref{eq:ProjProb} involves a bilinear form $a_k$.
The discretization of the bilinear form, e.g. $a^3_k$ given in \eqref{eq:SOI2}, yields
$  a^3_k(p_{disc},y_{disc})=
\bar p^T A_3^{(k)} \bar y $, where
\begin{align}\label{eq:A4}
\tfrac12 A_3^{(k)}=
\tfrac{\beta_2\varepsilon}{2} L +
  (\Fk )^T (\Bk )^{-1} \Fk
  + \beta_1 \delta \sum_{j=1}^M \omega (j\delta) (\Fcjk )^T  (\Bcjk )^{-1} \Fcjk 
\end{align}
with the stiffness matrix $L$ corresponding to
$\int_\Omega \nabla p : \nabla y \dx $.
In each iteration of the PDAS algorithm, we use MINRES to solve the arising linear saddle point problem, where the 1-1-block consists of the part of $A_3 ^{(k)}$ which corresponds to the current inactive set.
Hence, in each MINRES iteration of one PDAS iteration these parts of $A_3 ^{(k)}$ have to be applied to a vector.
Therefore, we assemble the matrix $L$ once, and in the $k-$th VMPT step  also the matrices $\Bk,\Bcjk, \Fk, \Fcjk $ are assembled and for $\Bk$ and $ \Bcjk$ sparse LU-decomposition is performed.
These are then used in the outer loop and in all inner loops of the $k-$th VMPT step. 
Furthermore, to accelerate the computation, $multiprocessing$ is employed.
}\end{remark}

%% file: Nesting.tex
\begin{remark}[Nested approach for the VMPT method]
\label{sec:nesting} {\rm \hfill \\
When we do not study the dependence of the iteration numbers on the mesh resolution or the number of layers, we employ a nested approach by successively increasing the number of layers while simultaneously decreasing the mesh width until the given number $M$ of layers is reached and $\delta_H$ is small enough. 
We proceed as follows.
First, the optimization problem is solved for a given minimal number $M_0$ of layers and a given number
$\hat K$ determining the mesh size
$\delta_H:= \tfrac{H}{M_0 \hat K}$ in direction $e_d$.
  Here $\hat K$ has to be such large that $\delta_H \leq \varepsilon$
  (we use $M_0=4$ and $\hat K=10$).
This solution $\varphi^{(k)}$
is used as an initial guess for the refined problem, where the number of layers is increased by a fixed factor, e.g., by $2$, and the mesh size is halved, such that it still holds $\hat K \delta_H=\tfrac{H}{M_i}$. 
This process is iterated as long as $M_i\leq M$.
When the required number $M$ of layers is reached, the mesh size is further halved successively until the interface is resolved appropriately, i.e., until at least $\delta_H\leq \varepsilon /4 $.
While the scaling parameter
  $\rho_k$
  is kept for the refined problem,
  the current displacement fields are not utilized
  given -amongst other reasons- that
  the number of displacement fields grow
  with increasing the number of layers.
  However,
  the first PDAS iteration for the quadratic subproblem
  \eqref{eq:ProjProb} of the refined problem
  is initialized with the last PDAS iterate
  of the current problem $\hat \varphi^{(k)}$ and
  its corresponding Lagrange multipliers.
}
\end{remark}

%% file: Numerics.tex
In this section, numerical results are discussed. The VMPT method and all experiments are implemented in $Python$ using libraries such as $scipy$, $FEniCS$ \cite{Fenics}, and $multiprocessing$. Every experiment is performed on an office computer with an \\ $Intel(R)\ Core(TM)\ i7-14700K\ 3.40 GHz$ processor and $64 GB$ of RAM.
In the first subsection, the performance of the VMPT method on the problem \eqref{eq:OptimProb} is discussed and analyzed.
Afterwards, the influence of the model parameters, as well as of the general setting on the obtained solutions, is studied. Furthermore, examples with three phases and in 3D are presented.

\begin{testexample} \label{testEx} 
\hfill   \\
{\rm
  As a test example, we employ the cantilever beam with the following specifications, which are used throughout the numerical experiments if not mentioned otherwise.
A scalar phase field $\tilde{\varphi}$ is used to model the distribution of material and void in $\Omega= [0,3] \times [0,1]$, with the fixed mass constant $\tilde{\mathfrak{m}}=-0.25$, which corresponds to $37.5\%$ material and $62.5\%$ void. 
In the mechanical system, only the outer force $\tilde g(\tilde{\varphi},x) = -e_d$ is applied on $\Gamma_N = [2.75,3] \times \{0\}$, while the structure is fixed on $\Gamma_D=\{0\} \times [0,1]$. 
The building plate is located on the bottom side, i.e. $\Gamma_B = [0,3]\times\{0\}$. The gravitational force during construction is chosen as $\tilde{ f}^c(\tilde{\varphi},h,x)= -0.5(1+\tilde{\varphi})e_d$. 
All forces chosen in this example are consistent with the choices in  \cite{AllaireBogoselSupports1,AllaireOverview2}. 
If not mentioned otherwise, the Lamé parameters are set to $\lambda_1=\mu_1=44$ for the material, and $\lambda_2=\mu_2=\varepsilon^2$ for the void, following \cite{AllaireBogoselSupports1,AllaireOverview2,VMPT}.
To account for reduced stiffness of the material during the construction process, the Lamé parameters are lowered to $\mu_1^c=\lambda_1^c = 32$, while $\mu_2^c=\lambda_2^c = \varepsilon^2$ remain unchanged.
Further, we set the parameters in the cost functional to $\beta_1 = 48$, $\beta_2 = 0.02$ and choose $M=10$. The weighting $\omega$ is chosen as $\omega(h) = \frac{1}{h}= \frac{3}{|\Omega_h|}$ as suggested in Section \ref{sec:AnalysisOCP}. 
The parameter corresponding to the interface width is set to $\varepsilon= 0.025$. 
The termination criterion is given by $\sqrt{\varepsilon\beta_2}|v^{(k)}|_{H^1}\leq 10^{-3}$ throughout the numerical results. 
As mentioned before, to properly resolve the interface, the mesh width should not exceed $\frac{\varepsilon}{4}$. 
To show the full potential of the VMPT method, we drive the maximal mesh width down to $\frac{\varepsilon}{8}$, i.e. to a uniform mesh with $308~481$ nodes.
By DOFs, we refer to the number of unknowns that are involved in the optimization problem, i.e. the number of unknowns for $\tilde{\varphi}$, $u$ and all $u^c_k$. For the introduced example, an optimization problem with $4~010~253$ DOFs has to be solved. 
The initial data for the solver
is chosen such that no pure phases are present and the mass constraint is roughly fulfilled. To be more precise, we set
$\tilde{\varphi}_0=\tilde{\mathfrak{m}} /|\Omega |$
with an additional normally distributed disturbance with standard deviation $0.05$.
Using this setting, the topology changes several times during the optimization process.
}
\end{testexample}

\subsection{Numerical analysis of the optimization solver}
\label{sec:SolverSpec}

\subsubsection{\texorpdfstring{Iteration numbers for increasing mesh sizes and for increasing the number of layers $M$}{Iteration numbers for increasing mesh sizes and number of layers M}}
Since we apply the VMPT method in function space, the iteration count is expected to be independent of the spatial discretization level and of the number of layers $M$.
Here, the layer number $M$ can be seen as the degree of approximation of $W$ by $W_{\Delta}$.
Numerically, this  is examined by solving the test problem using the $a_k^1$ inner product without nesting.
In Table \ref{tab:MeshDependence}, the results for decreasing mesh widths are given, while in Table \ref{tab:MDependence}, the results for increasing $M$ are presented.
When varying $M$, we set $\varepsilon=0.04$ to allow for a coarser mesh, while keeping interface resolutions sufficient. This ensures that even for $M=100$, multiprocessing on the given office computer remains feasible.

\begin{table}[th!]
\centering
\begin{tabular}{|c|c|c|c|c|}
\hline
\textbf{Spacial Nodes} & 4961 & 19521 & 77441 & 308481 \\
\hline
\hline
\textbf{VMPT iterations} & 2477 & 2589 & 2490 & 2561 \\
\hline
\textbf{Mean PDAS iterations} & 3.13 & 3.49 & 3.65 & 4.02 \\
\hline
\end{tabular}
\caption{Test example computed unnested on different spatial grids}
\label{tab:MeshDependence}
\end{table}

\begin{table}[th!]
\centering
\begin{tabular}{|c|c|c|c|c|c|c|}
\hline
\textbf{M} & 10 & 20 & 40 & 50 & 100 \\
\hline
\hline
\textbf{DOFs} & 1.570M & 2.778M & 5.194M & 6.402M & 12.442M \\
\hline
\textbf{VMPT iterations} & 1198 & 736 & 712 & 698 & 741 \\
\hline
\textbf{Mean PDAS iterations} & 3.93 & 4.34 & 4.40 & 4.26 & 4.17 \\
\hline
\end{tabular}
\caption{Test example computed unnested with different numbers of layers}
\label{tab:MDependence}
\end{table}

When increasing the number of mesh points or the number of layers, the necessary VMPT iterations stay roughly the same, which confirms the expected behaviour.
It is worth noting that due to the good initial values for the PDAS subroutine, as described in Section \ref{DiscPdas}, that the average number of PDAS steps to compute a search direction stays roughly bounded. Especially in later VMPT iterations, no more than 3 PDAS steps are taken.

\subsubsection{Benefits of the nested ansatz and of variable metrics}
We conclude this subsection by demonstrating the potential benefits of nesting and variable metrics on the performance of the VMPT method. To this end, we compare the nested and unnested approaches with $a_k^1$ as well as the nested VMPT method with inner products
$a_k^2$ and $a_k^3$ 
containing second order information given in Section \ref{sec:VMPT}. The results are displayed in Table \ref{tab:SolverTest} and Figure \ref{fig:conv-j}.

\begin{table}[ht]
\centering
\begin{tabular}{|r |r |r |r |r |r |r |r |r |r |}
\hline \multicolumn{3}{|c|}{}& \multicolumn{1}{|c|}{unnested}& \multicolumn{3}{|c|}{nested} \\
\hline
\multicolumn{1}{|c|}{$M$}&\multicolumn{1}{|c|}{$\delta_H$} &  \multicolumn{1}{|c|}{DOFs}& \multicolumn{1}{|c|}{ $a^1_k$} & \multicolumn{1}{|c|}{ $a^1_k$} & \multicolumn{1}{|c|}{ $a^2_k$} & \multicolumn{1}{|c|}{ $a^3_k$} \\
\hline
\hline
$4$& ${1}/{40}$& 34~727  & - & 2242 & 1539 & 1433 \\
\hline
$8$& ${1}/{80}$& 214~731  & - & 90 & 68 & 88  \\
\hline
$10$& ${1}/{100}$& 395~213 & - & 80 & 118 & 76  \\
\hline
$10$& ${1}/{200}$& 1~570~413 & - & 12 & 28 & 12  \\
\hline
$10$& ${1}/{320}$& 4~010~253 & 2561 & 2 & 3 & 2  \\
\hline
\hline
\multicolumn{3}{|c|}{\textbf{VMPT iterations}} & 2561 & 2426 & 1748 & 1612   \\
\hline
\multicolumn{3}{|c|}{\textbf{Mean PDAS steps}} & 4.02 & 3.81 & 3.29 & 3.20   \\
\hline
\multicolumn{3}{|c|}{\textbf{CPU time}} & 36h 34min & 35min
& 1h 48min & 6h 7min   \\
\hline
\hline
\multicolumn{3}{|c|}{\textbf{$\tilde{E}$}} & 12.363875 & 13.238477 & 13.230681 & 13.234126  \\
\hline
\multicolumn{3}{|c|}{\textbf{$F$}} & 0.153840 & 0.148654 & 0.148824 & 0.148681 \\
\hline
\multicolumn{3}{|c|}{\textbf{$W$}} & 0.001505 & 0.001377 & 0.001375 & 0.001377   \\
\hline  
\end{tabular}
\caption{Test of the (nested) VMPT method with different inner products
  \label{tab:SolverTest}}
\end{table}

\begin{center}
\begin{figure}[th!]
\centering
  \begin{subfigure}{0.4\textwidth}
    \centering \includegraphics[width =\linewidth]{
      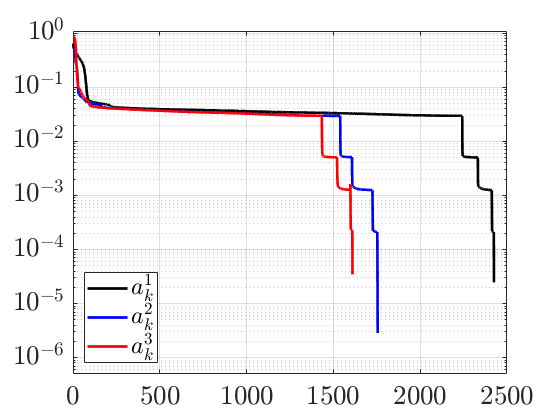}
    \caption{Error $|j(\varphi_k)-j(\varphi^\star)| $ for different $a_k$}
    \label{fig:conv-jmetrics}
  \end{subfigure}
  \qquad
  \centering
  \begin{subfigure}{0.4\textwidth}
    \centering \includegraphics[width =\linewidth]{
      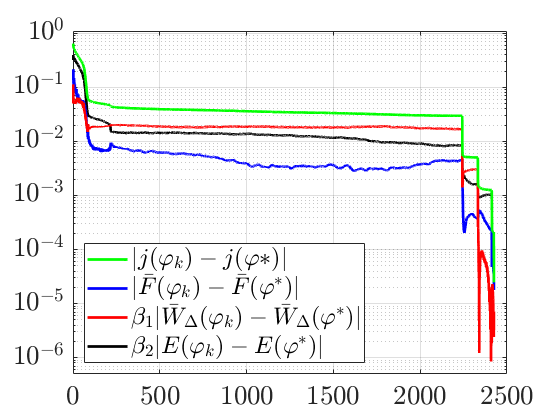}
    \caption{Errors for $j$, $E$, $\bar F$ and $\bar W$ using $a_k^1$}
    \label{fig:conv-ja1}
  \end{subfigure}
\caption{Error reductions in the scalar case}
\label{fig:conv-j}
\end{figure}
\end{center}

Nesting yields a clear speed-up in computation time compared to the unnested approach. 
Although roughly the same number of VMPT iterations
are necessary, in the nested algorithm most of these are performed on the coarsest grid.
The introduction of variable metrics shows a clear advantage in terms of necessary iteration numbers which nearly halved.
The average of 3 to 4 PDAS iterations remains similar in all cases. 
In contrast to the use of $a_k^1$, a quadratic problem with several linear PDEs as constraints has to be solved when 
a variable metric is employed.
For example using $a_k^3$, the arising QPs have
$M+1$ linearized state equations as constraints.
Hence, up to now no speed-up in CPU time is achieved. More sophisticated (PDE-) solvers may help to overcome this issue.

For all further numerical experiments, we employ the nested VMPT 
method with $a_k^1$ inner product.
In Figure \ref{fig:conv-j},
  the convergence behaviours for the reduced cost functional $j$ as well as
for $F$, $W$ and $E$  are plotted.
As reference solution the final iteration, denoted by $\varphi^\star$, is employed.
The progression of the error in $j$ behaves as is known for gradient methods.
Large decreases occur in the first steps of each nesting iteration.
While the error in $j$ is decreasing in accordance to the method monotonically, this does not hold for the summands $F$, $W$ and $E$ as they compete with each other.

\subsection{\texorpdfstring{Impact of the choice of $\beta_{1}$, $\beta_{2}$, $W$, $C^c$ and $\Gamma_B$ on the solutions}{Impact of the choice of beta1, beta2, Fc, Cc and GammaB on the solutions}}
\label{sec:ParameterStudy}

\subsubsection{Impact of the parameters \texorpdfstring{$\beta_{1}$}{beta1}, \texorpdfstring{$\beta_{2}$}{beta2} }
The impact of reducing the weighting $\beta_2$ of the perimeter regularization $E$ is shown in Table \ref{tab:Beta2ParameterStudy}, where $\beta_2$ is reduced from $0.08$ to $0.0002$ while keeping  $\beta_1 = 48$ fixed. As expected, both mean compliances $F$ and $W$ profit from the increase in perimeter $E$.
Figure \ref{fig:Beta2ParameterStudy} shows the emergence of internal filigree structures that increase resilience to the applied force $g$ and stabilize upper regions during construction.

\begin{table}[th!]
\centering
\begin{tabular}{|c|c|c|c|c|c|c|}
\hline
\textbf{$\beta_2$} & 0.08 & 0.02 & 0.01 & 0.002 & 0.0002 \\
\hline
\hline
\textbf{$\tilde{E}$} & 10.065890 & 13.238477 & 15.872777 & 26.981015 & 77.635747 \\
\hline
\textbf{$F$} & 0.185044 &0.148654 & 0.146341 & 0.133387 & 0.128814 \\
\hline
\textbf{$W$} & 0.001488 &0.001377 & 0.001304 & 0.001194 & 0.001172 \\
\hline
\end{tabular}
\caption{Parameter study for $\beta_2$}
\label{tab:Beta2ParameterStudy}
\end{table}
\vspace{-1cm}
\begin{center}
\begin{figure}[th!]
\centering
  \begin{subfigure}{0.32\textwidth}
    \centering \includegraphics[width=\linewidth,trim={0cm 6cm 0cm 5cm},clip]{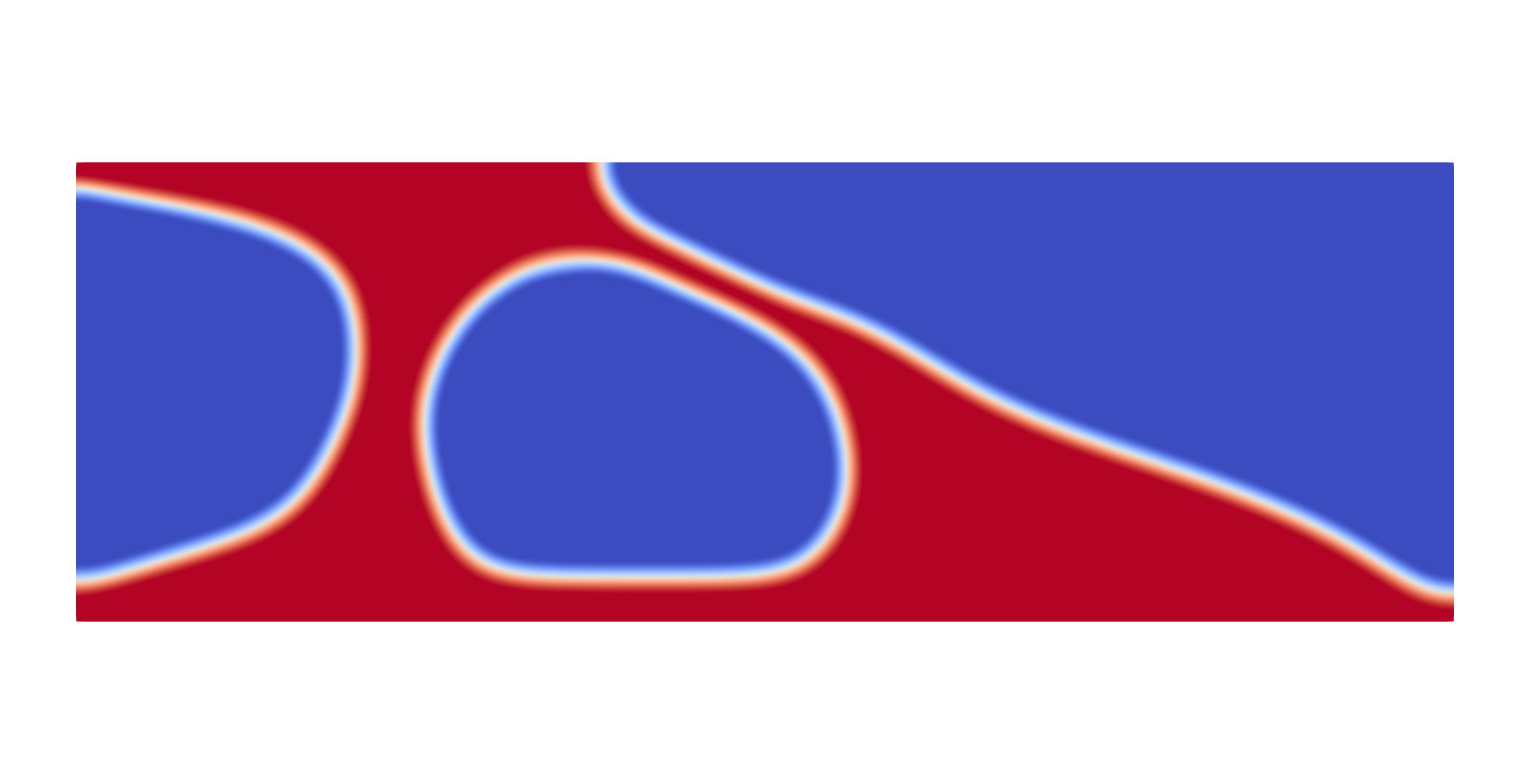}
    \caption{$\beta_2=0.08$}
    \label{fig:Beta2High}
  \end{subfigure}
  \begin{subfigure}{0.32\textwidth}
    \centering \includegraphics[width=\linewidth,trim={0cm 6cm 0cm 5cm},clip]{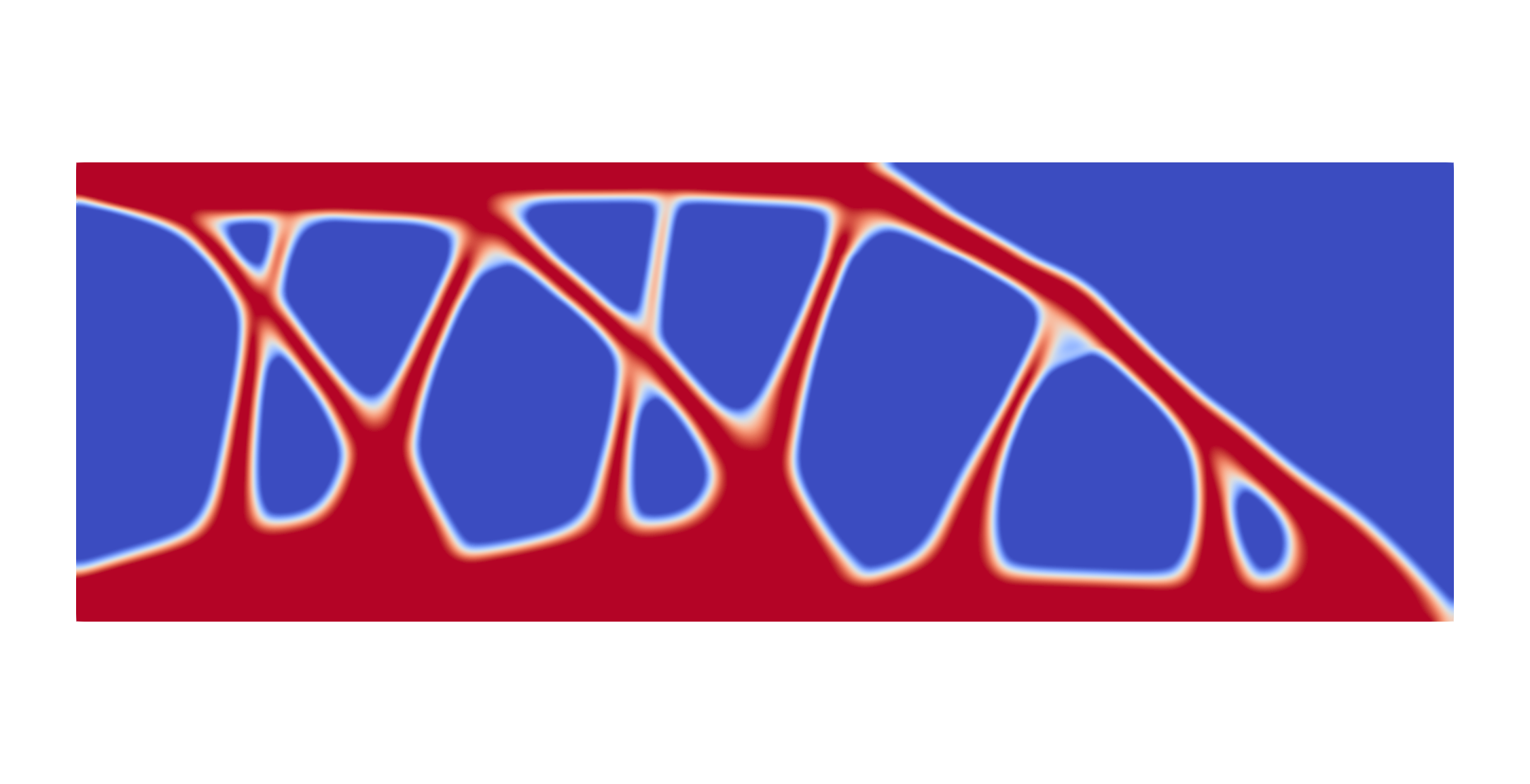}
    \caption{$\beta_2=0.002$}
    \label{fig:Beta2Low}
  \end{subfigure}
\caption{Test example with low and high perimeter}
\label{fig:Beta2ParameterStudy}
\end{figure}
\end{center}
In Table \ref{tab:Beta1ParameterStudy} and Figure \ref{fig:Beta1ParameterStudy},
the results are listed for increasing $\beta_{1}$,
  which is the parameter for  the penalization of the
  deformations during construction.
The parameter $\beta_2$ is fixed to $0.02$.
\begin{table}[th!]
\centering
\begin{tabular}{|c|c|c|c|c|c|c|}
\hline
\textbf{$\beta_1$} & 0                   & 48 & 96 & 384  \\
\hline
\hline
\textbf{$\tilde{E}$} & 11.728154 &13.238477 & 14.179275 & 16.229090  \\
\hline
\textbf{$F$} & 0.127871 &0.148654 & 0.163775 & 0.221231  \\
\hline
\textbf{$W$} & 0.011106 & 0.001377 & 0.001171 & 0.000871  \\
\hline
\end{tabular}
\caption{Parameter study for $\beta_1$}	
\label{tab:Beta1ParameterStudy}
\end{table}
\vspace{-1cm}
\begin{center}
\begin{figure}[th!]
\centering
  \begin{subfigure}{0.32\textwidth}
    \centering \includegraphics[width=\linewidth,trim={0cm 6cm 0cm 5cm},clip]{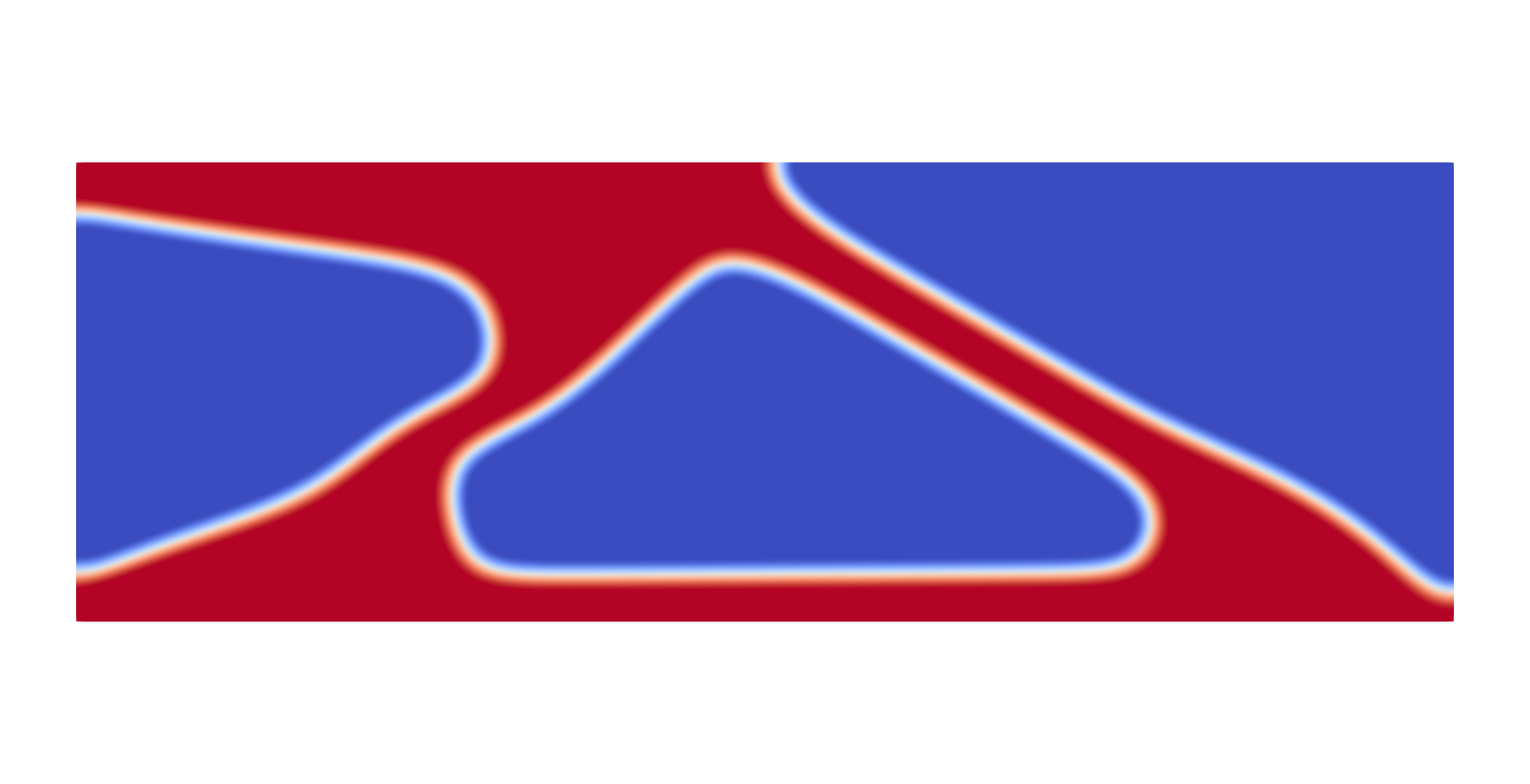}
    \caption{$\beta_1=0$}
    \label{fig:Beta1Zero}
  \end{subfigure}
  \begin{subfigure}{0.32\textwidth}
    \centering \includegraphics[width=\linewidth,trim={0cm 6cm 0cm 5cm},clip]{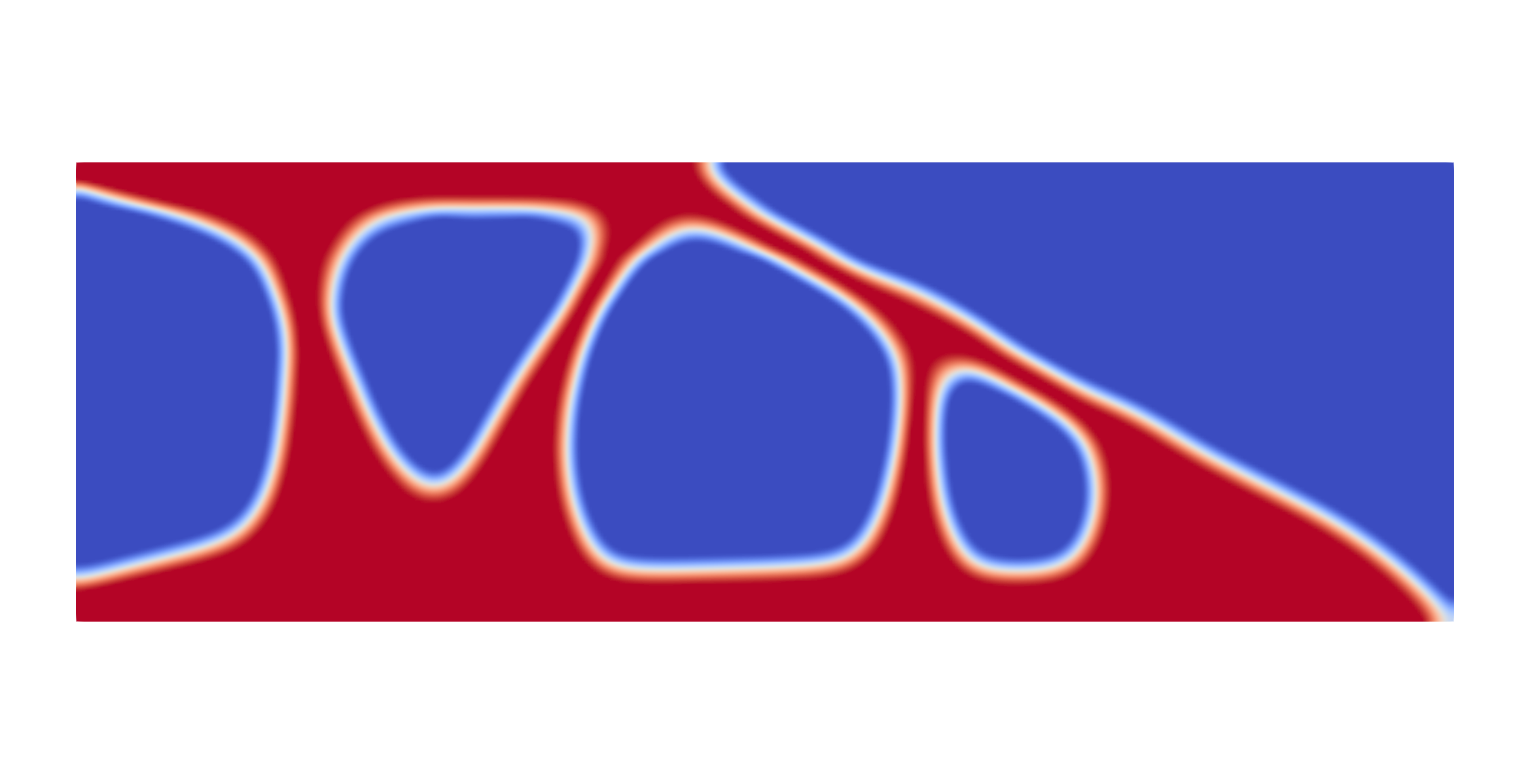}
    \caption{$\beta_1=96$}
    \label{fig:Beta1Low}
  \end{subfigure}
  \begin{subfigure}{0.32\textwidth}
    \centering \includegraphics[width=\linewidth,trim={0cm 6cm 0cm 5cm},clip]{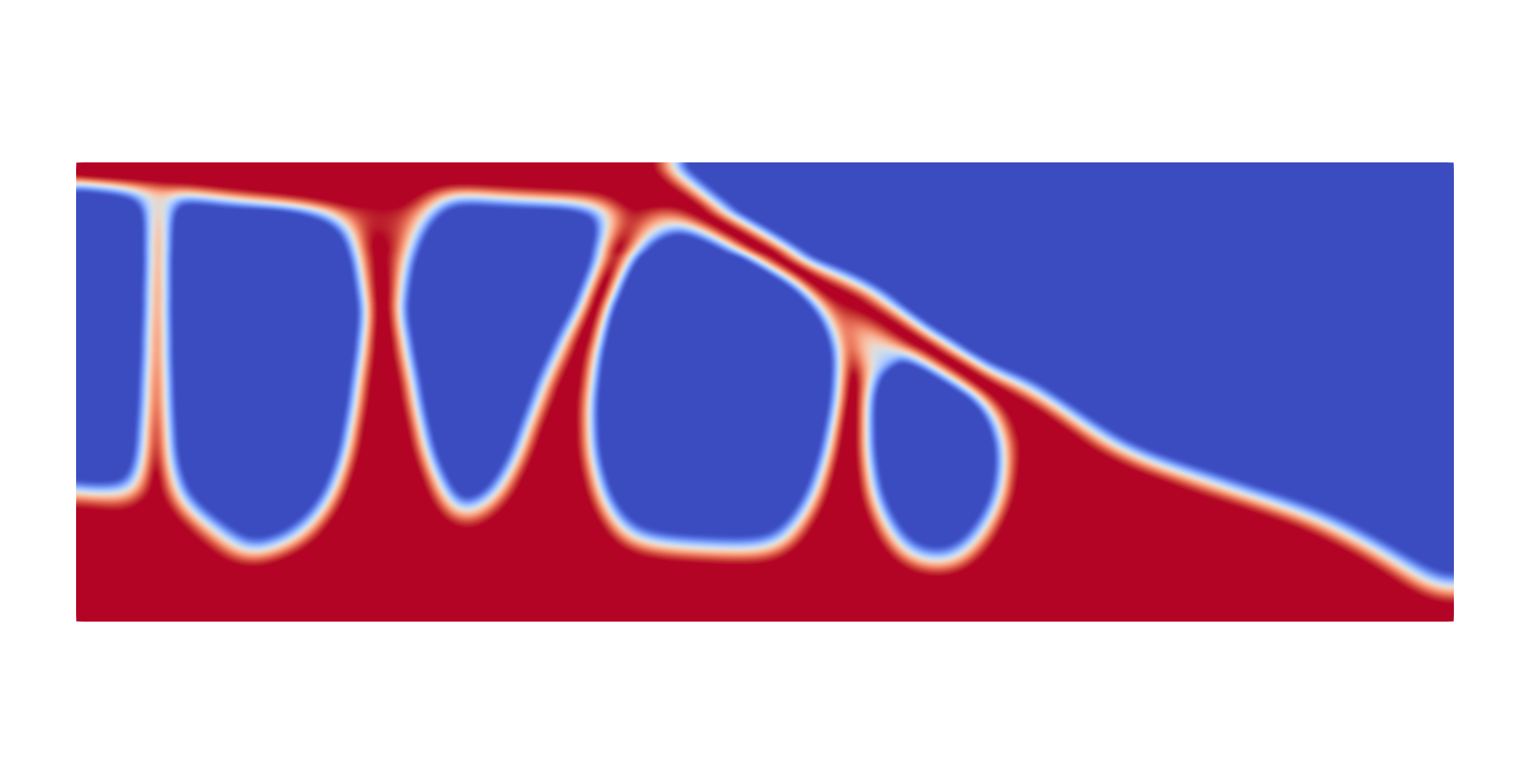}
    \caption{$\beta_1=384$}
    \label{fig:Beta1Med}
  \end{subfigure}
\caption{Test example with increasing $\beta_1$}
\label{fig:Beta1ParameterStudy}
\end{figure}
\end{center}
As expected, a large reduction in $W$ is achieved when $\beta_1$ is increased. In contrast to this, $E$ increases.
This is evident in the formation of multiple vertical bars
aiming to
 minimize the deformations during the construction phase,
 e.g. in the top left corner. Furthermore, high located trusses thin out, while mass gathers at the building plate to
 reduce deformations due to gravity. This leads to an increase in $F$.

When the construction process is neglected,
i.e. $\beta_1=0$, the solution features large overhangs
marked in Figure \ref{fig:Overhang0}
and with them large deformations 
visible  in  Figure \ref{fig:0_5_disp}- \ref{fig:0_10_disp},
which are reflected in a high value of $W$. 
\begin{center}
\begin{figure}[th!]
\centering
  \begin{subfigure}{0.32\textwidth}
    \centering \includegraphics[width=\linewidth,trim={0cm 6cm 0cm 5cm},clip]{
      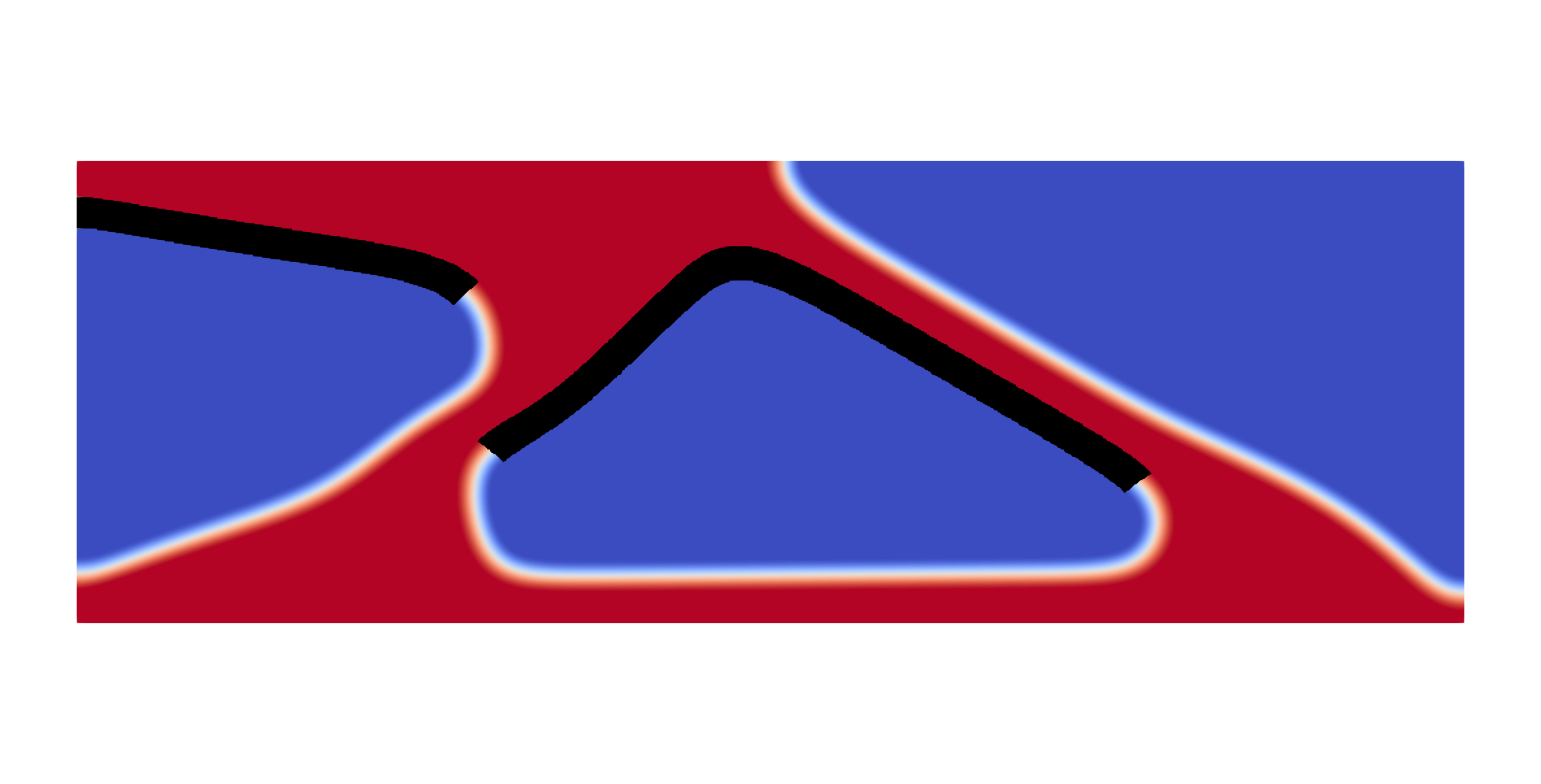}
    \caption{$\beta_1=0$}
    \label{fig:Overhang0}
  \end{subfigure}
  \qquad
  \begin{subfigure}{0.32\textwidth}
    \centering \includegraphics[width=\linewidth,trim={0cm 6cm 0cm 5cm},clip]{
      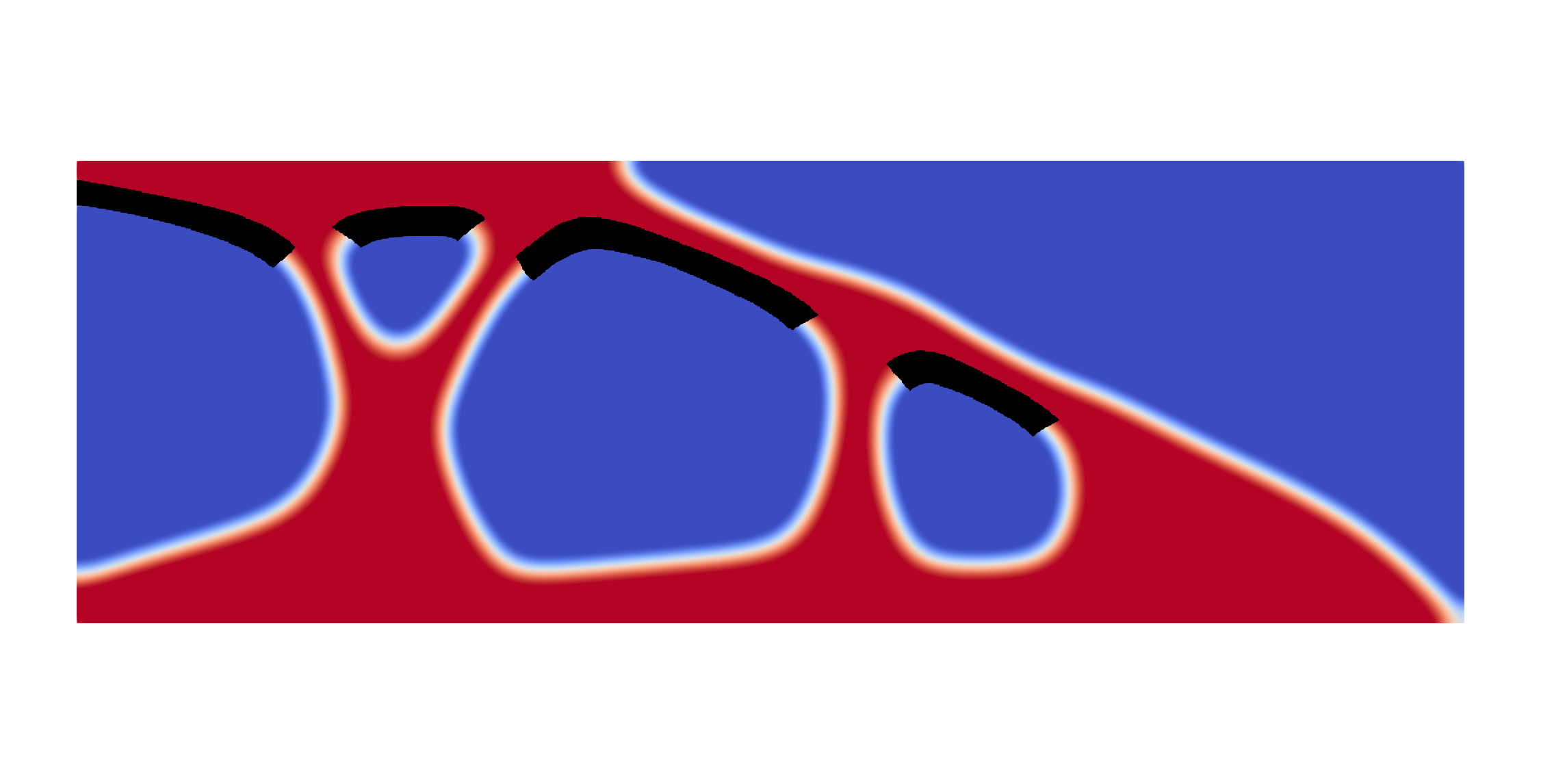}
    \caption{$\beta_1=48$}
    \label{fig:}
  \end{subfigure}
\caption{Overhangs violating the 45\textdegree angle condition marked black}
\label{fig:overhangV}
\end{figure}
\end{center}
\vspace{-1cm}
\begin{center}
\begin{figure}[th!]
\centering
  \begin{subfigure}{0.24\textwidth}
    \centering \includegraphics[width=\linewidth,clip]{
      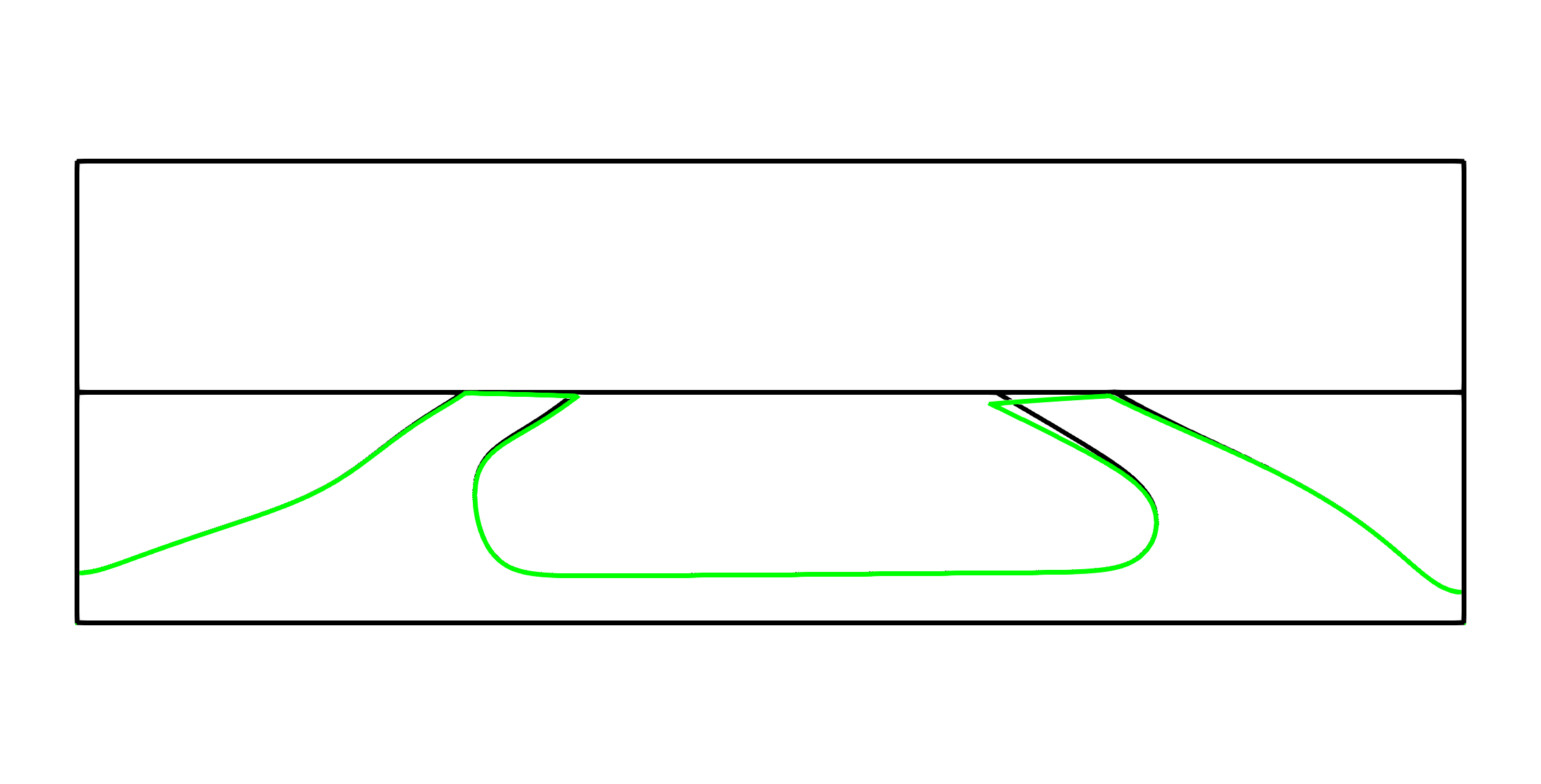}
    \caption{$\beta_1=0$, $k=5$}
    \label{fig:0_5_disp}
  \end{subfigure}
\centering
  \begin{subfigure}{0.24\textwidth}
    \centering \includegraphics[width=\linewidth,clip]{
      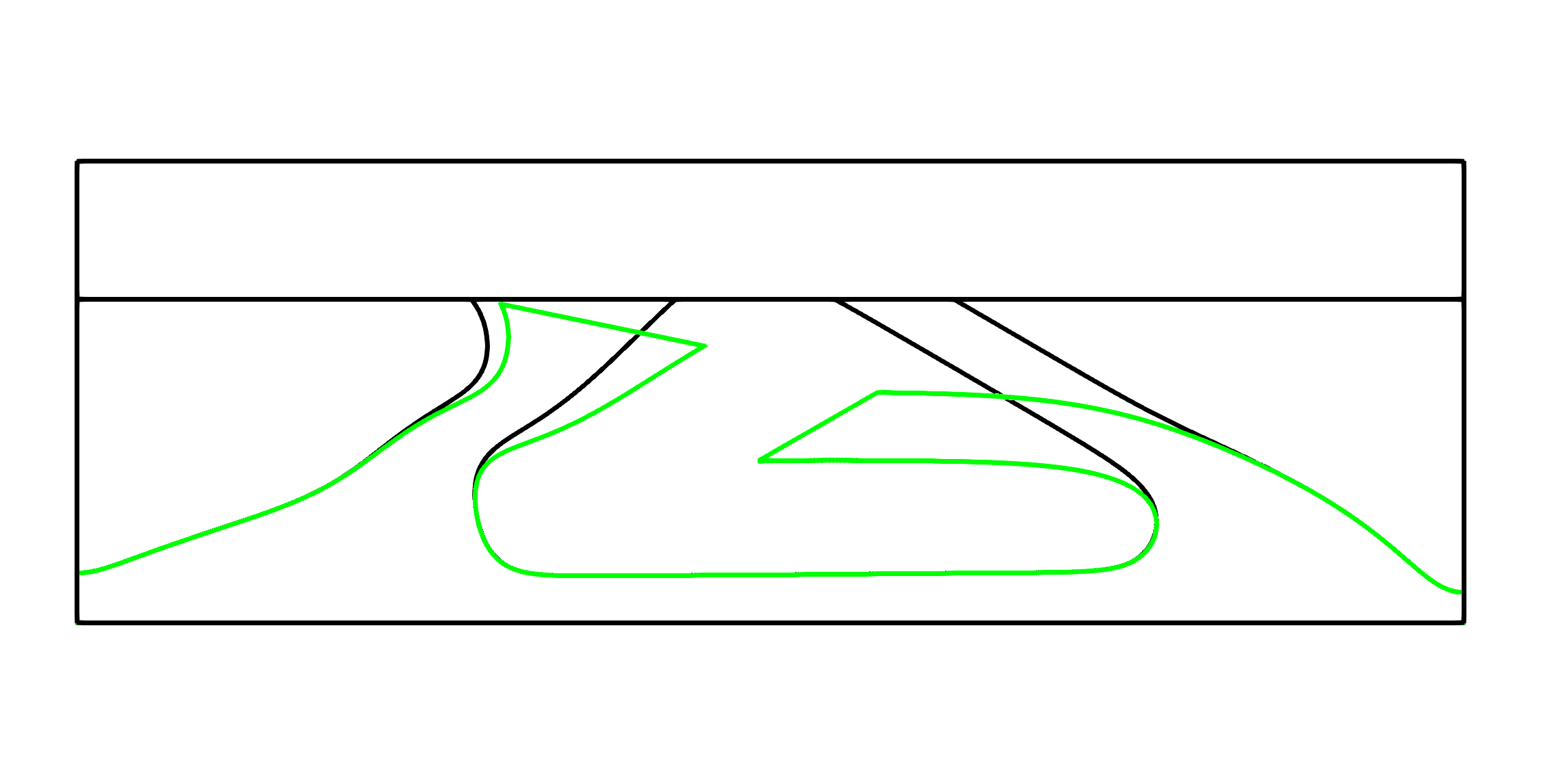}
    \caption{$\beta_1=0$, $k=7$}
    \label{fig:0_7_disp}
  \end{subfigure}
\centering
  \begin{subfigure}{0.24\textwidth}
    \centering \includegraphics[width=\linewidth,clip]{
      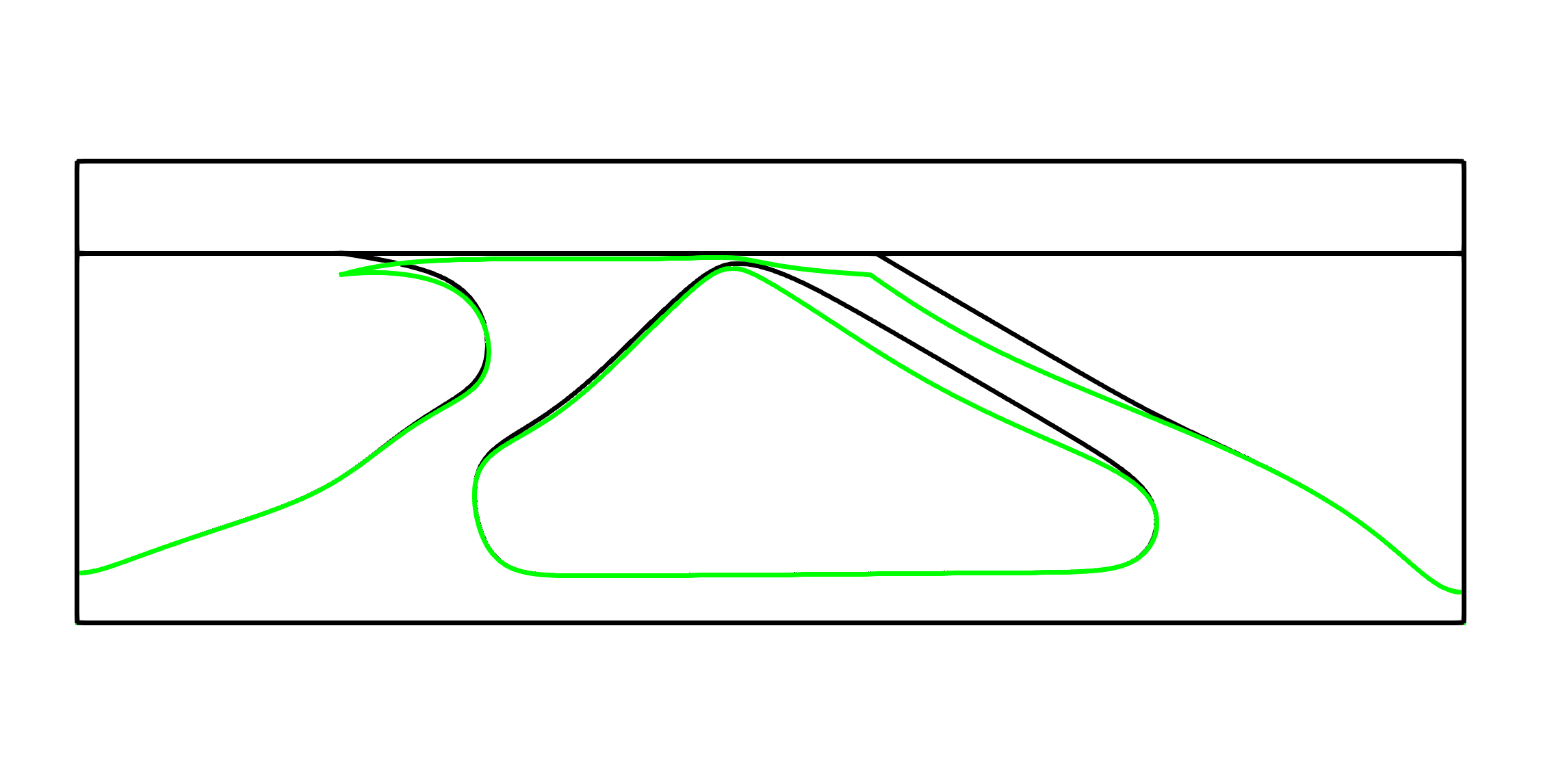}
    \caption{$\beta_1=0$, $k=8$}
    \label{fig:0_8_disp}
  \end{subfigure}
\centering
\begin{subfigure}{0.24\textwidth}
  \centering
  \includegraphics[width=\linewidth,clip]{
    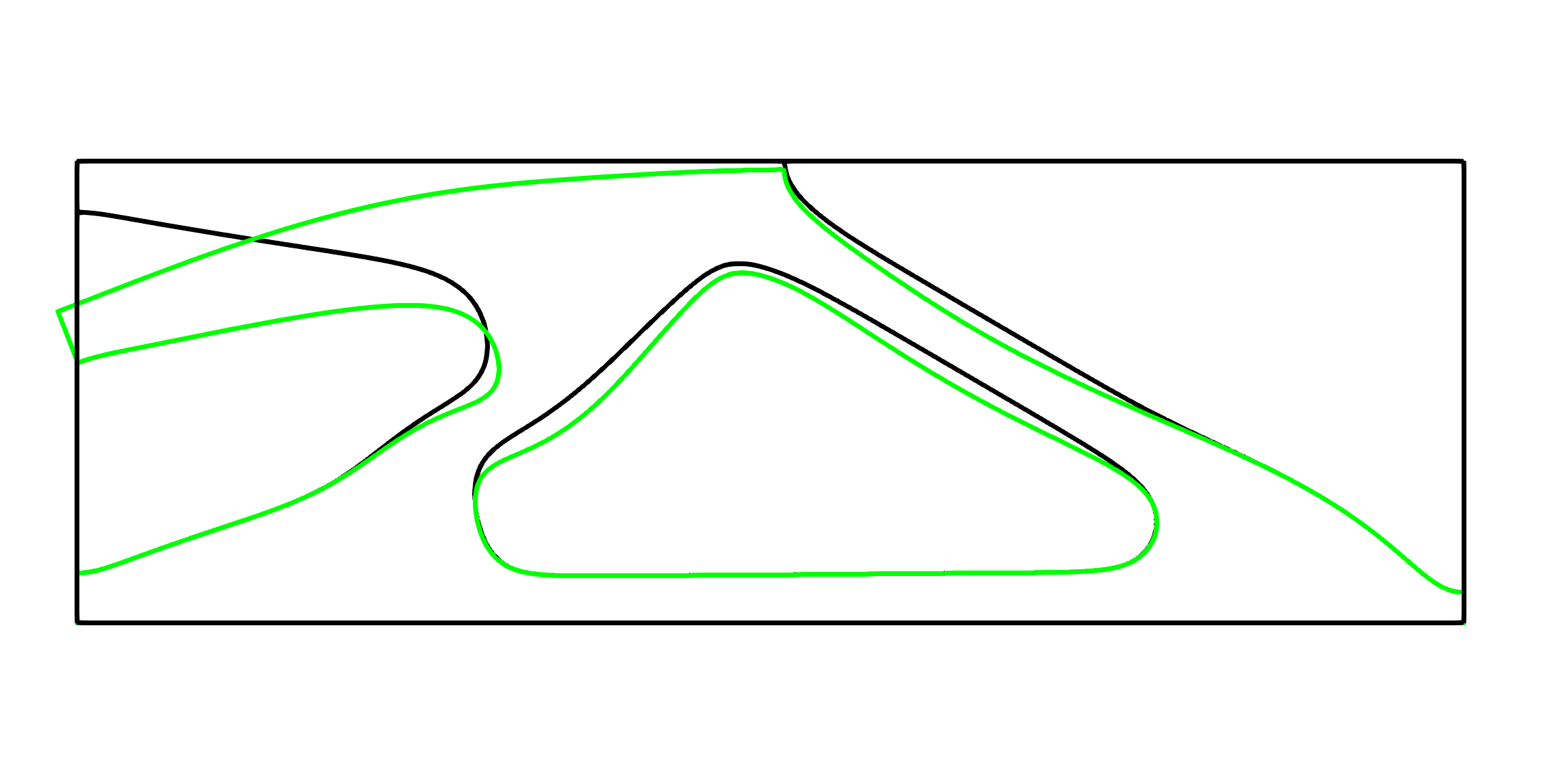}
    \caption{$\beta_1=0$, $k=10$}
    \label{fig:0_10_disp}
  \end{subfigure}
  
\centering
  \begin{subfigure}{0.24\textwidth}
    \centering \includegraphics[width=\linewidth,clip]{
      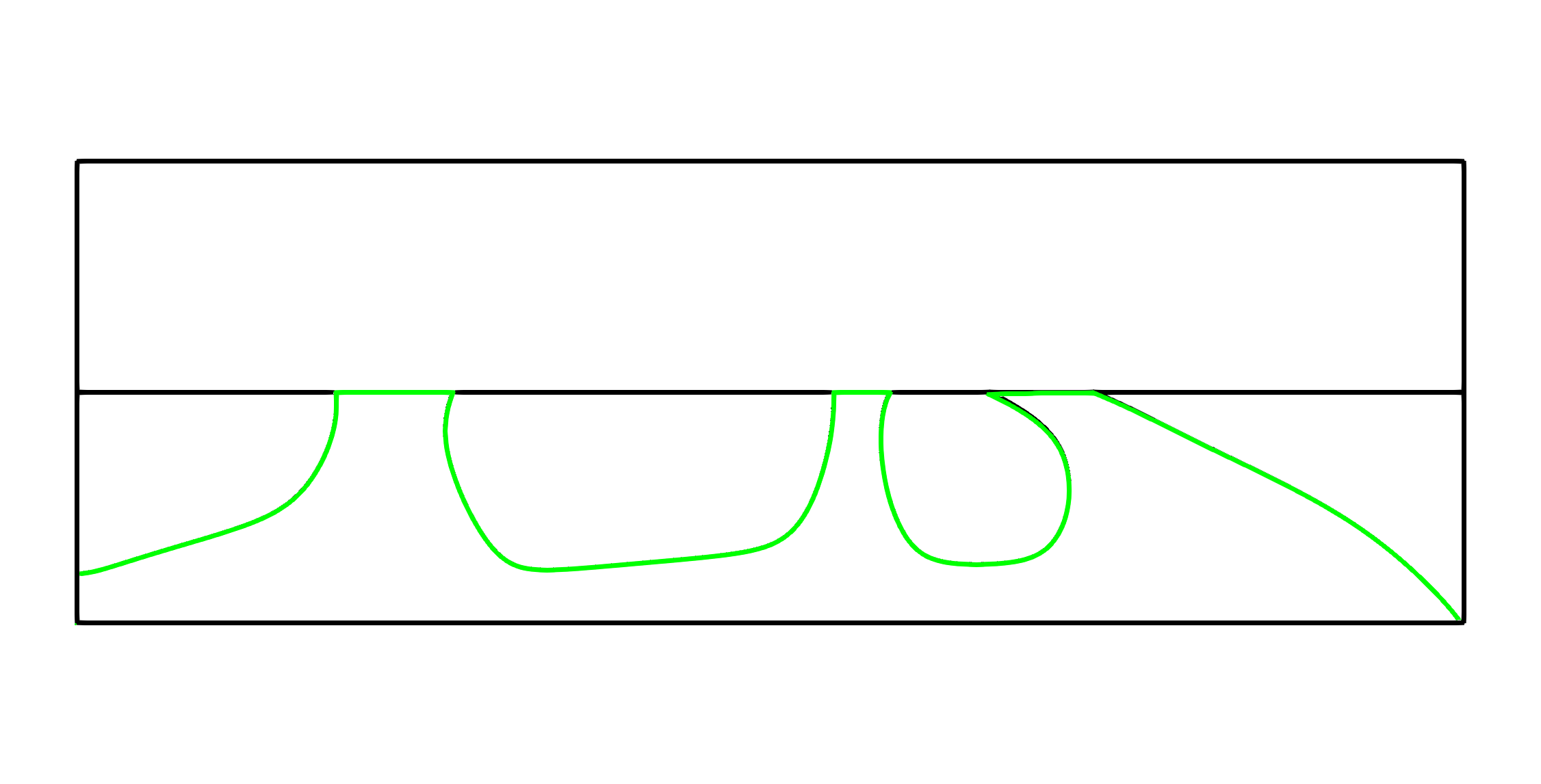}
    \caption{$\beta_1=48$, $k=5$}
    \label{fig:48_5_disp}
  \end{subfigure}
\centering
  \begin{subfigure}{0.24\textwidth}
    \centering \includegraphics[width=\linewidth,clip]{
      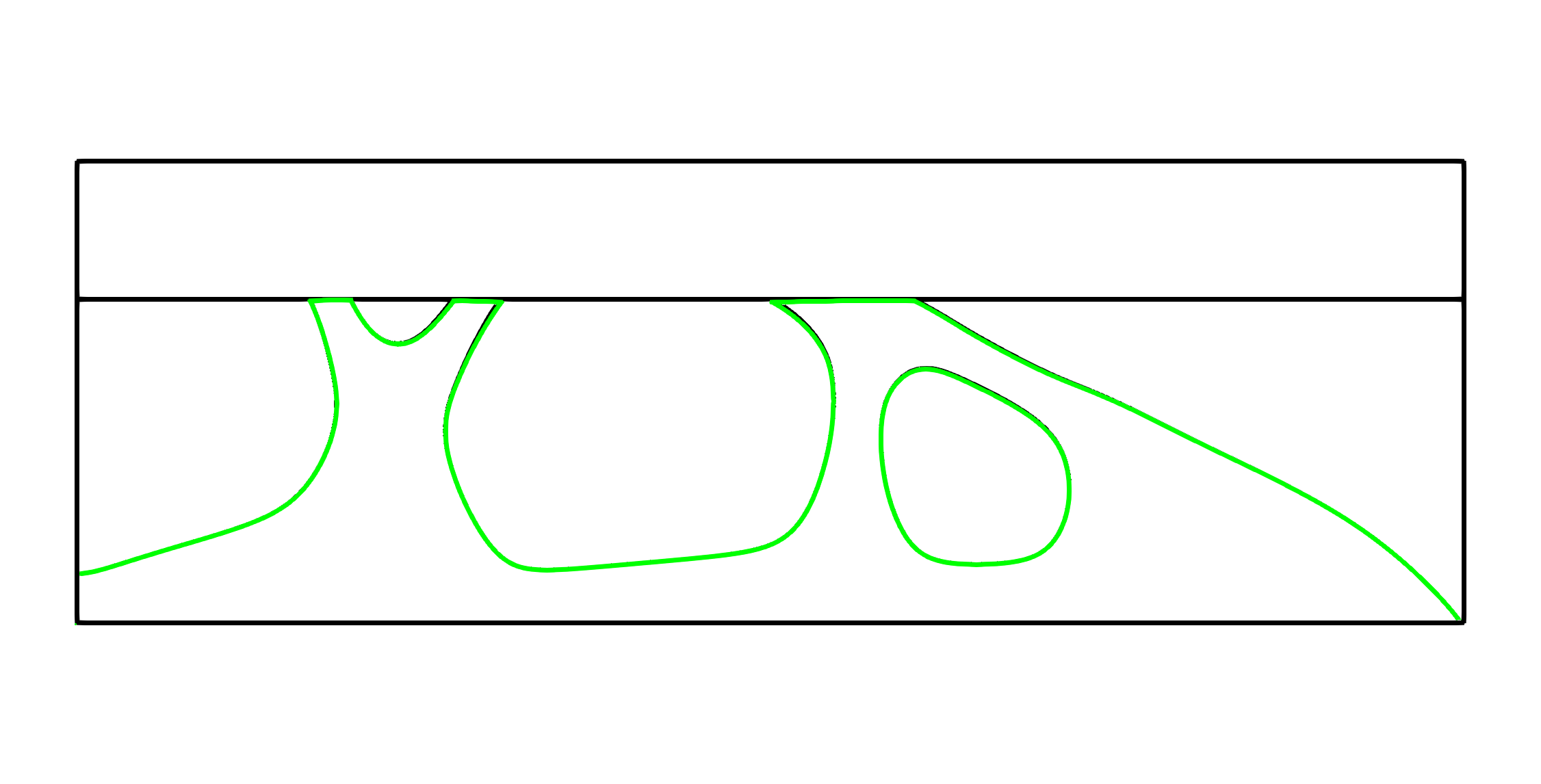}
    \caption{$\beta_1=48$, $k=7$}
    \label{fig:48_7_disp}
  \end{subfigure}
\centering
  \begin{subfigure}{0.24\textwidth}
    \centering \includegraphics[width=\linewidth,clip]{
      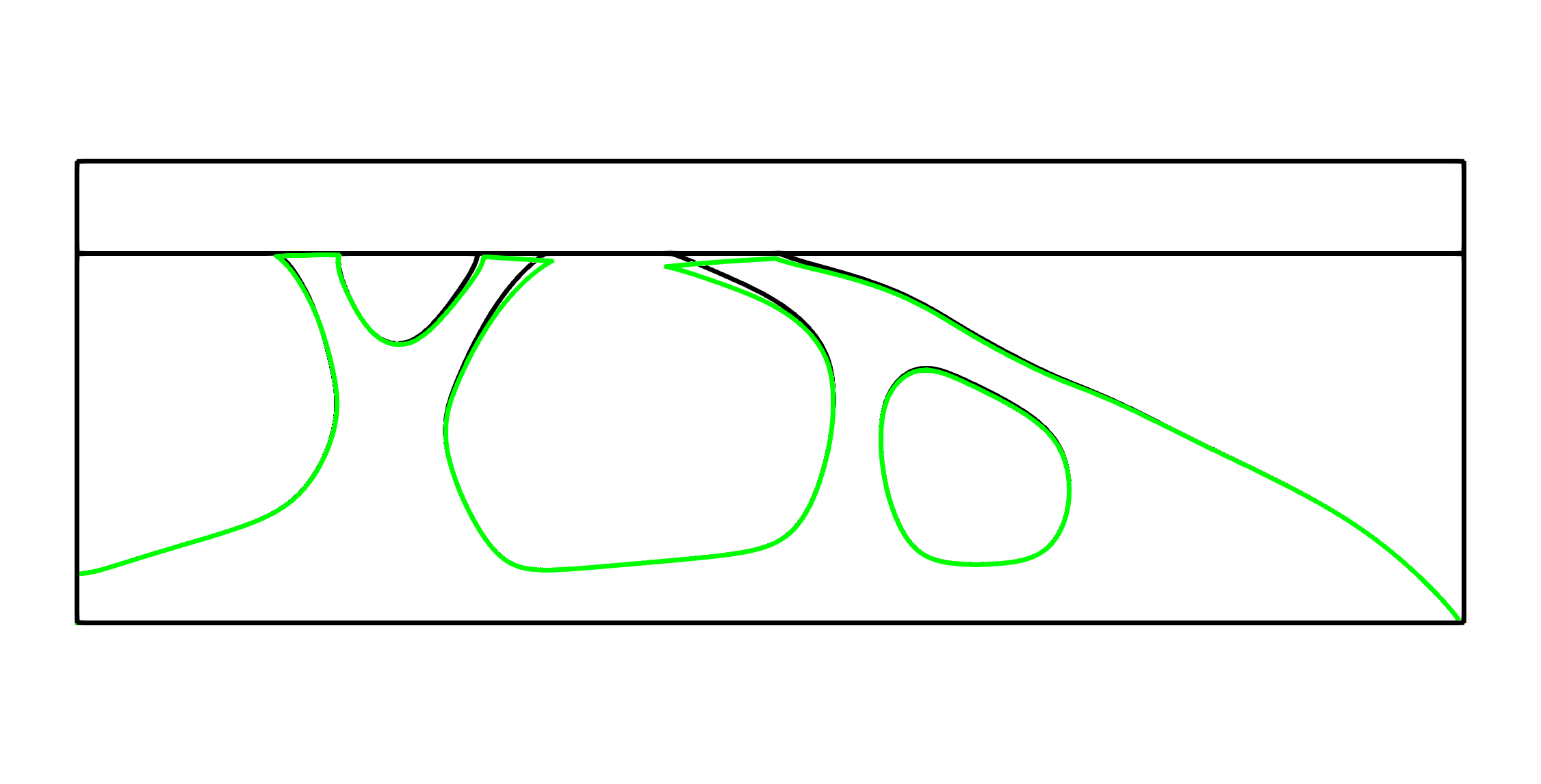}
    \caption{$\beta_1=48$, $k=8$}
    \label{fig:48_8_disp}
  \end{subfigure}
\centering
  \begin{subfigure}{0.24\textwidth}
    \centering \includegraphics[width=\linewidth,clip]{
      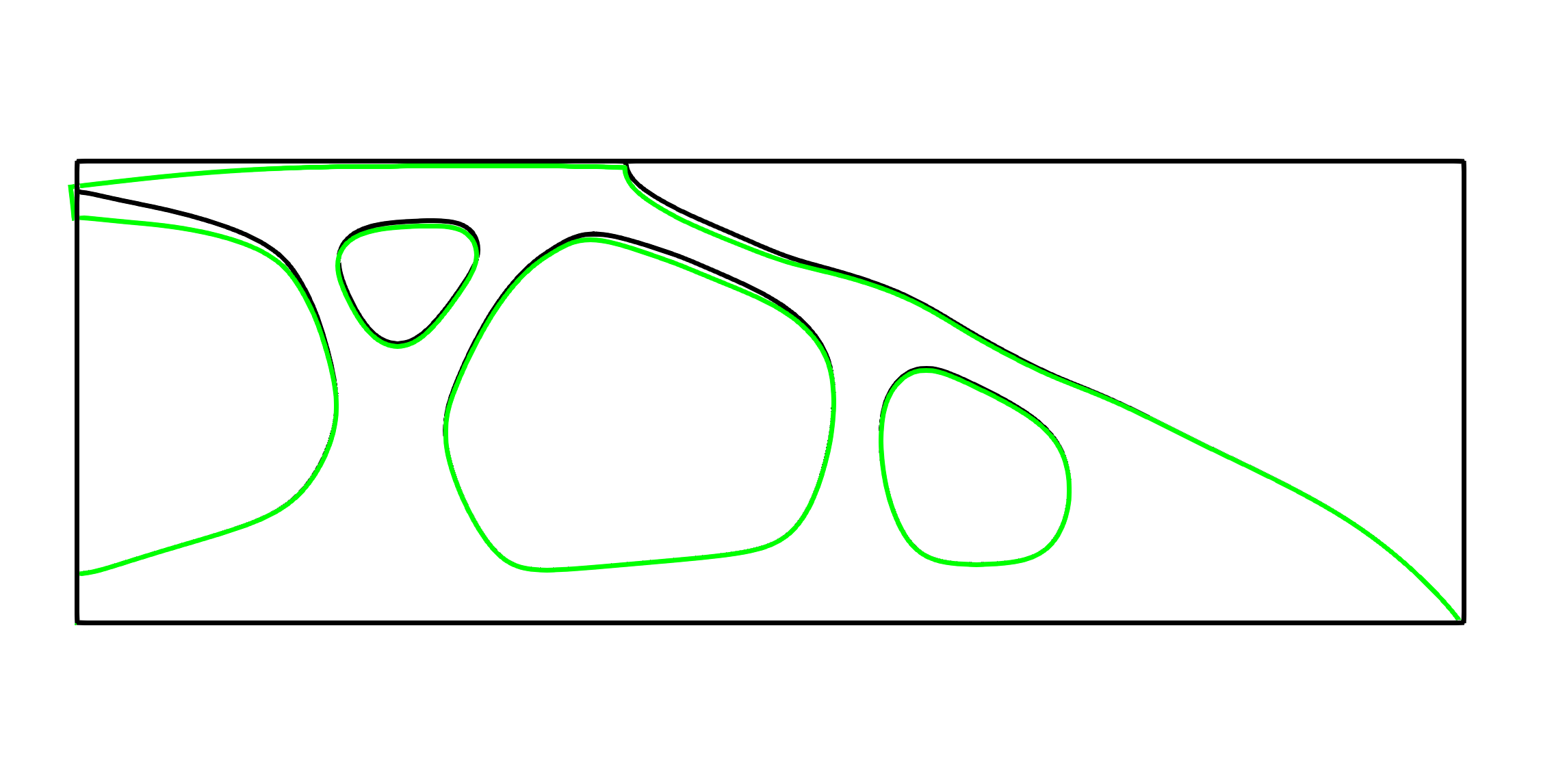}
    \caption{$\beta_1=48$, $k=10$}
    \label{fig:48_10_disp}
  \end{subfigure}
  
\caption{Level sets of intermediate structures (black) and their corresponding displaced version (green) for different $\Omega_{k\delta}$}
\label{fig:Displacements}
\end{figure}
\end{center}

In Figure \ref{fig:overhangV}, the black lines depict overhangs, which violate the
  often employed 45\textdegree  angle condition  for overhangs in additive manufacturing.
 In this paper, constructability is measured with the self-weights to control the deformations 
$u_k^c$,
which occur during the construction process.
Therefore in Figure \ref{fig:Displacements},
the structures (in black) and
the by gravity deformed structures (in green)
are presented at several
construction stages $\Omega_{k\delta }$,
where only the level set $\tilde \varphi (x)=0$
are plotted for better visibility.

\subsubsection{\texorpdfstring{Comparison of three different choices of $W$}{Comparison of three different choices of Fc}}\label{sec:StudyFc}
In Section \ref{sec:AnalysisOCP}, three different choices of $f^c$ and $\omega$ are discussed: we refer to the case $\omega(h)=\frac{1}{h}$, $f^c = f_{grav}$ as $W_1$, the case $\omega(h)\equiv 1$, $f^c = f_{grav}$ as $W_2$ and $\omega(h)= \frac{1}{h}$,
$f^c(h,.) = \chi_{\Omega_h\setminus \Omega_{h-\delta}} f_{grav}$ as $W_3$.
In Table \ref{tab:MDependenceWeighting}, the impact of increasing $M$ on the individual $W_i$, $i=1,2,3$ is given.
As before for increasing M, we set $\varepsilon = 0.04$. 
It is observed that $W_3$ approximately halves as $M$ doubles, which is consistent with the result $\lim_{M\rightarrow\infty} W_3 = 0$ derived in  Section \ref{sec:AnalysisOCP}. In contrast to this, $W_1$ and $W_2$ decrease only slightly, and $W_1$ stabilizes with increasing $M$. 

\begin{table}[ht!]
\centering
\begin{tabular}{|c|c|c|c|c|c|c|}
\hline
\textbf{$M$} & 10 & 20 & 40 & 50 & 100 \\
\hline
\hline
\textbf{$W_1$} & 1.478E-3 & 1.335E-3 & 1.265E-3 & 1.254E-3 & 1.253E-3 \\
\hline
\textbf{$W_2$} & 1.265E-3 & 1.114E-3 & 1.068E-3 & 1.053E-3 & 1.023E-3 \\
\hline
\textbf{$W_3$} & 7.245E-4 & 3.573E-4 & 1.127E-4 & 8.870E-5 & 3.894E-5 \\
\hline
\end{tabular}
\caption{Impact of the layer numbers on different choices of $W$}
\label{tab:MDependenceWeighting}
\end{table}

The shapes and topologies in the case of $W_1$ and $W_2$ seem to change very little for increasing $M$,
see Figure \ref{fig:WeighingMImpact}. 
When using $W_3$,
  the structure changes with $M$.
  For $M=100$, the result closely resembles
  the case of $\beta_1=0$,
  i.e. when the construction phase is neglected.
  This is in accordance with the analytical result
  $\lim_{M\to\infty}\bar W_\Delta (\varphi)=0$
  for the choice of $W_3$,  given
   in Section \ref{sec:AnalysisOCP}.
 Hence using $W_3$  given in \cite{WIAS},
 the penalization of the deformations
 during the construction phase shrink
  for increasing $M$.
\begin{center}
\begin{figure}[th!]
\centering
  \begin{subfigure}{0.32\textwidth}
    \centering \includegraphics[width=\linewidth,trim={0cm 6cm 0cm 5cm},clip]{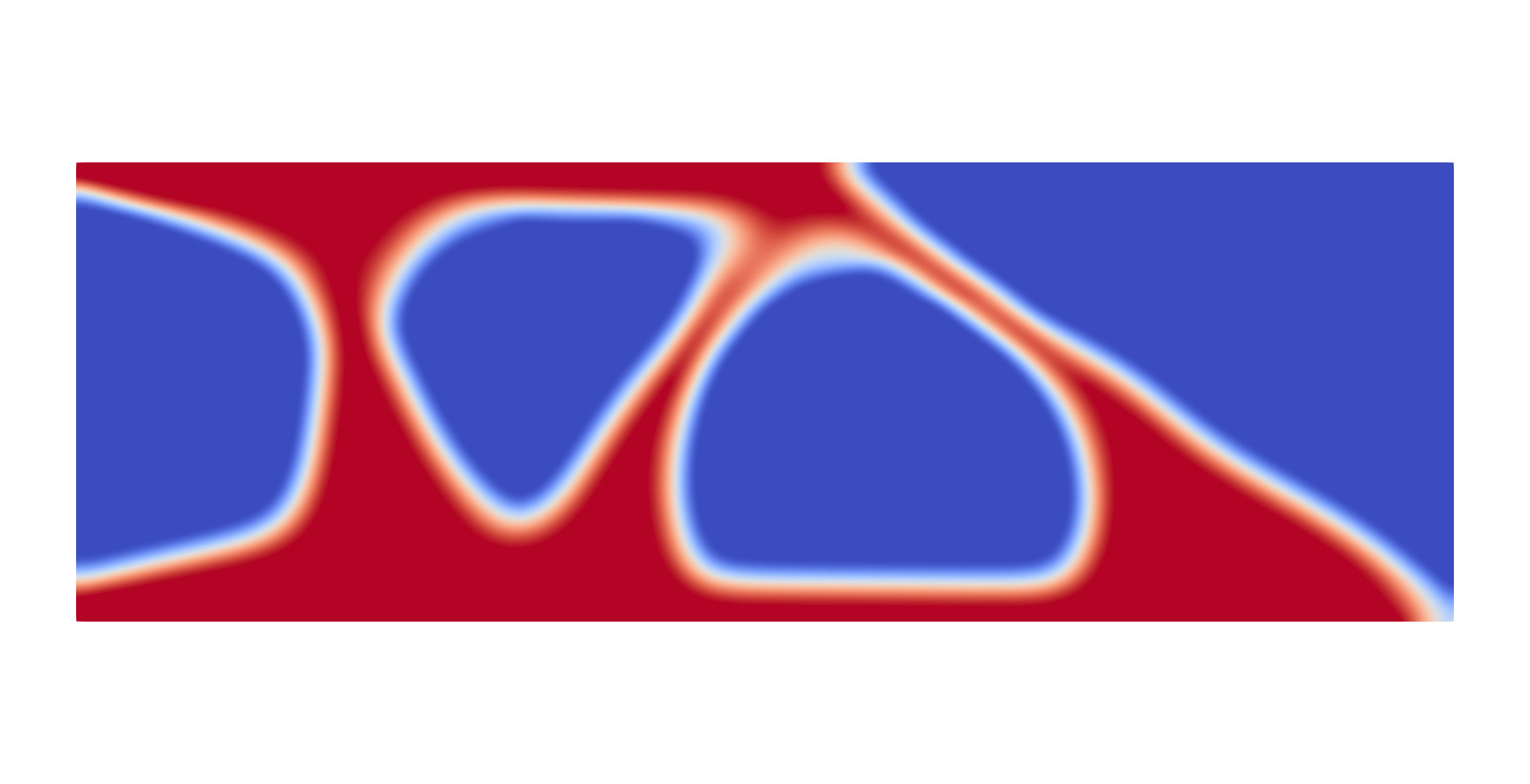}
    \caption{$W_1$, $M=10$}
    \label{fig:meanweighted10}
  \end{subfigure}
\centering
  \begin{subfigure}{0.32\textwidth}
    \centering \includegraphics[width=\linewidth,trim={0cm 6cm 0cm 5cm},clip]{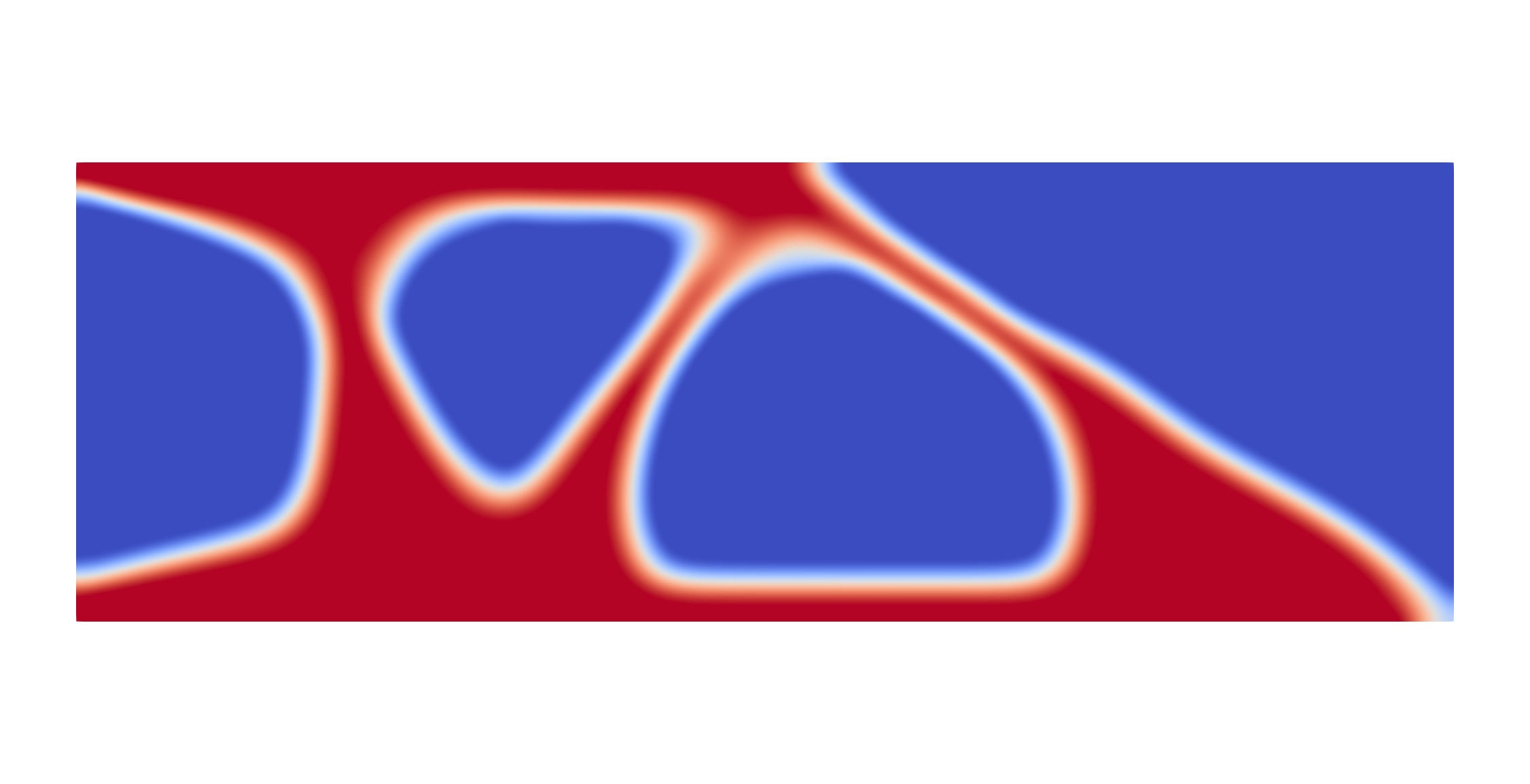}
    \caption{$W_2$, $M=10$}
    \label{fig:unweighted10}
  \end{subfigure}
\centering
  \begin{subfigure}{0.32\textwidth}
    \centering \includegraphics[width=\linewidth,trim={0cm 6cm 0cm 5cm},clip]{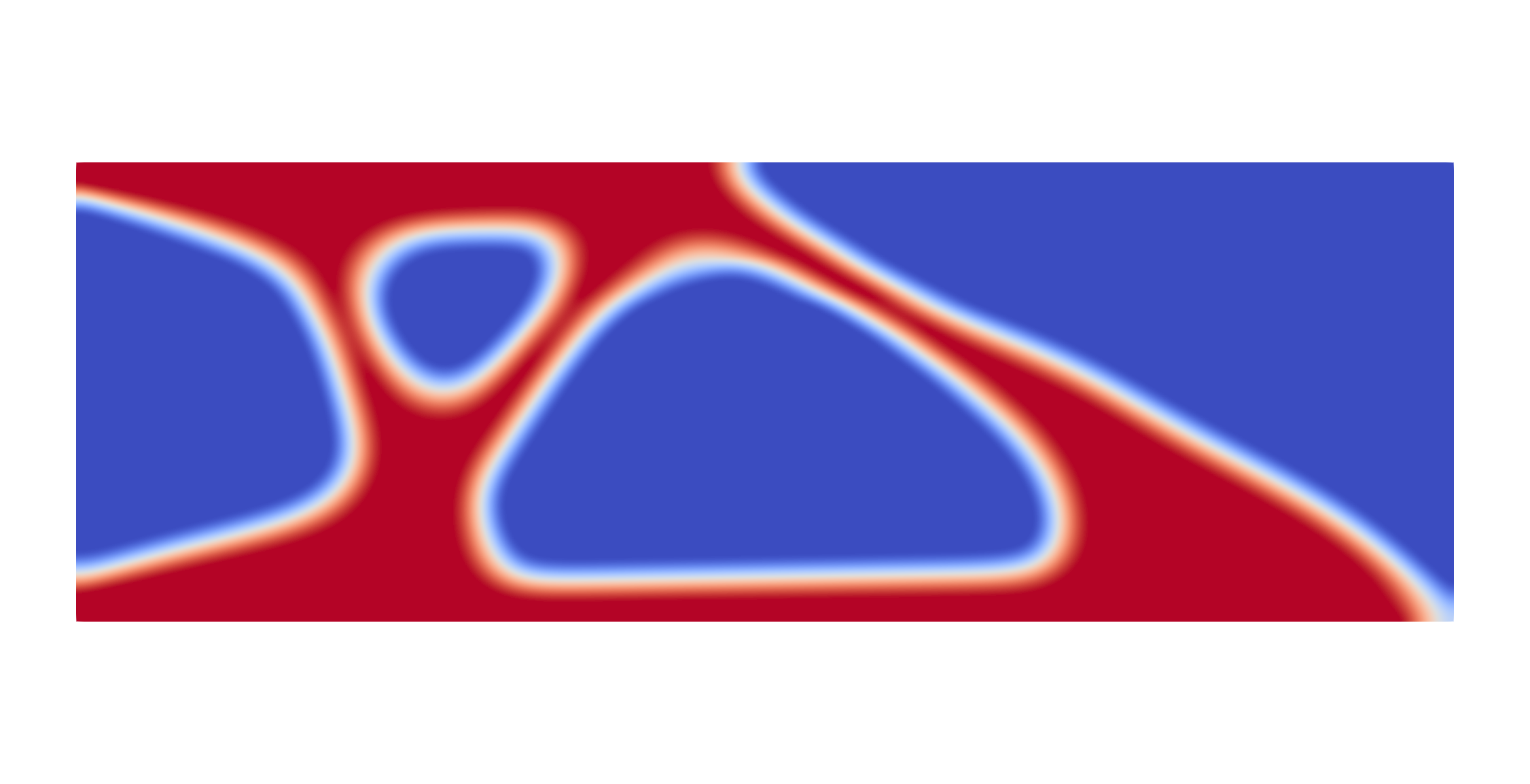}
    \caption{$W_3$, $M=10$}
    \label{fig:layerweighted10}
  \end{subfigure}

\centering
  \begin{subfigure}{0.32\textwidth}
    \centering \includegraphics[width=\linewidth,trim={0cm 6cm 0cm 5cm},clip]{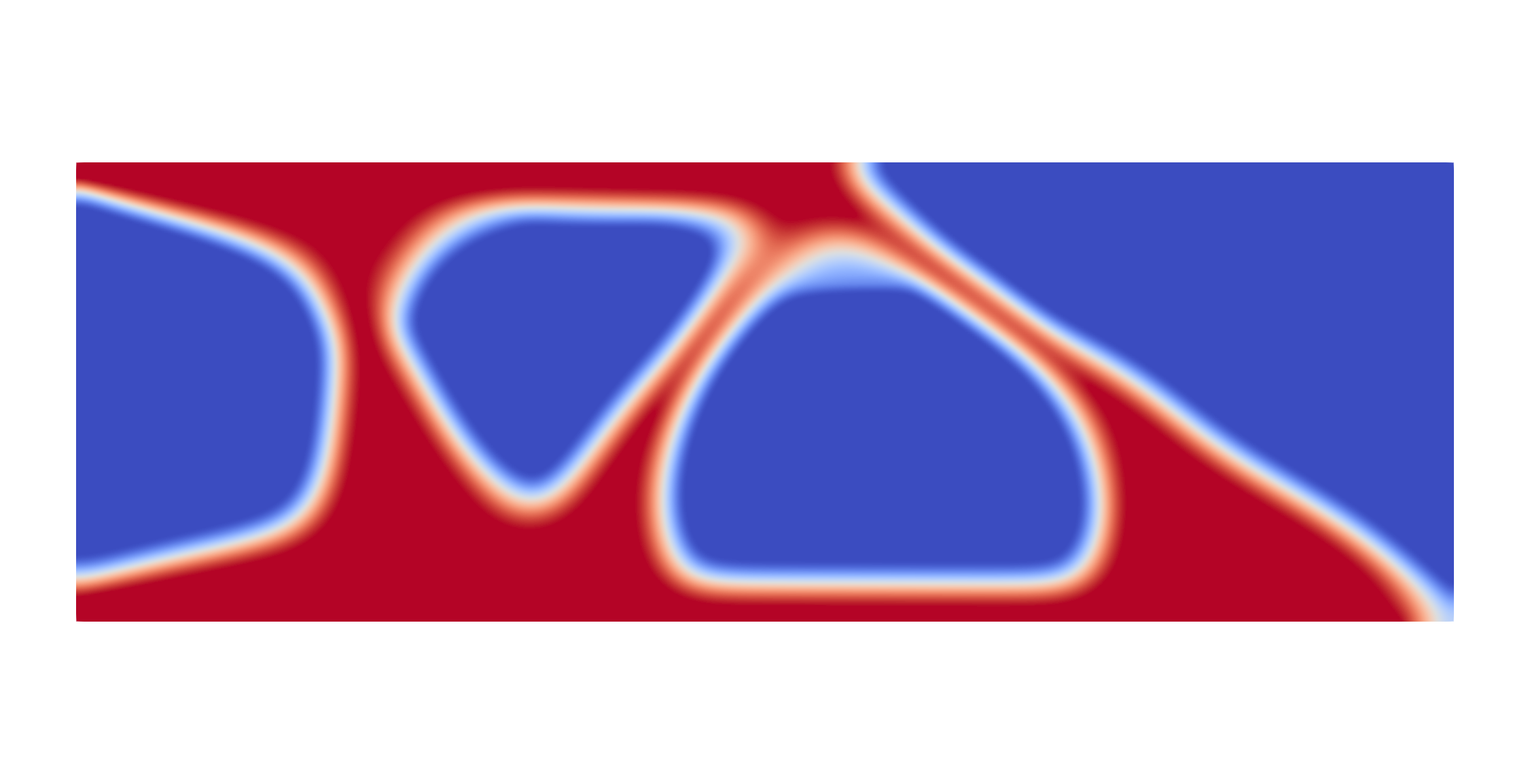}
    \caption{$W_1$, $M=100$}
    \label{fig:meanweighted100}
  \end{subfigure}
\centering
  \begin{subfigure}{0.32\textwidth}
    \centering \includegraphics[width=\linewidth,trim={0cm 6cm 0cm 5cm},clip]{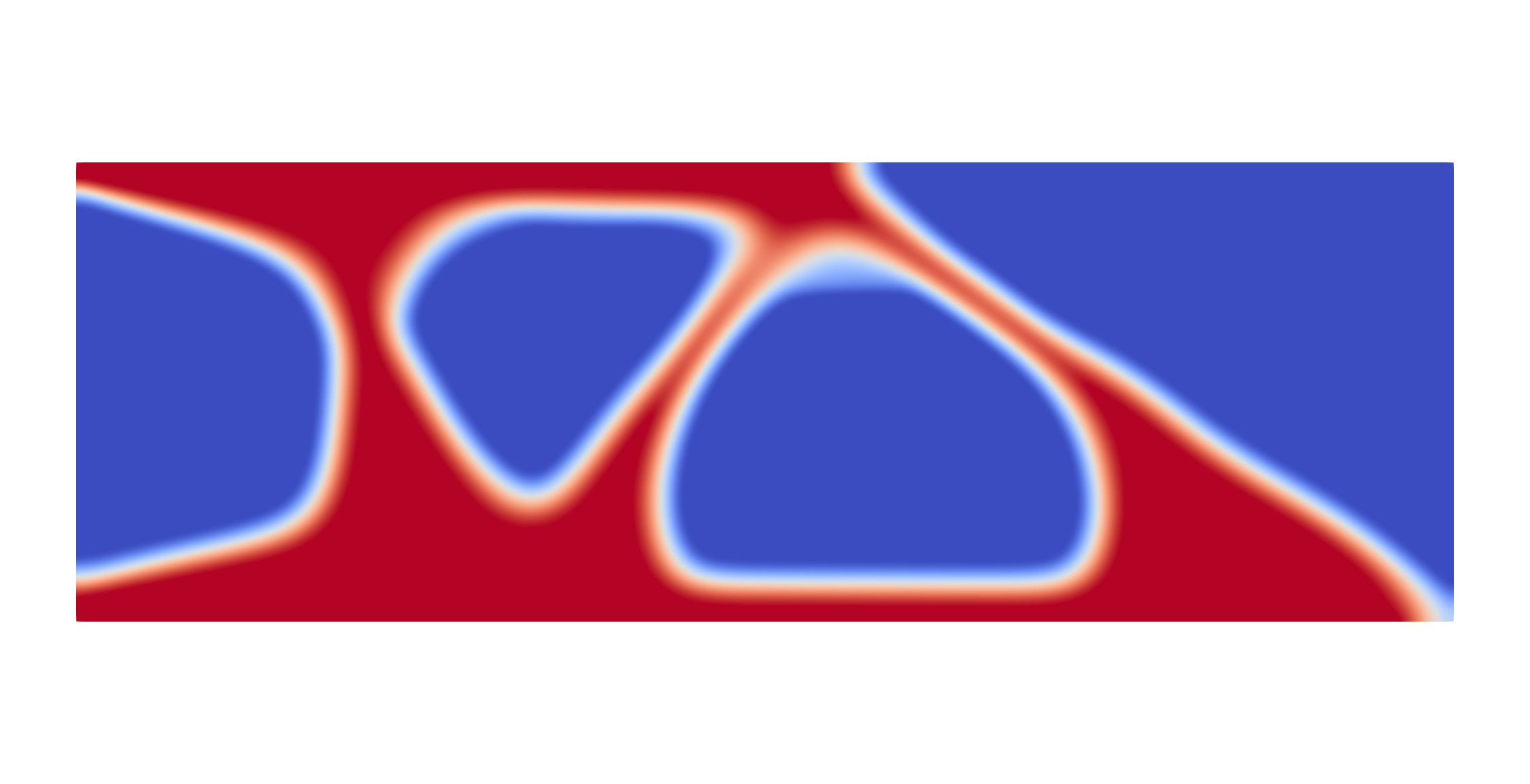}
    \caption{$W_2$, $M=100$}
    \label{fig:unweighted100}
  \end{subfigure}
\centering
  \begin{subfigure}{0.32\textwidth}
    \centering \includegraphics[width=\linewidth,trim={0cm 6cm 0cm 5cm},clip]{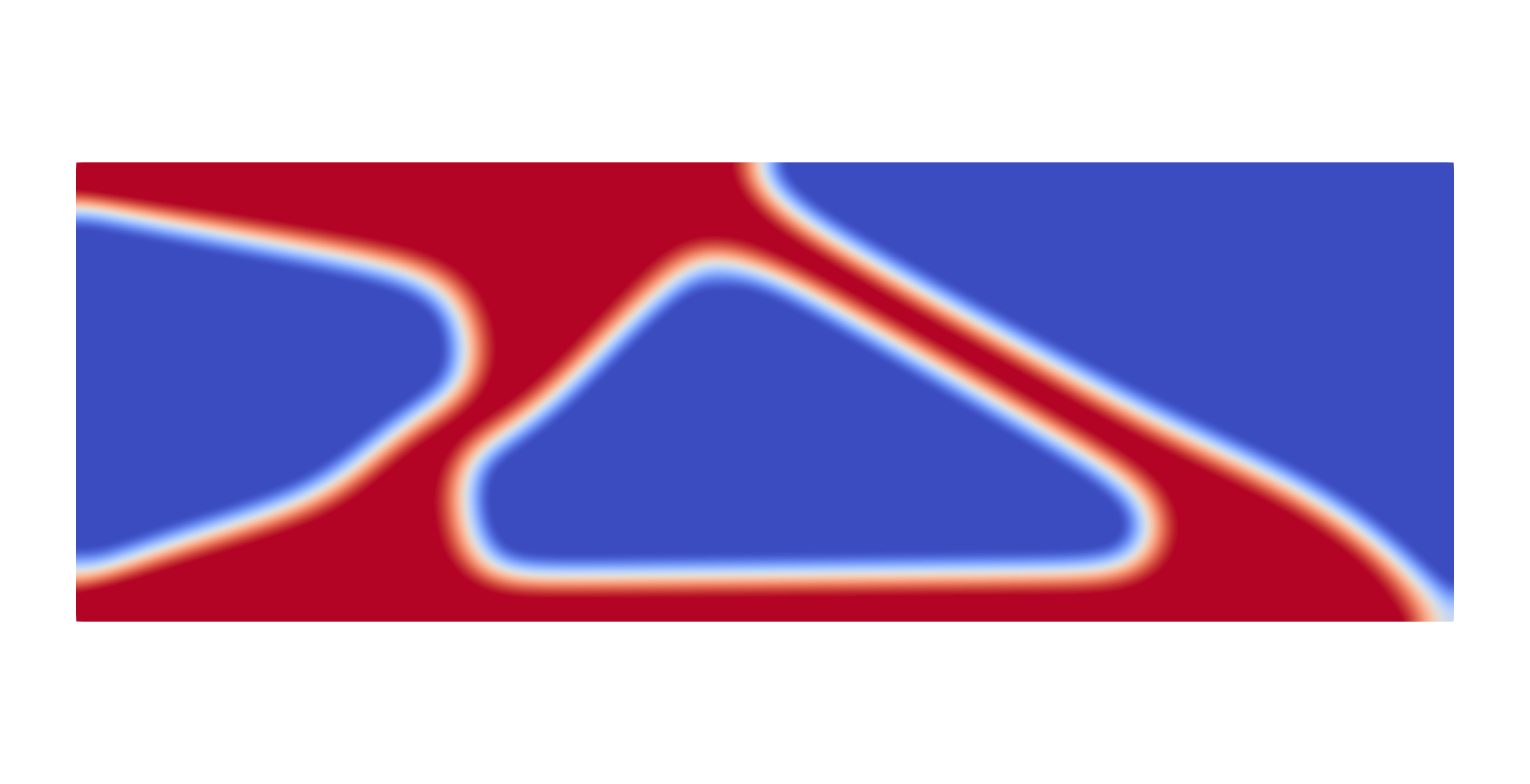}
    \caption{$W_3$, $M=100$}
    \label{fig:layerweighted100}
  \end{subfigure}
\caption{Comparison of different $W_i$ with increasing $M$ for the test example}
\label{fig:WeighingMImpact}
\end{figure}
\end{center}
To point out the  different impacts of
$W_1$ and of $W_2$,
we vary the test setting by relocating
the load $g$ to the right hand side of the cantilever, i.e. to
$\Gamma_N=\{3\}\times[0.4,0.6]$.
To balance the terms in the cost functional, we choose
and we set $\beta_1=20$, $\beta_2=0.0008$ and $M=20$. 
The results are depicted in Figure \ref{fig:WeighingImpactStructures}.
The structures obtained including the construction phase deviate strongly from the solution without AM. 
Furthermore, both structures with AM
feature two supports for the bar in the lower right corner where the outer force is acting.
These supports seem to be slightly thicker in case of $W_1$.
In Figure \ref{fig:3Pcomp1}, more vertical bars
can be seen supporting the top regions compared to Figure \ref{fig:3Pcomp2}, where supports seem to be more equally distributed, for example, in the large void region in the center of the domain. 
This affirms the claim in \cite{AllaireOverview1,WIAS}, that
$W_2$ penalizes higher located
overhangs
stronger than lower ones.
\begin{center}
\begin{figure}[th!]
\centering
  \begin{subfigure}{0.32\textwidth}
    \centering \includegraphics[width=\linewidth,trim={0cm 6cm 0cm 5cm},clip]{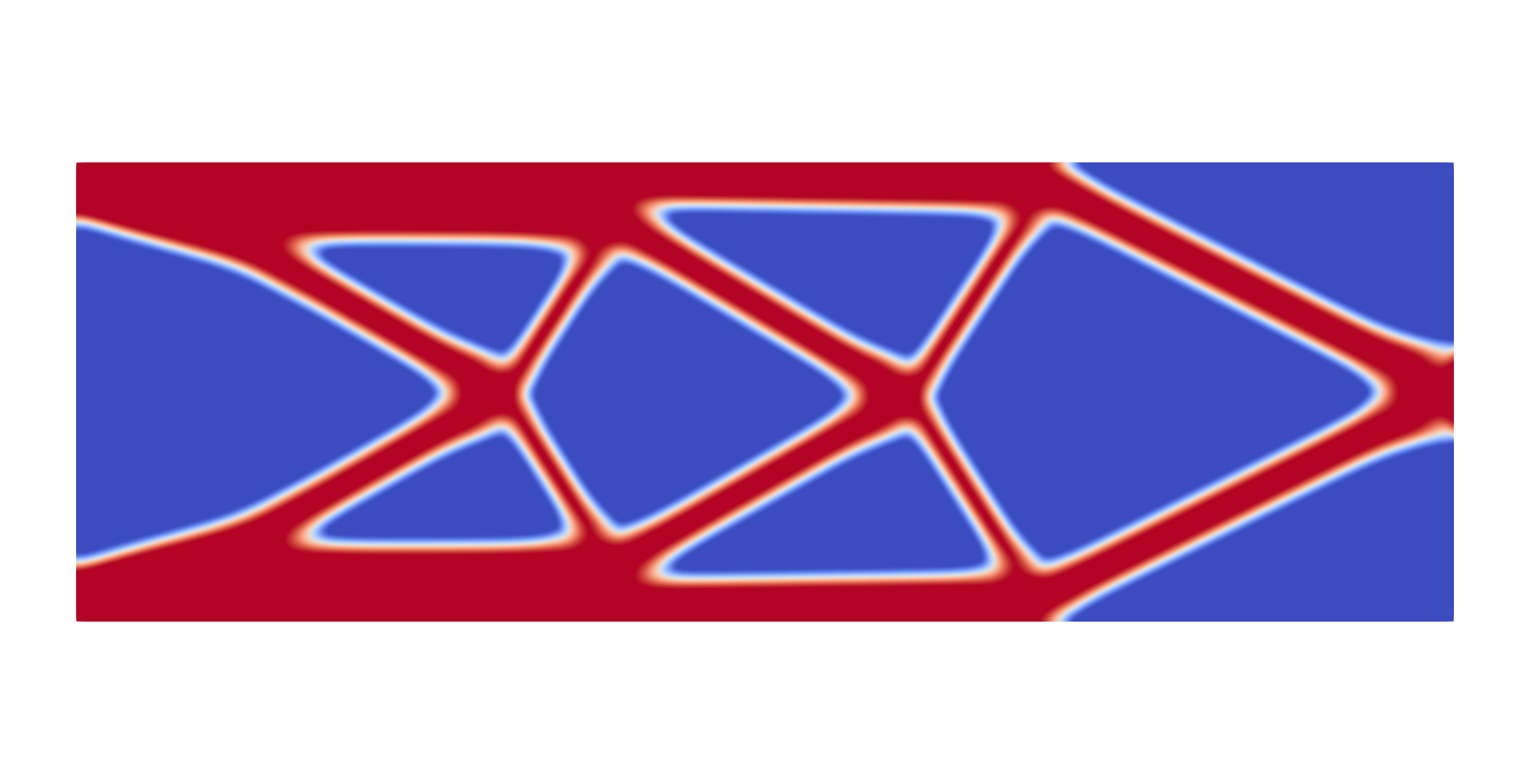}
    \caption{Structure without AM}
    \label{fig:3NoAM}
  \end{subfigure}
\centering
  \begin{subfigure}{0.32\textwidth}
    \centering \includegraphics[width=\linewidth,trim={0cm 6cm 0cm 5cm},clip]{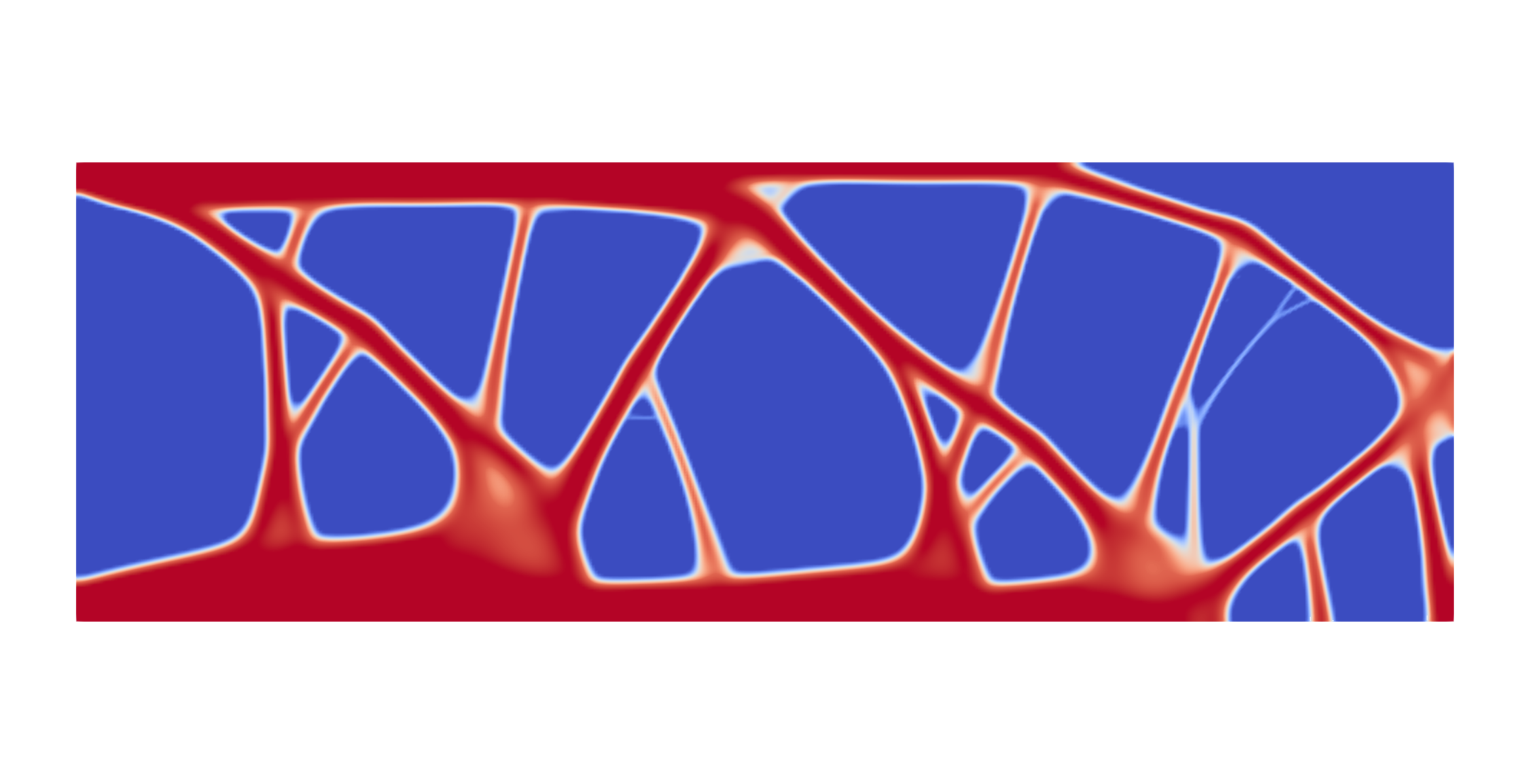}
    \caption{Structure using $W_1$}
    \label{fig:3Pcomp2}
  \end{subfigure}
\centering
  \begin{subfigure}{0.32\textwidth}
    \centering \includegraphics[width=\linewidth,trim={0cm 6cm 0cm 5cm},clip]{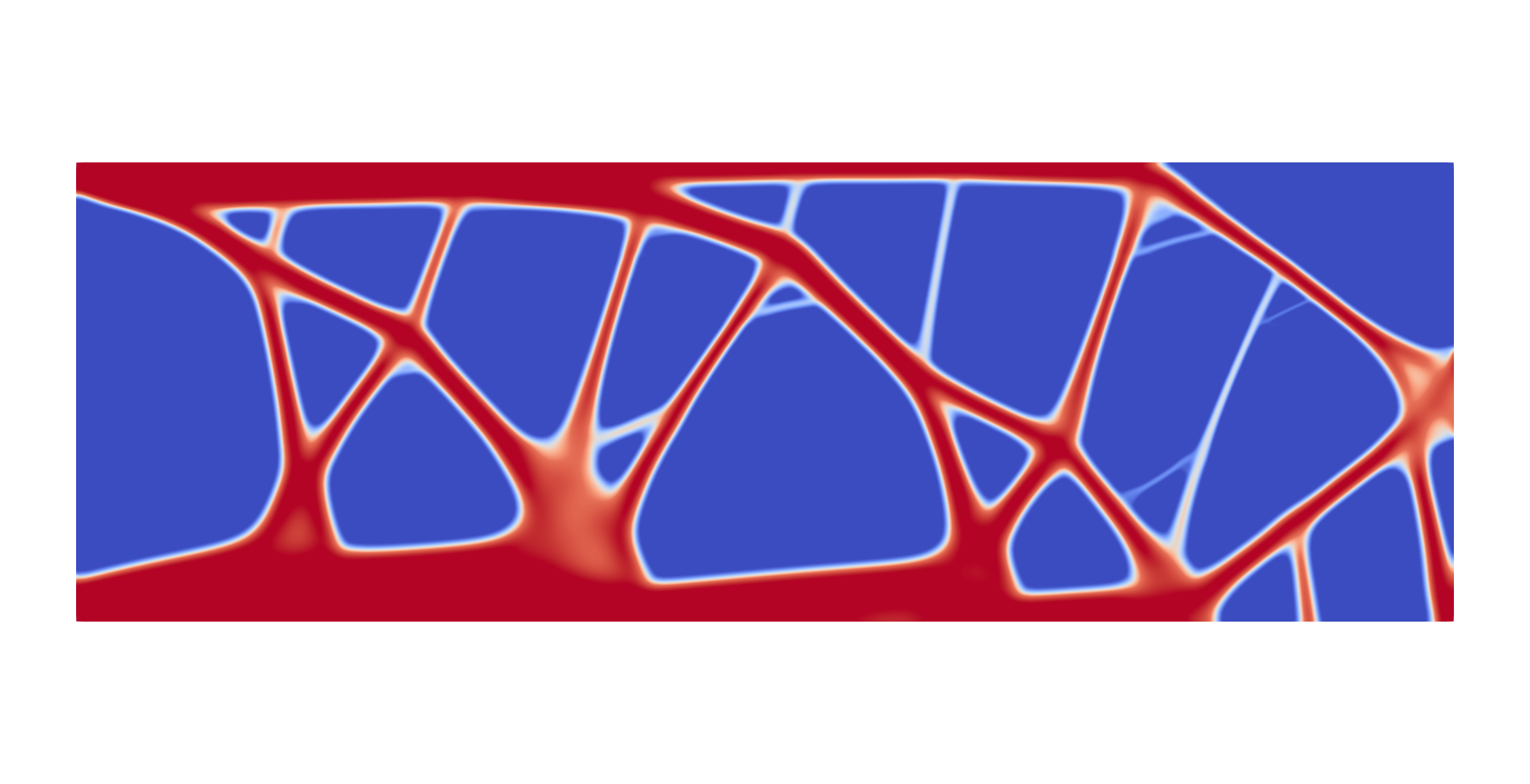}
    \caption{Structure using $W_2$}
    \label{fig:3Pcomp1}
  \end{subfigure}
\centering
\caption{Impact of $W$ on the structure when
$\Gamma_N= \{3\}\times[0.4,0.6]$}
\label{fig:WeighingImpactStructures}
\end{figure}
\end{center}

\subsubsection{Hardening of the structure}
In Section \ref{sec:Intro}, we introduced the possibility of a height dependent stress tensor $\mathcal{C}^c$ during construction, enabling the modelling of hardening effects.
To this end, we assume that in the given test example with 10 layers, the material is fully hardened within 3 layers with linear speed, i.e.
$\lambda^c_1=\mu^c_1$ take the values 32, 36, 40 and then 44 for lower layers.
The obtained local minima is depicted in Figure
\ref{fig:hardening10},
while in Figure \ref{fig:meanweighted10nh}
the result without hardening is given for comparison.
As expected, when hardening is included,
larger
 overhangs are allowed.
However, there are fewer
overhangs 
with hardening than using
the $W_3$ setting (see Figure \ref{fig:layerweighted10nh}), where gravity  acts
only on the current topmost layer.
\begin{center}
\begin{figure}[th!]
\centering
  \begin{subfigure}{0.32\textwidth}
    \centering \includegraphics[width=\linewidth,trim={0cm 6cm 0cm 5cm},clip]{Figures/Weighting/mean_weighted_10.png}
    \caption{Structure using $W_1$ \\
     without hardening}
    \label{fig:meanweighted10nh}
  \end{subfigure}
\centering
  \begin{subfigure}{0.32\textwidth}
    \centering \includegraphics[width=\linewidth,trim={0cm 6cm 0cm 5cm},clip]{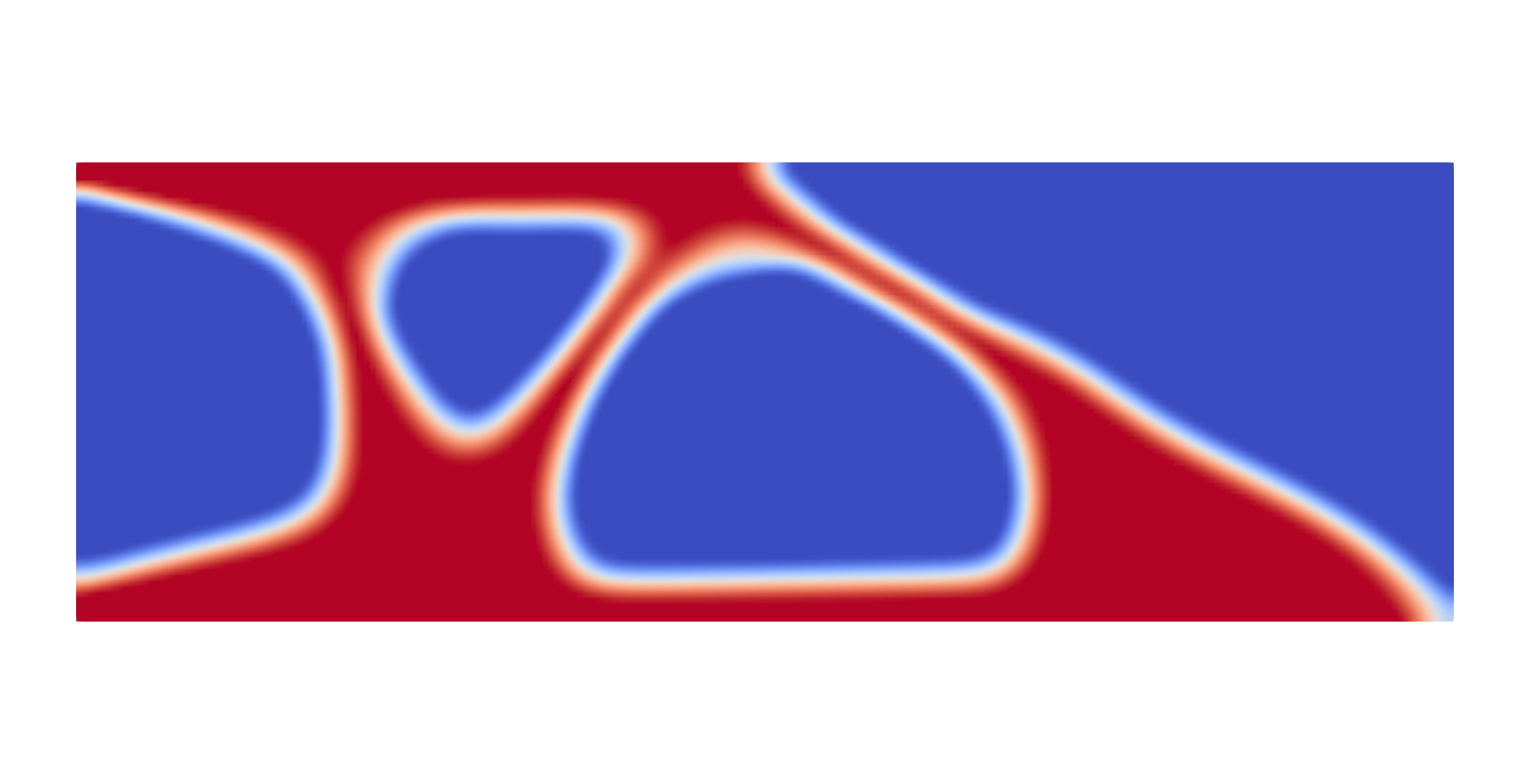}
    \caption{Structure using $W_1$ \\
     with hardening}
    \label{fig:hardening10}
  \end{subfigure}
\centering
  \begin{subfigure}{0.32\textwidth}
    \centering \includegraphics[width=\linewidth,trim={0cm 6cm 0cm 5cm},clip]{Figures/Weighting/layer_weighted_10.png}
    \caption{Structure using $W_3$ \\
     without hardening}
    \label{fig:layerweighted10nh}
  \end{subfigure}
\centering
\caption{Comparison of results with and without hardening effects}

\label{fig:HardeningImpactStructures}
\end{figure}
\end{center}

\subsubsection{Impact of different building plate locations on the final structure}
The location of the building plate and the corresponding building direction has a large impact on the final structure.
To illustrate this,
the solutions for our test problem, where the building plate is located on the bottom or on left hand side, 
are depicted in Figure \ref{fig:BuildingPlateLoc}.
The corresponding data are given in Table \ref{tab:BuildingPlateLoc}.
Although for large $M$ the solution hardly changes according to Subsection \ref{sec:StudyFc}, for small $M$ the thickness of the layers may play a role. Hence,
we fixed the thickness of the layers, resulting into $M=10$ layers for the location at the bottom and
into $M=30$ layers otherwise.
The resulting topologies differ drastically. 
In this example, locating the building plate on the left hand side yields a structure with a lower perimeter and lower $F$, though larger $W$ than locating the building plate on the bottom.
In particular, it resembles the topology of the structure obtained by neglecting the construction phase.
\begin{table}[th!]
\centering
\begin{tabular}{|c|c|c|c|c|}
\hline
\textbf{$\Gamma_B$ location} & \textbf{$j$} & \textbf{$\tilde{E}$} & \textbf{$F$} &  \textbf{$W$} \\
\hline
\hline
\textbf{bottom}&0.347160  & 13.238477 &0.148654 & 0.001377 \\
\hline
\textbf{left} & 0.339910 & 12.624342 & 0.136045 & 0.001617 \\
\hline
\end{tabular}
\caption{Impact of the building plate location}	
\label{tab:BuildingPlateLoc}
\end{table}
\begin{center}
\begin{figure}[th!]
  \begin{subfigure}{0.32\textwidth}
    \centering \includegraphics[width=\linewidth,trim={0cm 6cm 0cm 5cm},clip]{Figures/Overhang/overhang_0.png}
    \caption{Without consideration of AM effects
    }
    \label{fig:Beta1ZeroFinal}
  \end{subfigure}
\centering
  \begin{subfigure}{0.32\textwidth}
    \centering \includegraphics[width=\linewidth,trim={0cm 6cm 0cm 5cm},clip]{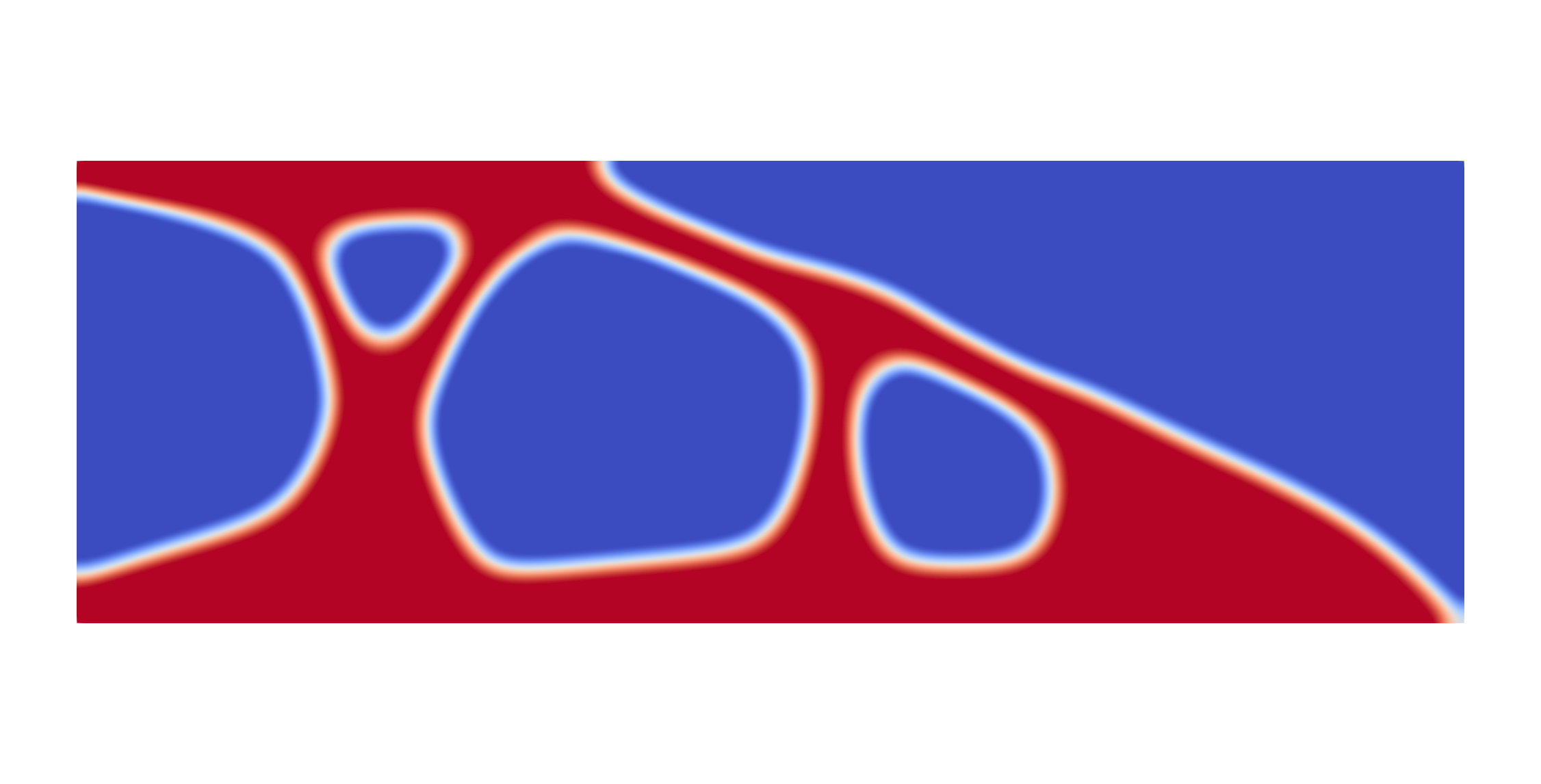}
    \caption{Building plate located on the bottom}
    \label{fig:BottomBuild}
  \end{subfigure}
\centering
  \begin{subfigure}{0.32\textwidth}
    \centering \includegraphics[width=\linewidth,trim={0cm 6cm 0cm 5cm},clip]{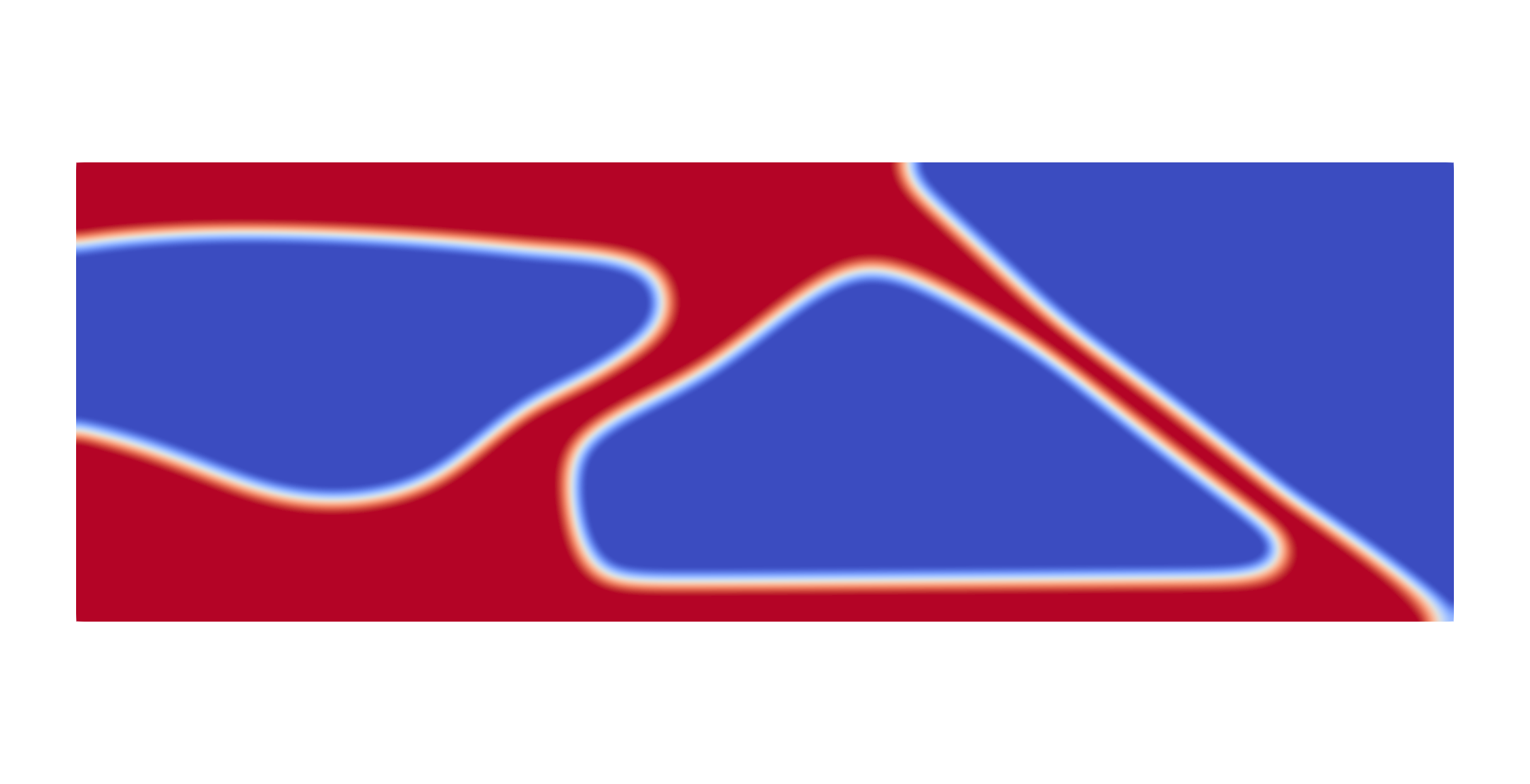}
    \caption{Building plate located on the left}
    \label{fig:LeftBuild}
  \end{subfigure}
\caption{Test example with different building plate locations}
\label{fig:BuildingPlateLoc}
\end{figure}
\end{center}

\subsection{Numerical examples with three phases and in 3D.}
\subsubsection{Example in 2D with three phases (multiphase setting)}\label{Test2D3phase}
In the following test example, we use the multiphase ansatz to model the distribution of a stiff and a weak material, as well as of void based on the cantilever setting given in the test example specification \ref{testEx}. 
However, the outer force $g$ is increased to
$g(\varphi,.) = - 1.5  e_d$, 
$\beta_1=120$ is employed if not mentioned otherwise, and
the maximal mesh width is set to $\delta_H=\varepsilon /4 $. 
While the Lamé parameters for the void and the stiff material remain unchanged,
the parameters for the softer material are chosen as $\lambda_2=\mu_2=32$ and $\lambda_2^c=\mu_2^c=25$.
The mass vector is set to $\mathfrak{m}=(0.2,0.2,0.6)$.

In Table \ref{tab:NestingComparisonMulti}, the iteration numbers of the nested VMPT approach are listed as a reference.
The behaviour is in accordance with the scalar case.
However,  let us mention that the quadratic subproblem is more elaborate given that instead of the constraint $\tilde \varphi(x) \in [-1,1]$ in the scalar case, one has $\varphi(x) \in G$ where $G$ is the Gibbs simplex.
Consequently, the CPU-time for one PDAS iteration is enlarged in the multiphase case.
This leads to an increase of the overall CPU-time even with a similar number of DOFs.
\begin{figure}[th!]
\centering
\begin{subfigure}[t]{0.48\textwidth}
\centering
\vspace{7pt}   
\setlength{\tabcolsep}{4pt}
\resizebox{\linewidth}{!}{
\begin{tabular}{|r |r |r |r |r |r |r |}
\hline
\multicolumn{1}{|c|}{$M$}&\multicolumn{1}{|c|}{$\delta_H$} &  \multicolumn{1}{|c|}{DOFs}& \multicolumn{1}{|c|}{ $
{\mbox{\small VMPT}}
{\mbox{\small iterations}}
$ } \\
\hline
\hline
4      & ${1}/{40}$  & 44~649  
& 644     \\ 
\hline                                             
$8$  & ${1}/{80}$  & 253~773  
& 156 \\
\hline                                             
$10$& ${1}/{100}$& 456~015   
& 116  \\
\hline                                             
$10$& ${1}/{160}$&1~176~105 
& 32  \\
\hline
\hline
\multicolumn{3}{|c|}{\textbf{Total VMPT steps}}
&   948 \\
\hline
\multicolumn{3}{|c|}{\textbf{Mean PDAS steps}}
&   3.41 \\
\hline
\multicolumn{3}{|c|}{\textbf{CPU time}}
&   1h 13min \\
\hline
\hline
\multicolumn{3}{|c|}{\textbf{Final $|j(\varphi_{k+1}) - j(\varphi_k)|$}}
&   6.35E-7 \\
\hline
\end{tabular}}
\vspace{0pt}   
\caption{Results for the nested VMPT approach in  the multiphase setting \ref{Test2D3phase}%
}	
\label{tab:NestingComparisonMulti}
\end{subfigure}
\hfill
\begin{subfigure}[t]{0.4\textwidth}
\centering
\vspace{0pt}   
\includegraphics[width=\linewidth]{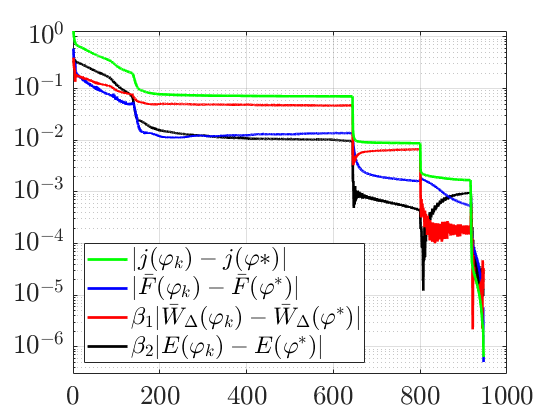}
\caption{Error reductions for the \\
 nested VMPT approach in the\\ 
  multiphase setting \ref{Test2D3phase}}
\label{fig:conv-multi}
\end{subfigure}

\caption{Comparison of nested VMPT results and convergence history.}
\end{figure}

We now focus on the distribution
of the soft and stiff materials
in dependence on $\beta_1$.
The results are depicted in Figure \ref{fig:Beta1ParameterStudyMulti}, where the stiff material corresponds to the red phase and the soft material corresponds to the green phase. The corresponding data is given in Table \ref{tab:Beta1ParameterStudyMultiphase}.
In all observed cases, stiff material is located near $\Gamma_D$. 
When no additive manufacturing is considered, the stiffer material further gathers where the outer force acts, i.e., in a neighborhood of $\Gamma_N$.
If we set at least $\beta_1=10$, this part consists of the soft material, which forms regions with large diameter while the stiff material is primarily used
to form rigid filigree structures.
\begin{table}[th!]
\centering
\begin{tabular}{|c|c|c|c|c|c|c|c|}
\hline
  \textbf{$\beta_1$} & 0 & 5 & 10 & 45 & 90 &  360 \\
\hline
\hline
\textbf{$E$} & 6.717309 & 6.464207 & 6.481637 & 6.481637 & 7.174341 &  8.133135 \\
\hline
\textbf{$F$} & 0.268670 & 0.280294 & 0.281345 & 0.281345 & 0.295840 &  0.369773 \\
\hline
\textbf{$W$} & 0.010774 & 0.003966 & 0.003068 & 0.003068 & 0.001950 &  0.001209 \\
\hline
\end{tabular}
\caption{Parameter study for $\beta_1$
in the multiphase setting \ref{Test2D3phase}}
\label{tab:Beta1ParameterStudyMultiphase}
\end{table}

\begin{center}
\begin{figure}[th!]
\centering
  \begin{subfigure}{0.32\textwidth}
    \centering \includegraphics[width=\linewidth,trim={0cm 6cm 0cm 5cm},clip]{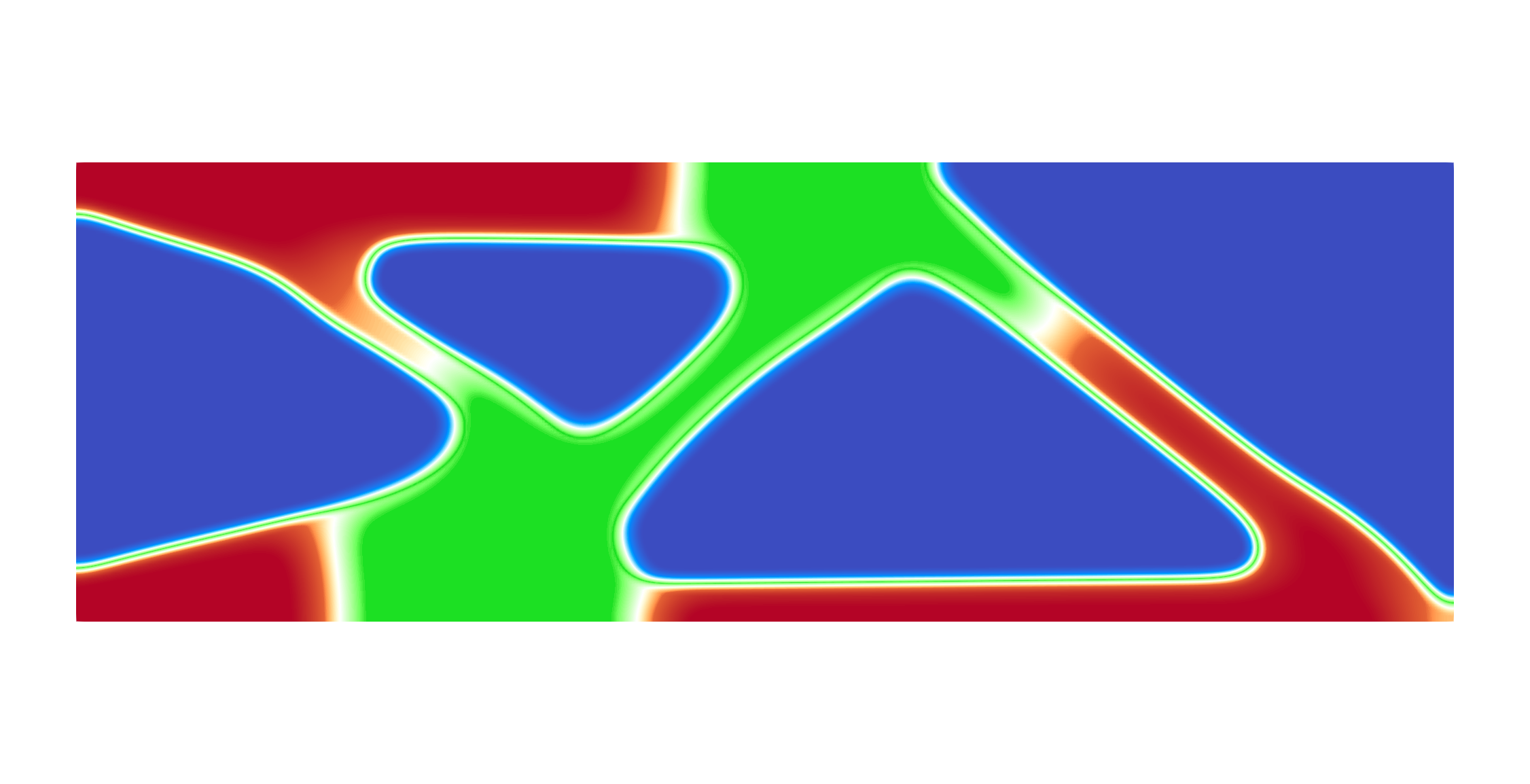}
    \caption{$\beta_1=0$}
    \label{fig:Beta1Multi0}
  \end{subfigure}
  \begin{subfigure}{0.32\textwidth}
    \centering \includegraphics[width=\linewidth,trim={0cm 6cm 0cm 5cm},clip]{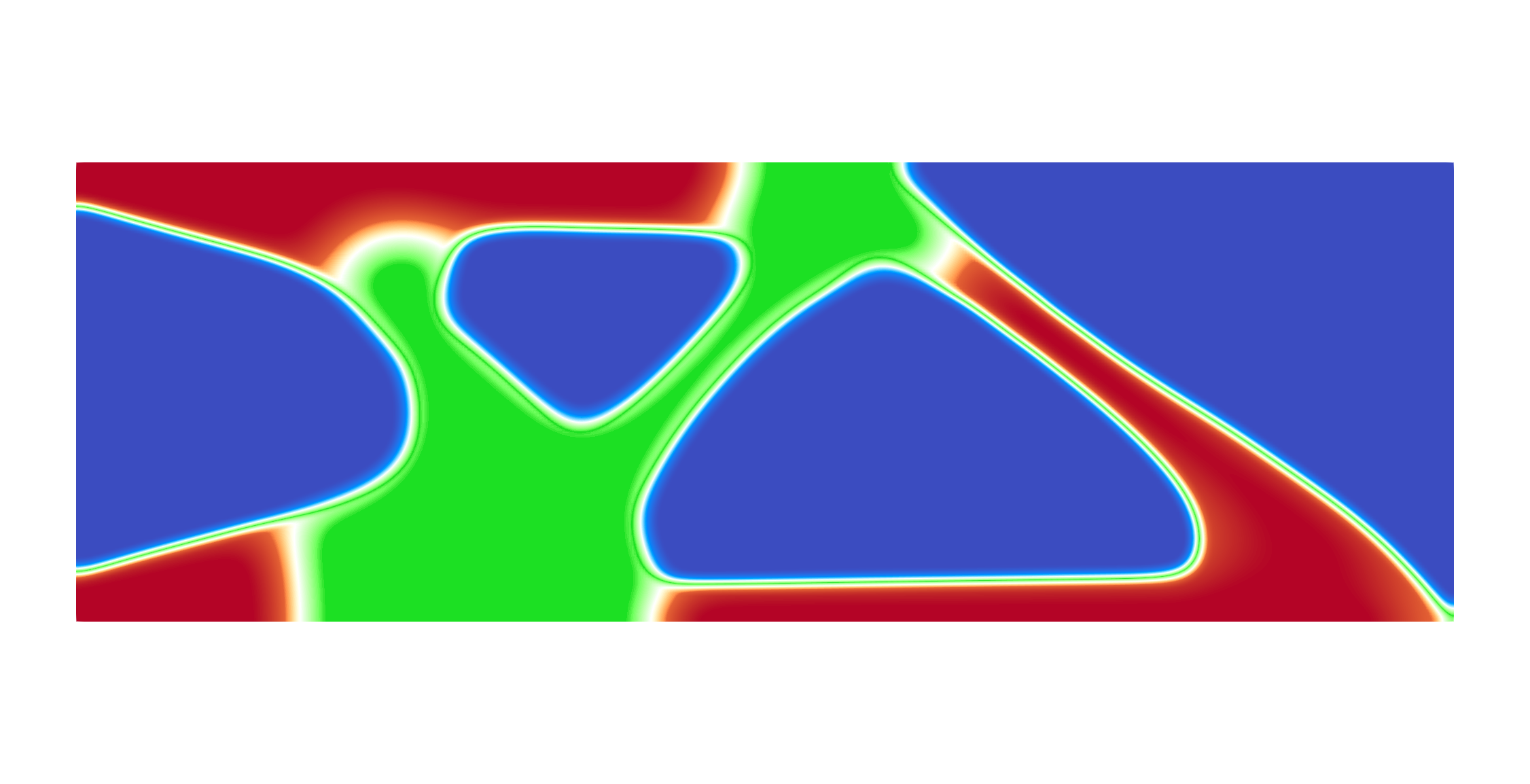}
    \caption{$\beta_1=5$}
    \label{fig:Beta1Multi5}
  \end{subfigure}
  \begin{subfigure}{0.32\textwidth}
    \centering \includegraphics[width=\linewidth,trim={0cm 6cm 0cm 5cm},clip]{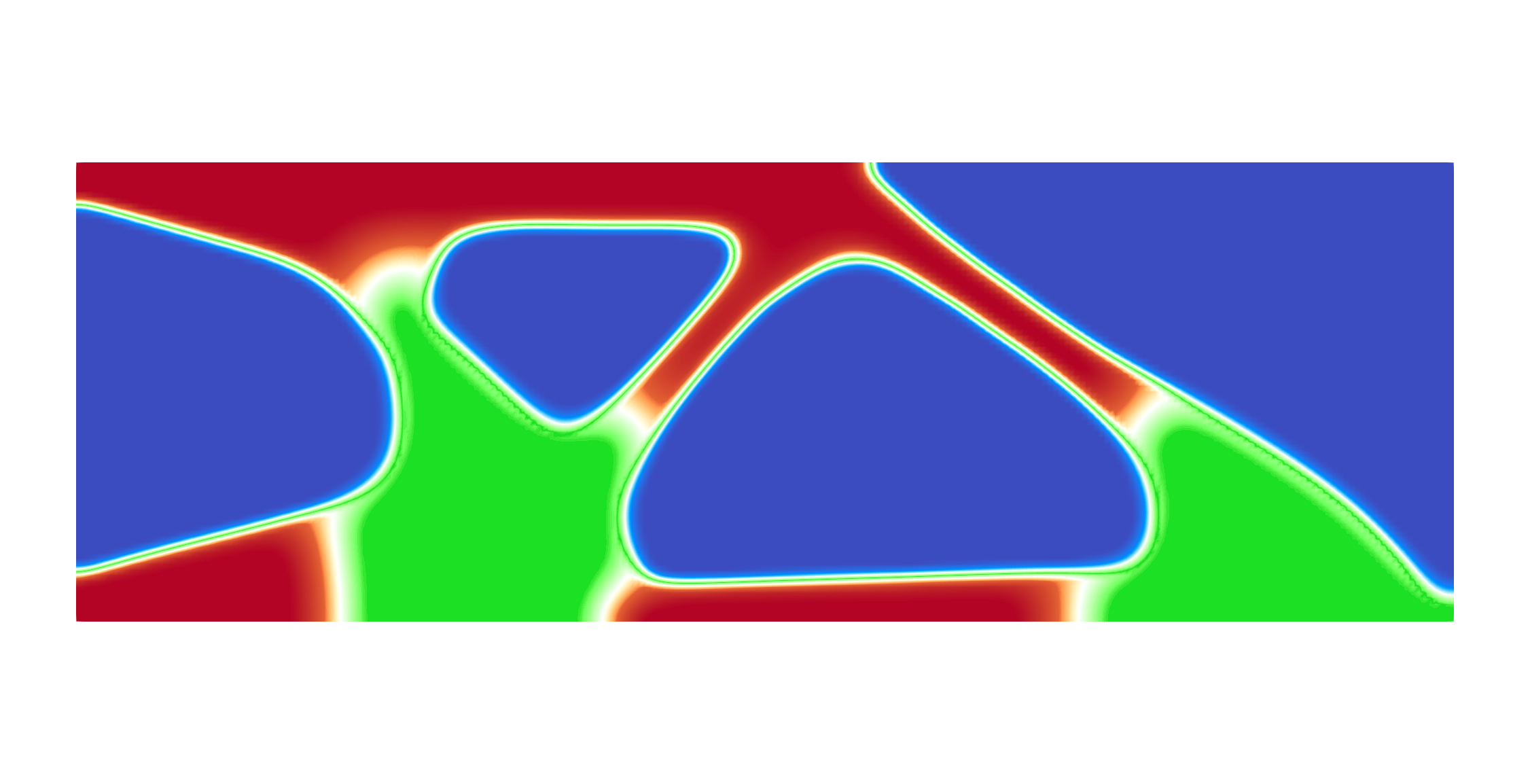}
    \caption{$\beta_1=10$}
    \label{fig:Beta1Multi10}
  \end{subfigure}
  
\centering
  \begin{subfigure}{0.32\textwidth}
    \centering \includegraphics[width=\linewidth,trim={0cm 6cm 0cm 5cm},clip]{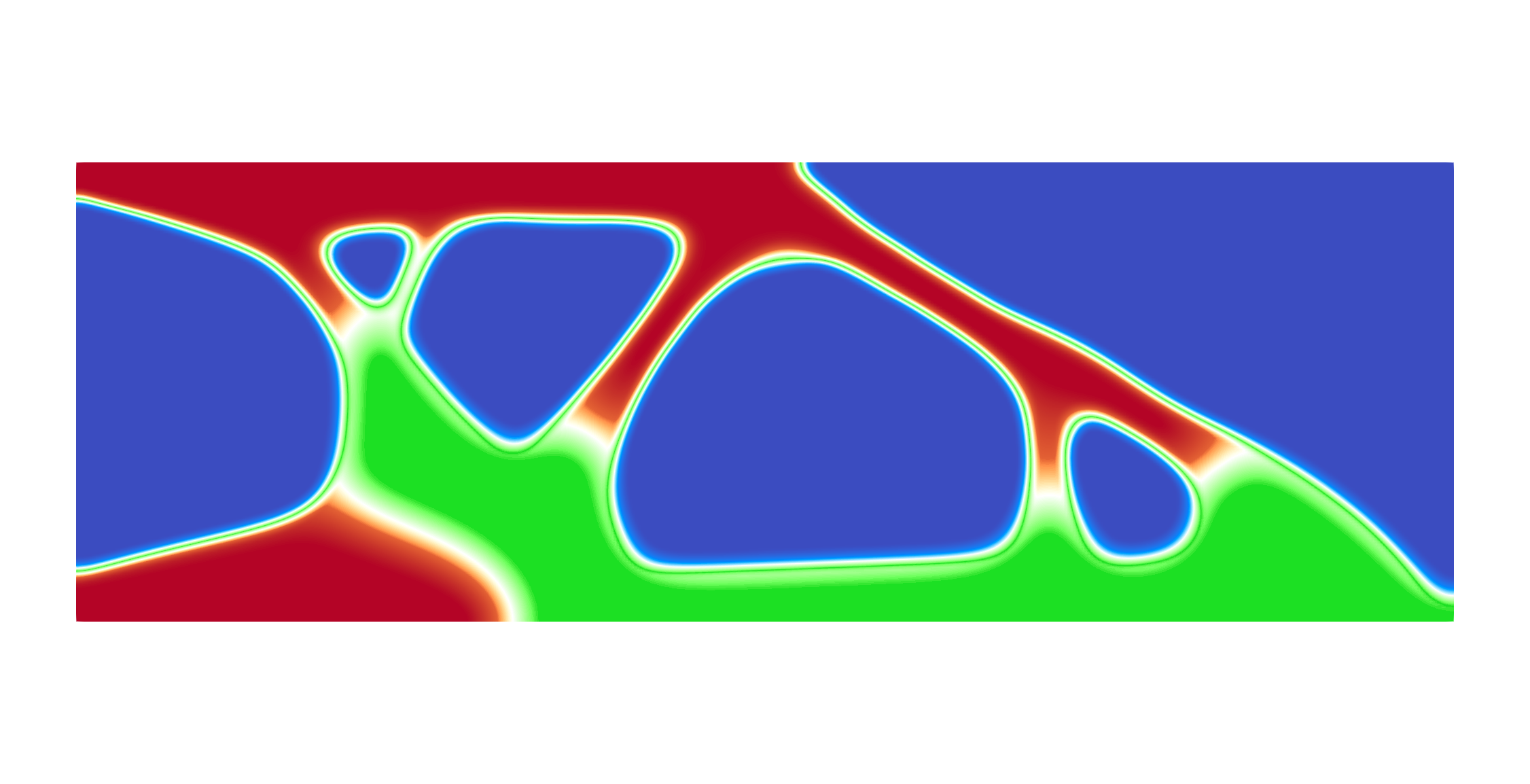}
    \caption{$\beta_1=45$}
    \label{fig:Beta1Multi45}
  \end{subfigure}
  \begin{subfigure}{0.32\textwidth}
    \centering \includegraphics[width=\linewidth,trim={0cm 6cm 0cm 5cm},clip]{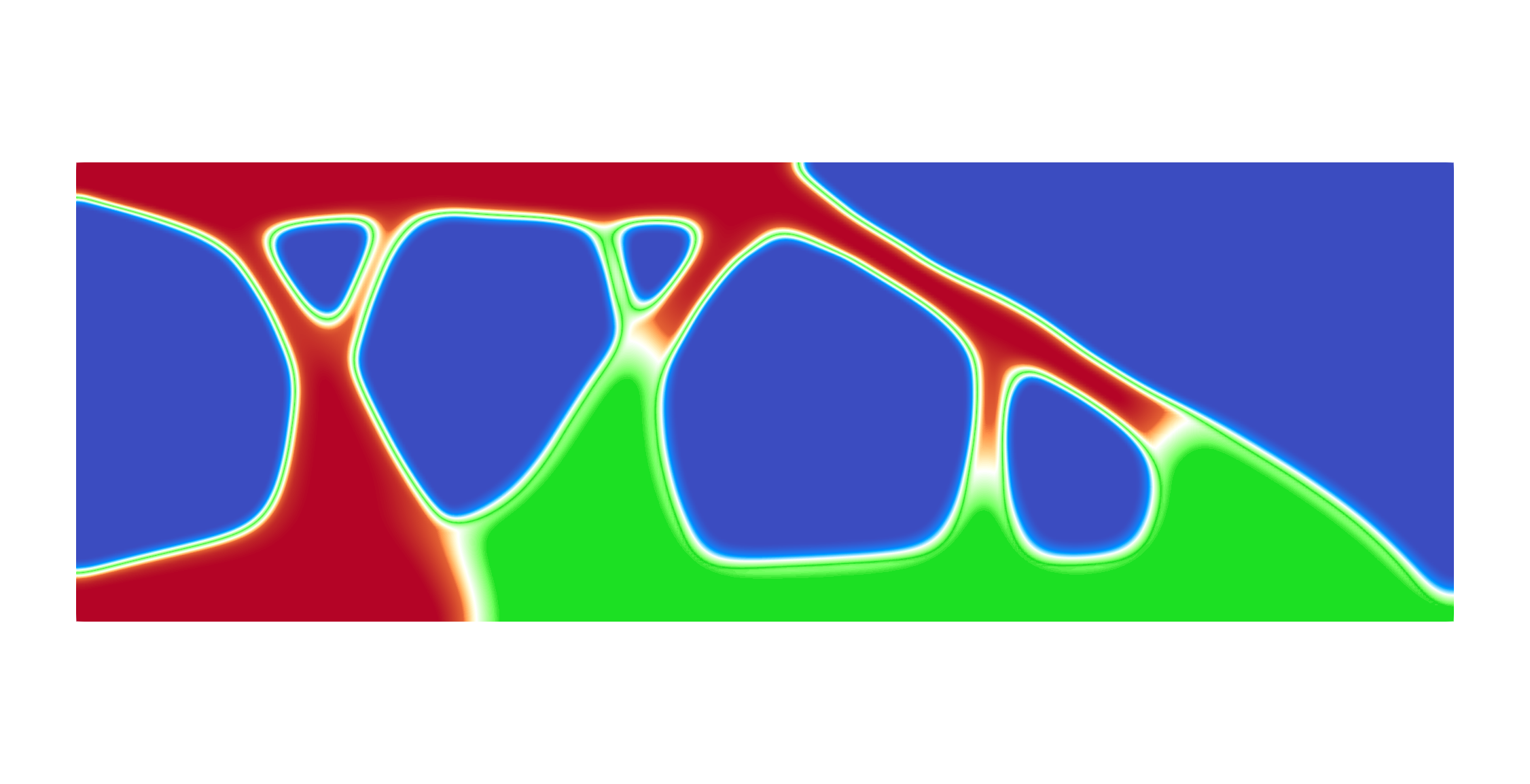}
    \caption{$\beta_1=90$}
    \label{fig:Beta1Multi90}
  \end{subfigure}
  \begin{subfigure}{0.32\textwidth}
    \centering \includegraphics[width=\linewidth,trim={0cm 6cm 0cm 5cm},clip]{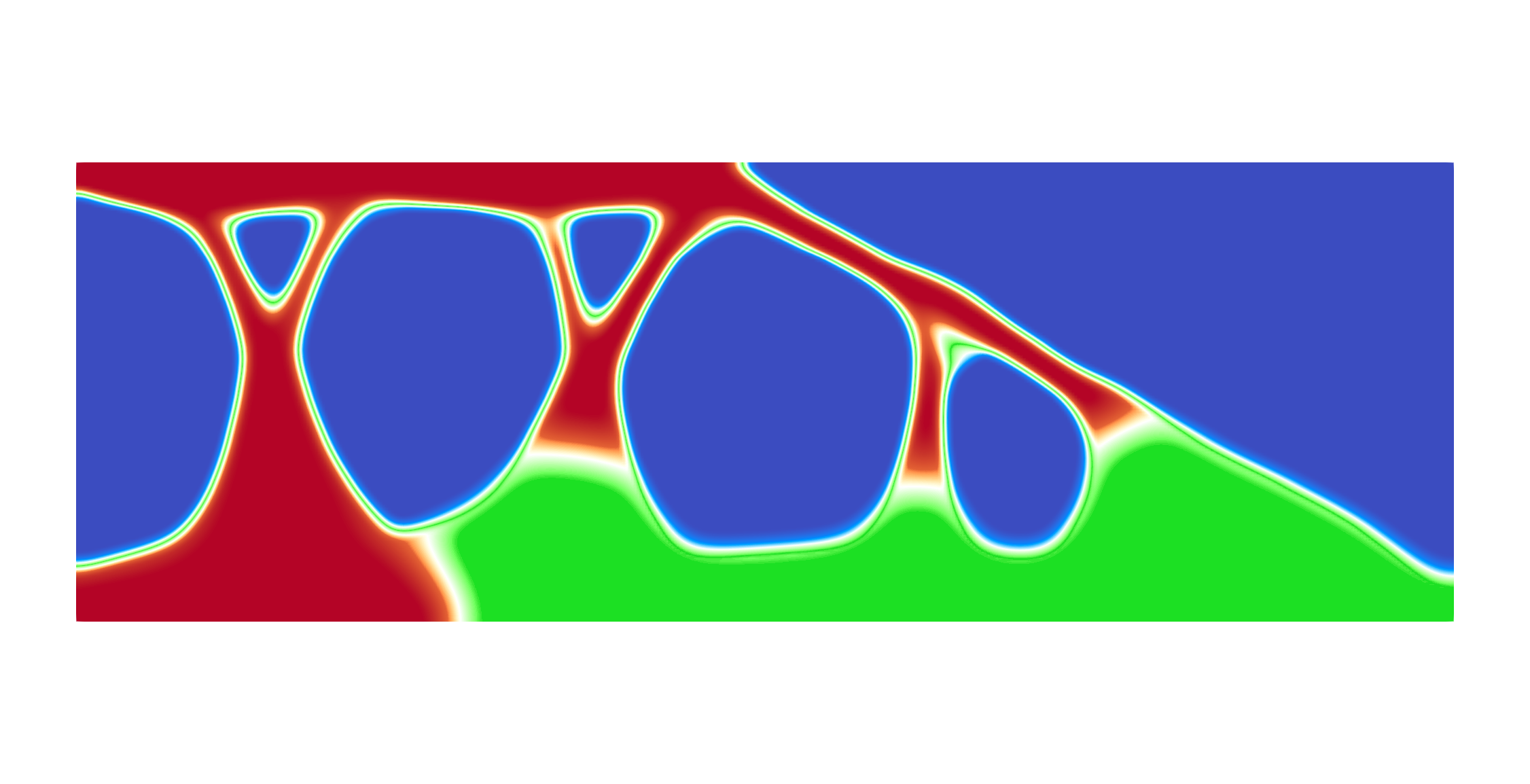}
    \caption{$\beta_1=360$}
    \label{fig:Beta1Multi360}
  \end{subfigure}

\caption{Multiphase test example \ref{Test2D3phase}
 with increasing $\beta_1$}
\label{fig:Beta1ParameterStudyMulti}
\end{figure}
\end{center}

\subsubsection{Allowing for removable support structures}
Often removable supports are employed in the construction phase, e.g. in \cite{AllaireBogoselSupports2}.
They provide stability during construction, but have to be taken off afterwards.
The multiphase approach allows the joint optimization of the structure and of the removable supports by including an extra phase variable for the removable support. 
As a test example, we consider a MBB beam setting, which features multiple
overhangs
when no AM effects are considered, which are reflected in a high self-weight $W$.
The design domain is $\Omega=[0,5]\times [0,1]$ and the structure is fixed at $\Gamma_D=[0,0.1]\times\{0\}
 \cup[4.9,5]\times\{0\}$,
an outer force $g=-e_d$ acts at
$\Gamma_N=[2.25,2.75]\times\{1\}$.
The gravitational force is again $f^c(\varphi,.) = -(1-\varphi_N) e_d$.
During construction, the Lamé parameters of the material and of the support coincide and are set to $32$.
When construction is completed, the Lamé parameters of the support and of the void match.
Moreover, we choose $\beta_1=20$ and $\beta_2=0.0002$.

In Table \ref{tab:SupportResults} and in Figure \ref{fig:Supports} the results for the structure without AM and with AM but without removable supports are listed, where the mass of the material amounts to $30\% $.
They are opposed in this table and figure to the structures obtained with removable supports, where first we allow no further but $5\% $ of the material to be used for the removable structures (i.e. $\mathfrak{m} = [0.25, 0.05, 0.7]$),
while in the second test additional support material is added with the amount of $8\% $ (i.e. $\mathfrak{m} = [0.3, 0.08, 0.62]$).

\begin{table}[th!]
\centering
\begin{tabular}{|c|c|c|c|c|c|c|c|}
\hline
\textbf{Method} & $\mathfrak{m}$ & $E$ & $F$ & $W$ \\
\hline
\hline
\textbf{No AM considered} & $[0.3,0.7]$ & 28.241766 & 0.155738 & 0.039798 \\
\hline
\textbf{No additional supports} & $[0.3,0.7]$ & 36.1056322 & 0.175231 & 0.002056 \\
\hline 
\textbf{Additional supports I} & $[0.25, 0.05 ,0.7]$ & 60.223846 & 0.192696 & 0.001887 \\
\hline
\textbf{Additional supports II} & $[0.3,0.08, 0.62]$ & 46.050218 & 0.161178 & 0.001961 \\
\hline
\end{tabular}
\caption{Results for the MBB beam test example with and without removable support}	
\label{tab:SupportResults}
\end{table}

\begin{center}
\begin{figure}[th!]
\centering
  \begin{subfigure}{0.48\textwidth}
    \centering \includegraphics[width=\linewidth,trim={0cm 12cm 0cm 12cm},clip]{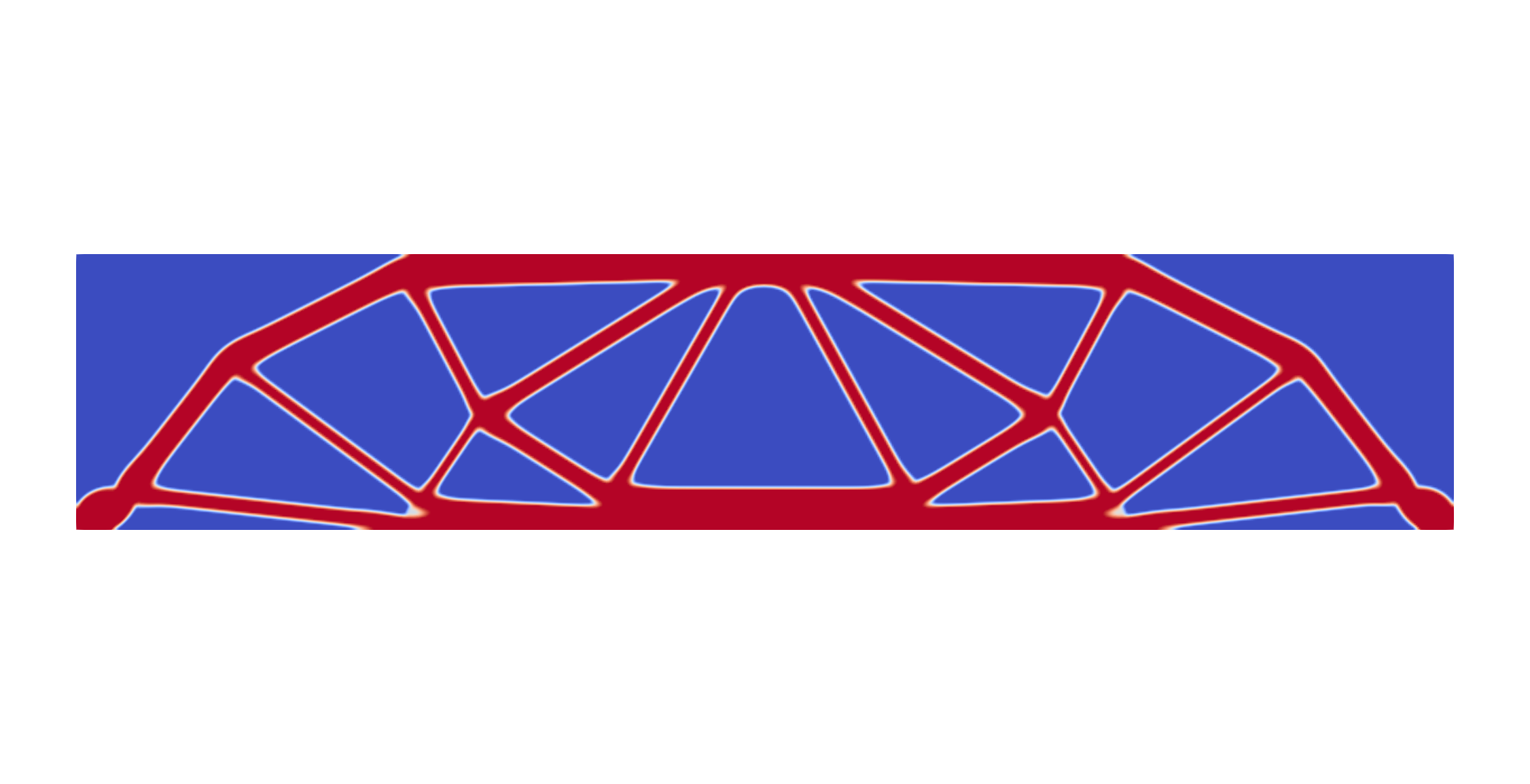}
    \caption{No AM}
    \label{fig:NoAM}
  \end{subfigure} 
 \centering
  \begin{subfigure}{0.48\textwidth}
    \centering \includegraphics[width=\linewidth,trim={0cm 12cm 0cm 12cm},clip]{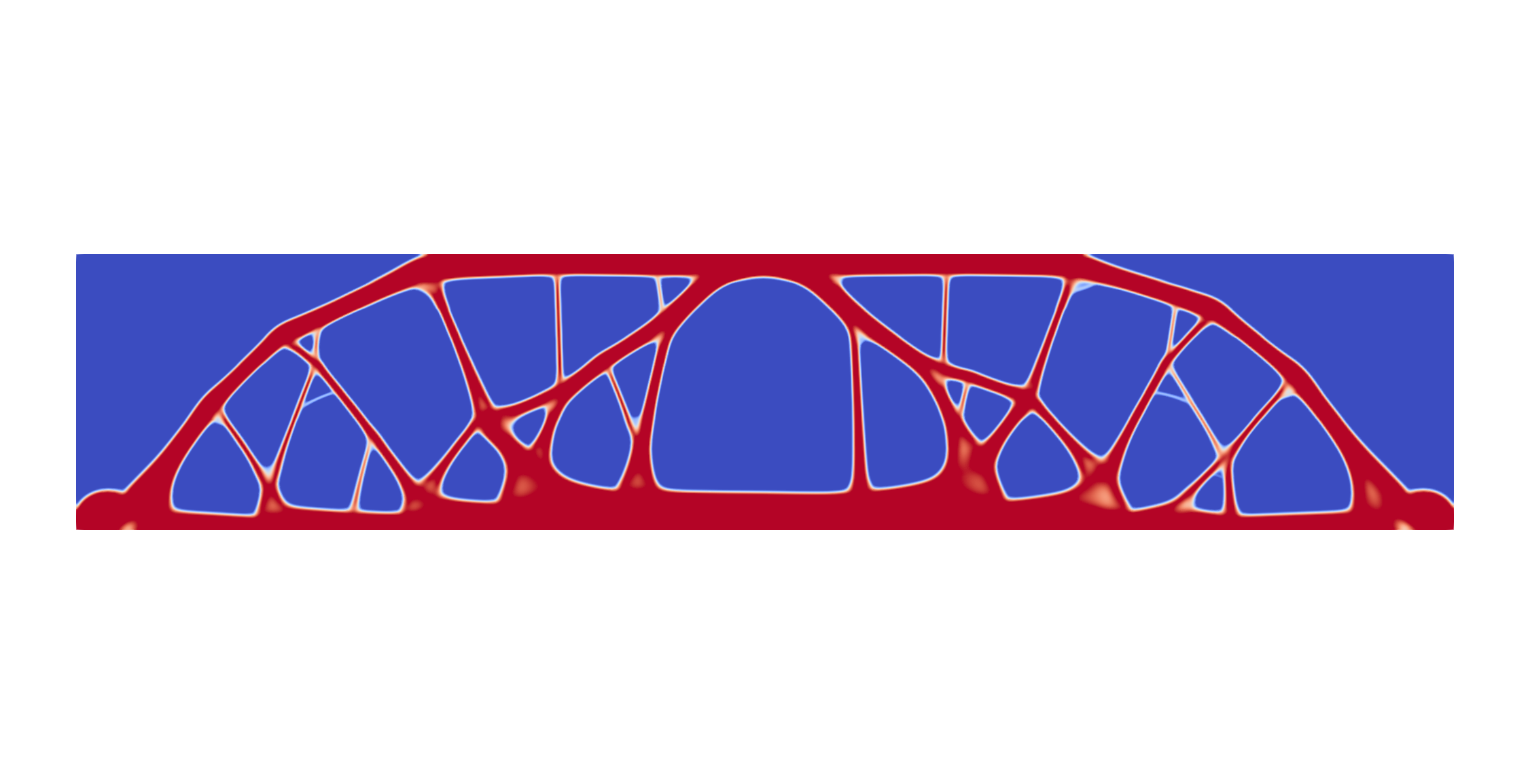}
    \caption{With AM}
    \label{fig:ScalarAM}
  \end{subfigure}
  
\centering
  \begin{subfigure}{0.48\textwidth}
    \centering \includegraphics[width=\linewidth,trim={0cm 12cm 0cm 12cm},clip]{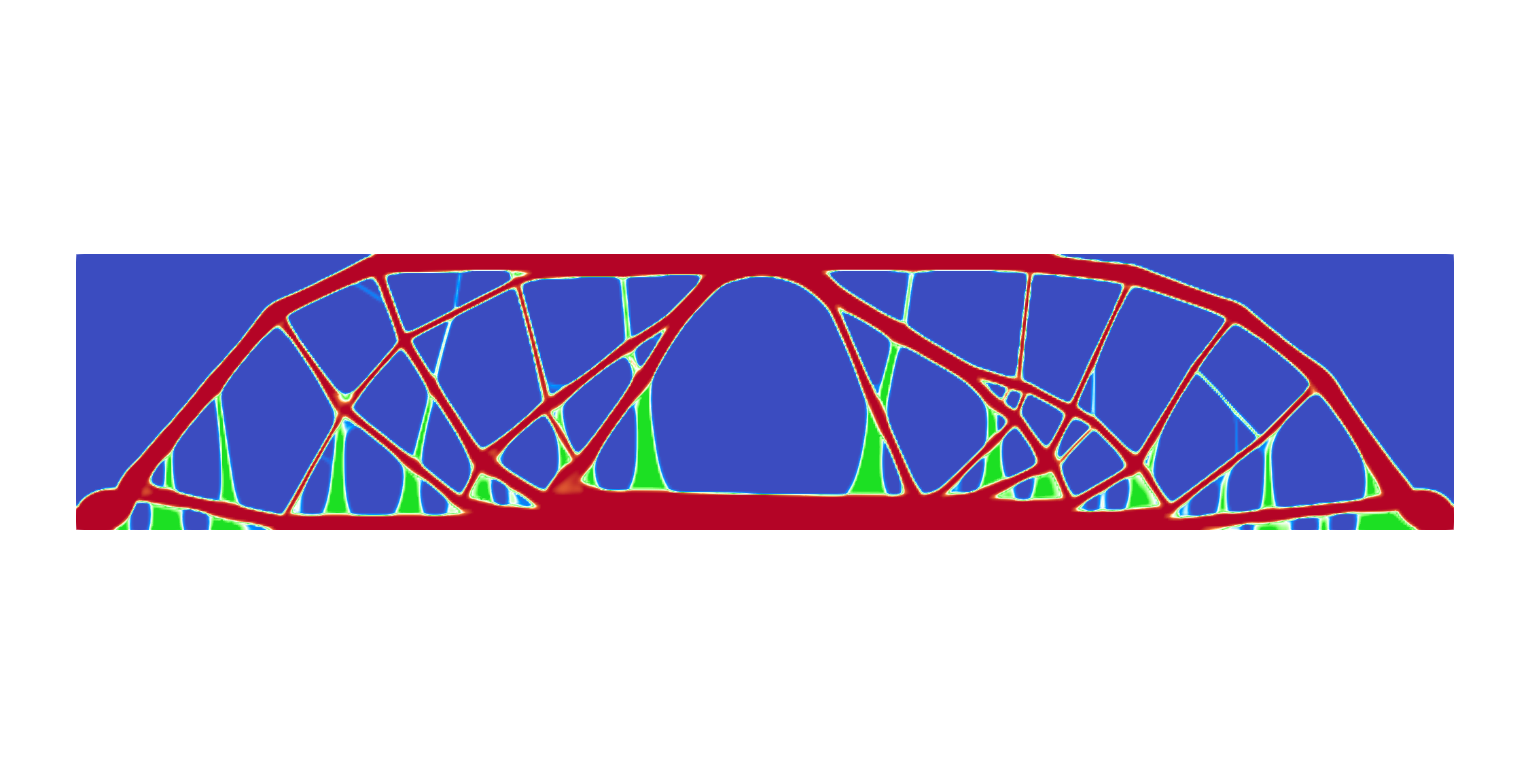}
    \caption{Removable supports included with mass constraint $\mathfrak{m}=[0.25,0.05,0.7]$}
    \label{fig:SupportsLowMass}
  \end{subfigure}  
\centering
  \begin{subfigure}{0.48\textwidth}
    \centering \includegraphics[width=\linewidth,trim={0cm 12cm 0cm 12cm},clip]{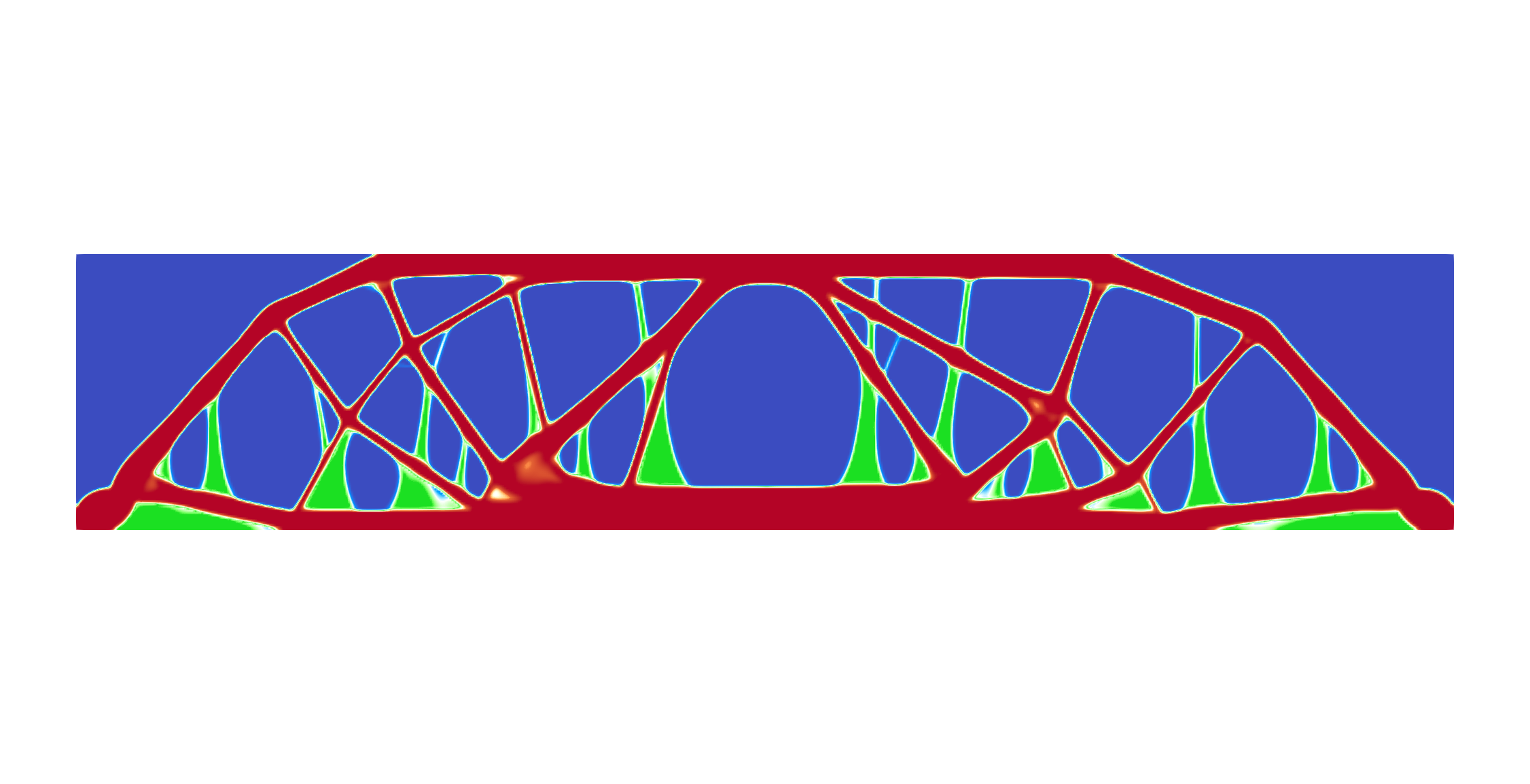}
    \caption{Removable supports included with mass constraint
    $\mathfrak{m}=[0.3,0.08,0.62]$}
    \label{fig:SupportsHighMass}
  \end{subfigure}

\caption{MBB beam test example with and without removable support structures}
\label{fig:Supports}
\end{figure}
\end{center}
In Figure \ref{fig:Supports}, the red phase corresponds to the main material, while green and blue represent the support phase and void, respectively.
The structure without AM consideration achieves the lowest value of $F$ among all examples, but shows poor constructability,
which is reflected in high $W$.
Including AM in the optimization problem,
$W$ improves by roughly $95\%$, while $F$ only increases moderately by $13\%$.
Visually, multiple vertical bars
are incorporated into the structure.
Using removable support structures with the same overall material mass shows similar low $W$, but degrades performance in $F$ by around $24\%$ compared to the result without AM.
The reason is most likely the decreased amount of remaining material.
Allowing additional $8\%$ support material almost restores the performance in $F$, while assuring constructability with the lowest $W$ of all examples.
However, then around $21\%$ of the material has to be disposed, additional labour is required, and the risk of damaging the structure while removing the supports remains.

\subsubsection{Example in 3D}
\label{3Dexample}
Finally, we present the impact of $W$ on a cantilever test example in three dimensions with two phases.
To this end, we set
$\Omega=[0,2]\times[0,1]\times[0,1]$, 
$\Gamma_N = [1.75,2]\times[0,1]\times\{0\}$,
$\Gamma_D=\{0\}\times[0,1]\times[0,1]$
and the mass is fixed to
$\tilde {\mathfrak{m}}=-0.6$.
Moreover, we increase the interface thickness to $\varepsilon=0.05$.
The remaining setting stays as in the 2D test example 
\ref{testEx}. Around 15h of CPU time for approximately 700 nested VMPT iterations were necessary per example, which is due to the $1~056~321$ employed spatial nodes
leading to an optimization problem with
$13~732~173$ DOFs.
\begin{center}
\begin{figure}[th!]
\centering
  \begin{subfigure}{0.24\textwidth}
    \centering \includegraphics[width=\linewidth,clip]{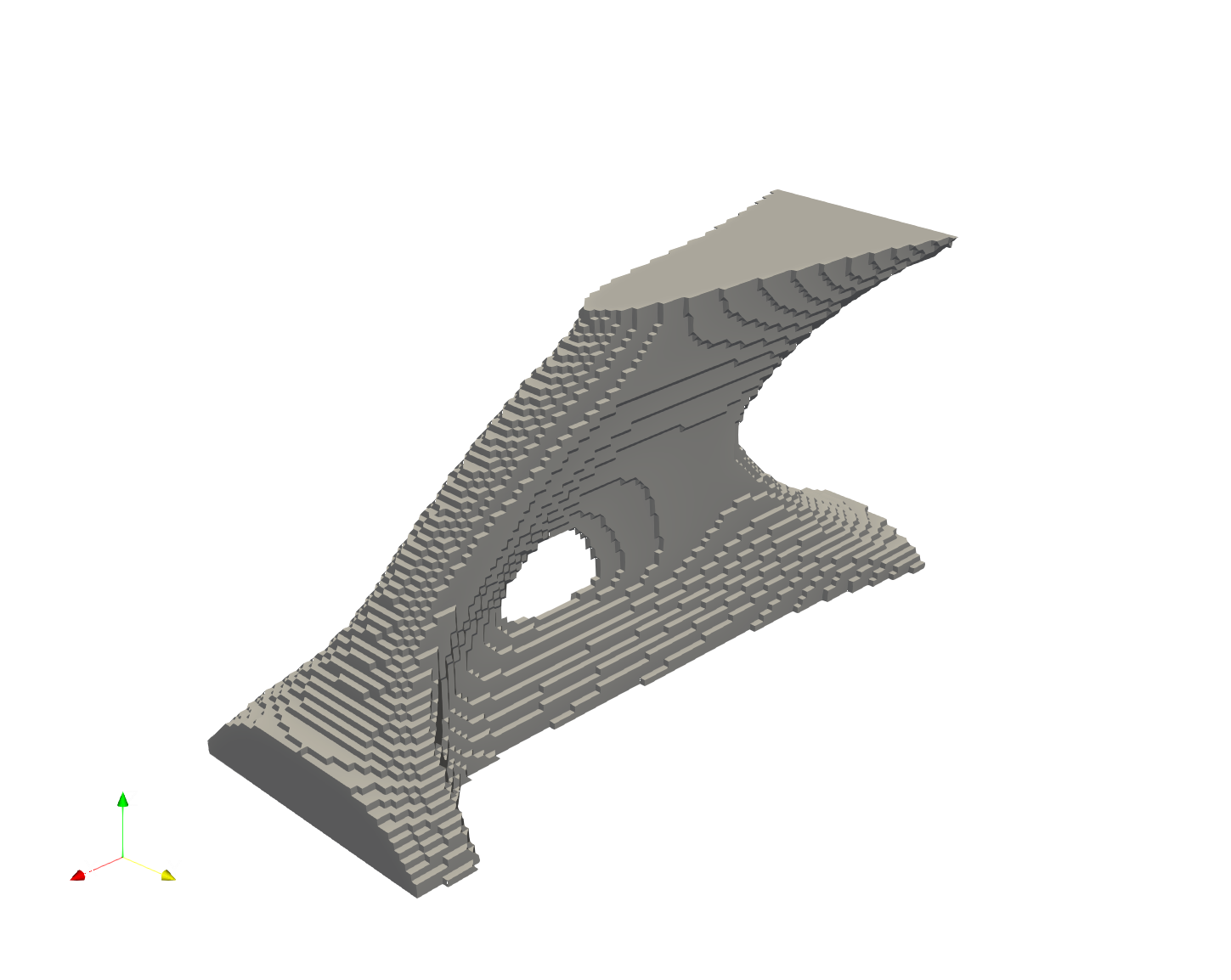}
    \caption{$\beta_1=0$, diagonal view}
    \label{fig:Diag0}
  \end{subfigure}
\centering
  \begin{subfigure}{0.24\textwidth}
    \centering \includegraphics[width=\linewidth,clip]{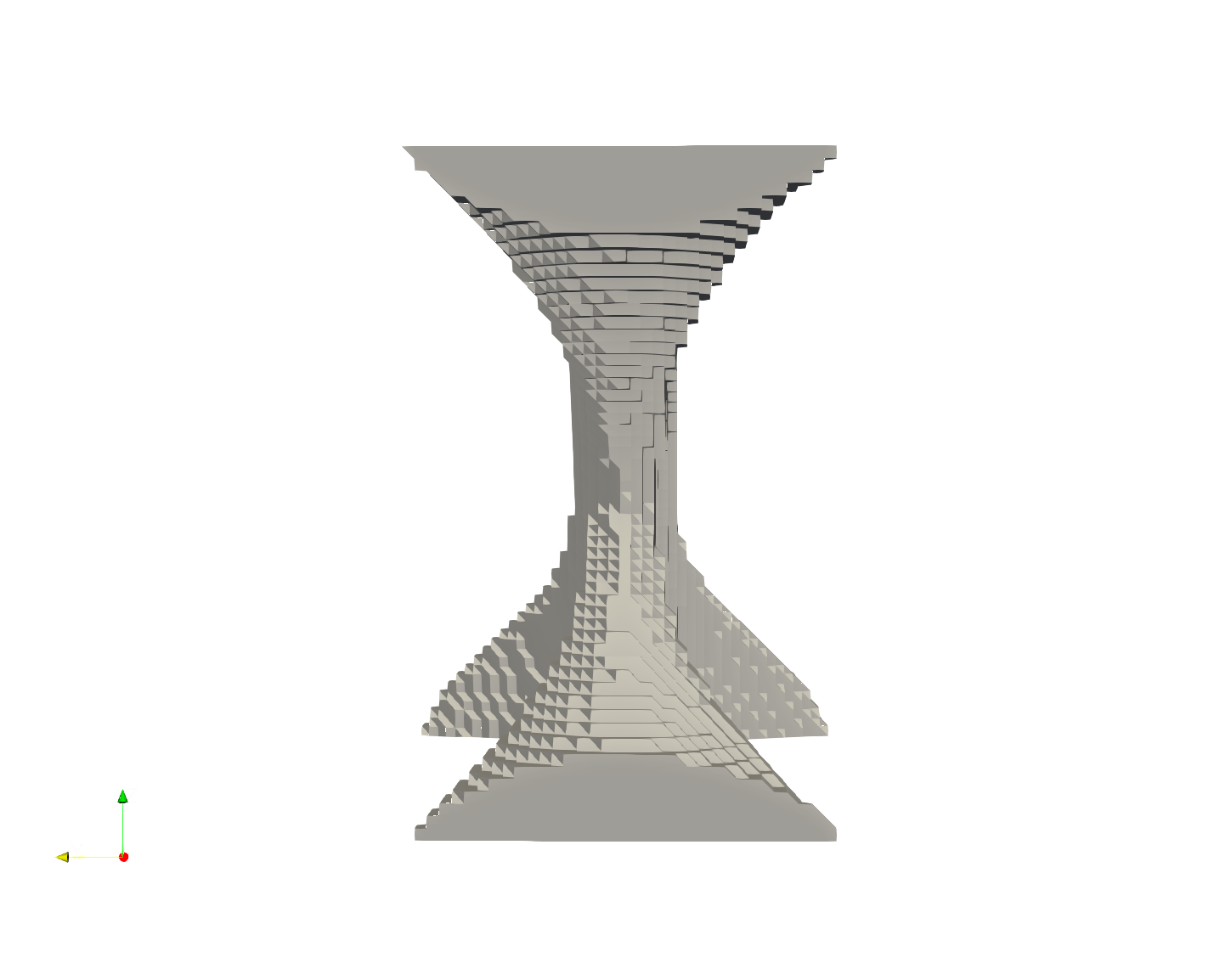}
    \caption{$\beta_1=0$, view from the back}
    \label{fig:Back0}
  \end{subfigure}
\centering
  \begin{subfigure}{0.24\textwidth}
    \centering \includegraphics[width=\linewidth,clip]{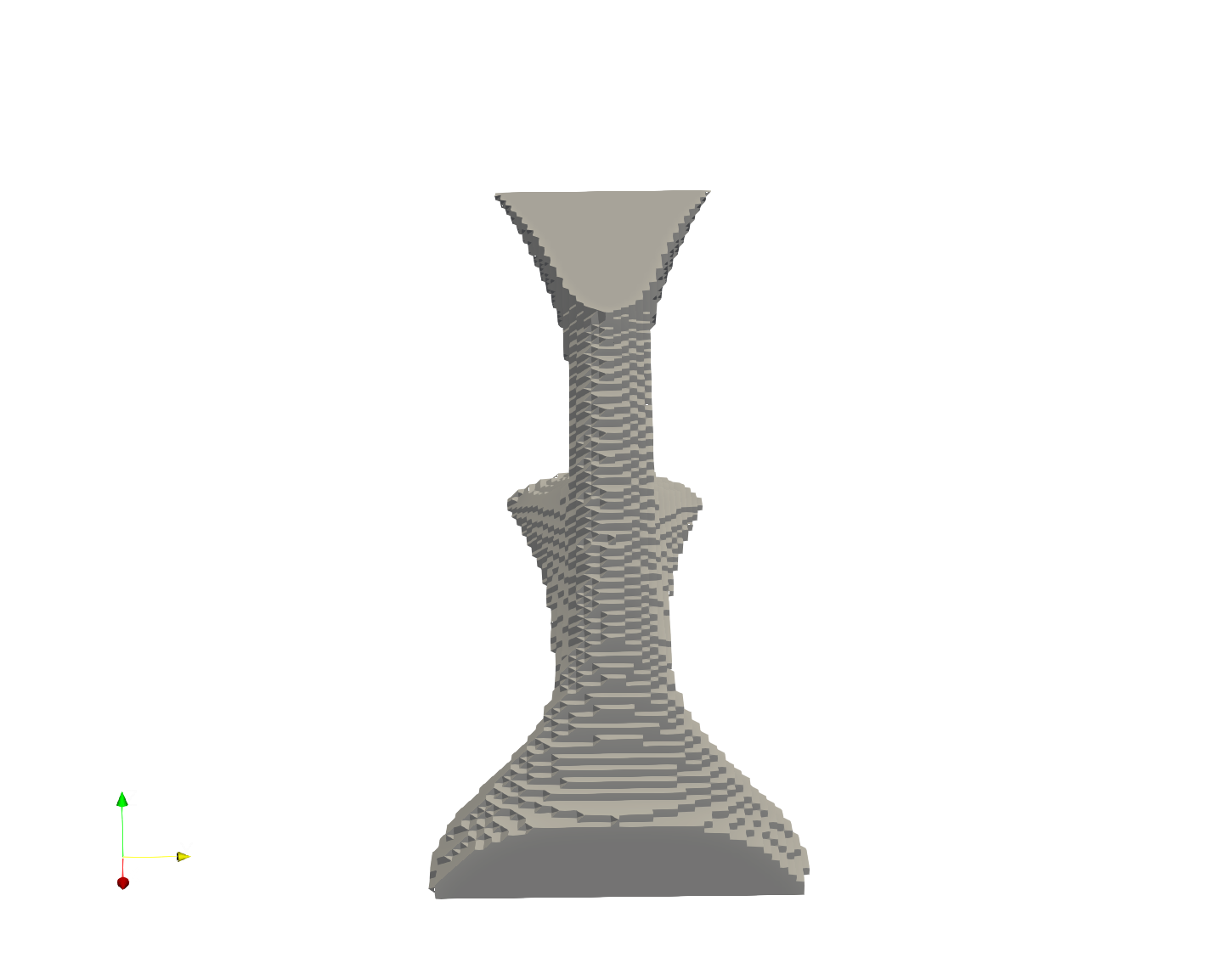}
    \caption{$\beta_1=0$, view from the front}
    \label{fig:Front0}
  \end{subfigure}
\centering
  \begin{subfigure}{0.24\textwidth}
    \centering \includegraphics[width=\linewidth,clip]{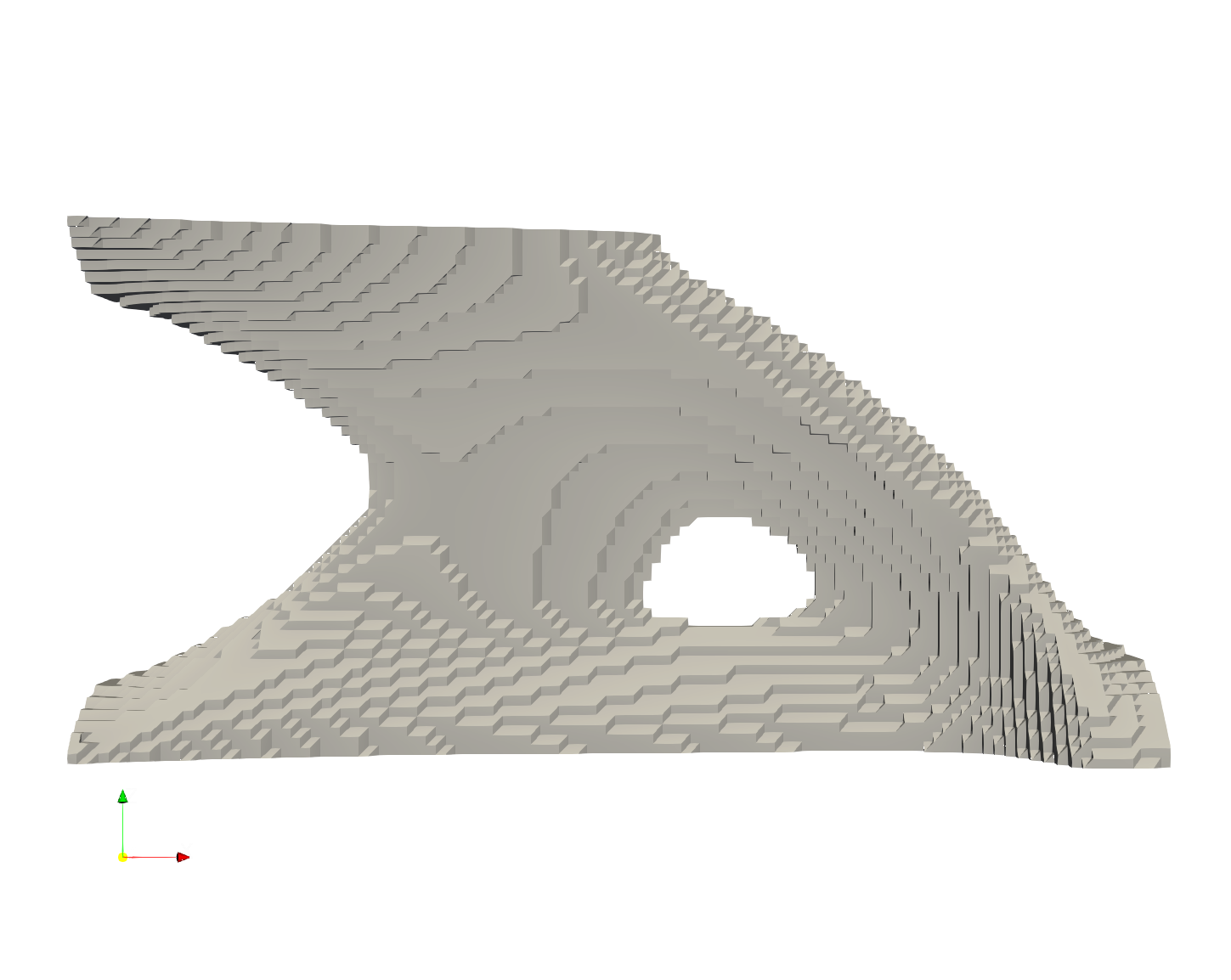}
    \caption{$\beta_1=0$, view from the side}
    \label{fig:Side0}
  \end{subfigure}

\centering
  \begin{subfigure}{0.24\textwidth}
    \centering \includegraphics[width=\linewidth,clip]{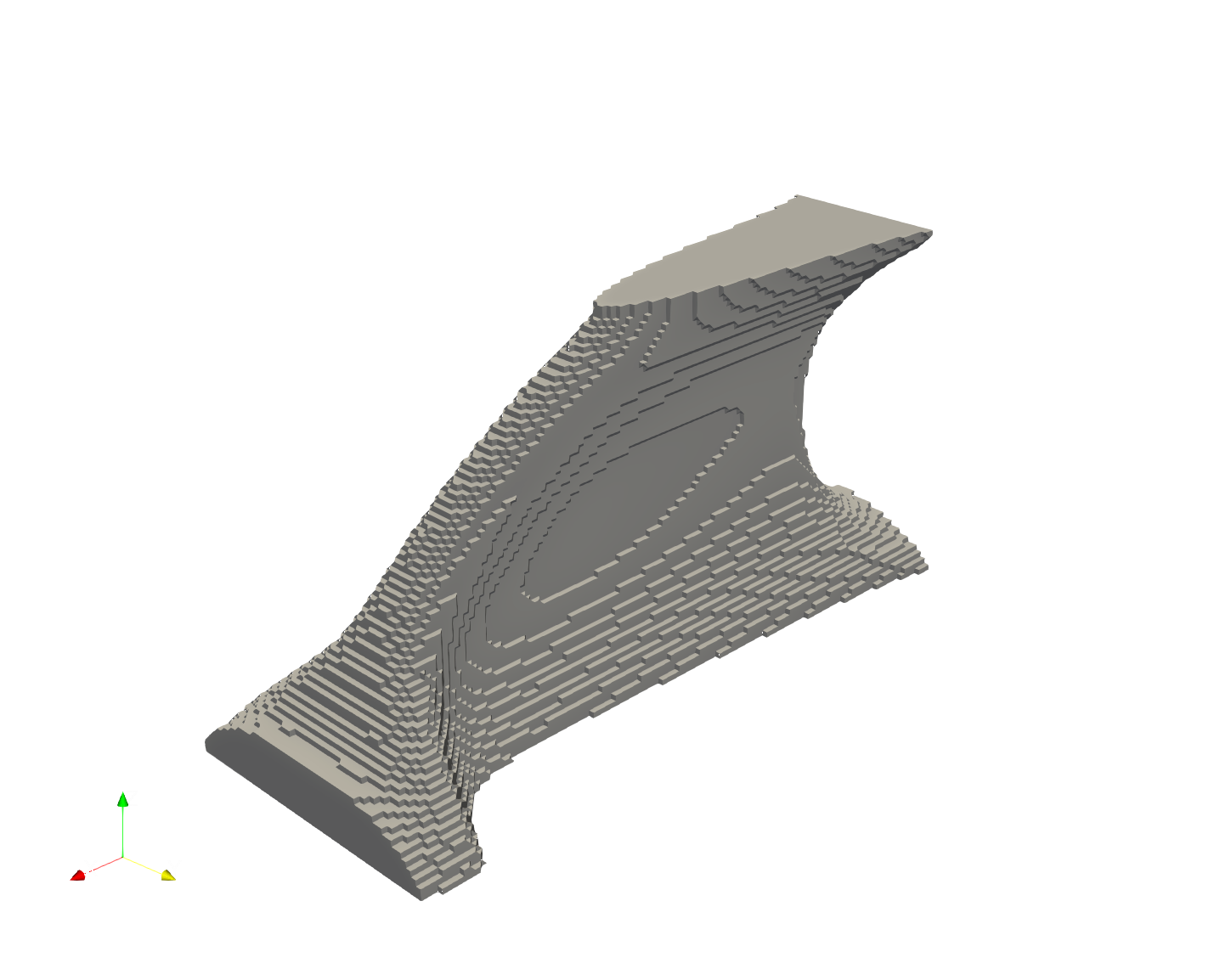}
    \caption{$\beta_1=48$, diagonal view}
    \label{fig:Diag48}
  \end{subfigure}
\centering
  \begin{subfigure}{0.24\textwidth}
    \centering \includegraphics[width=\linewidth,clip]{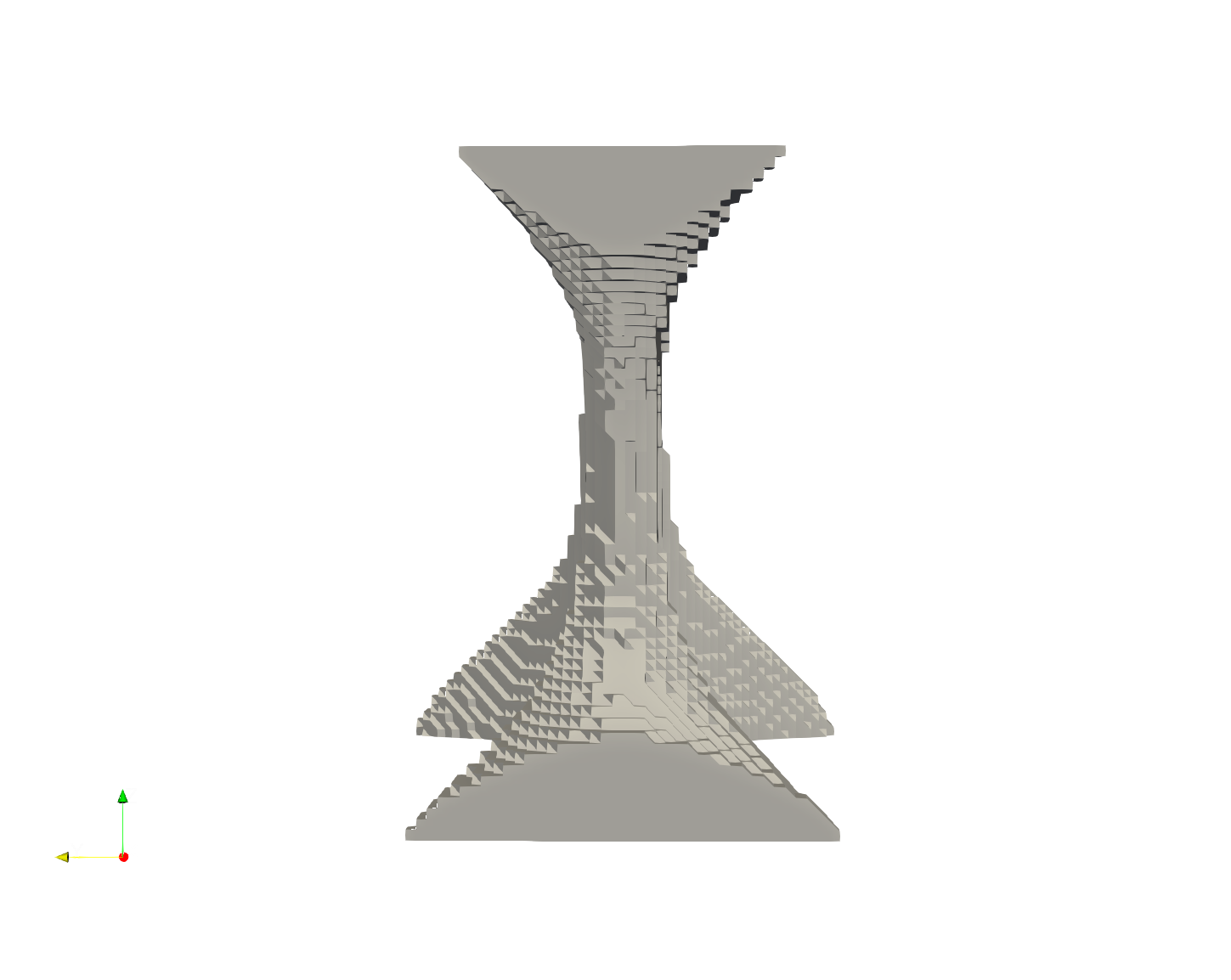}
    \caption{$\beta_1=48$, view from the back}
    \label{fig:Back48}
  \end{subfigure}
\centering
  \begin{subfigure}{0.24\textwidth}
    \centering \includegraphics[width=\linewidth,clip]{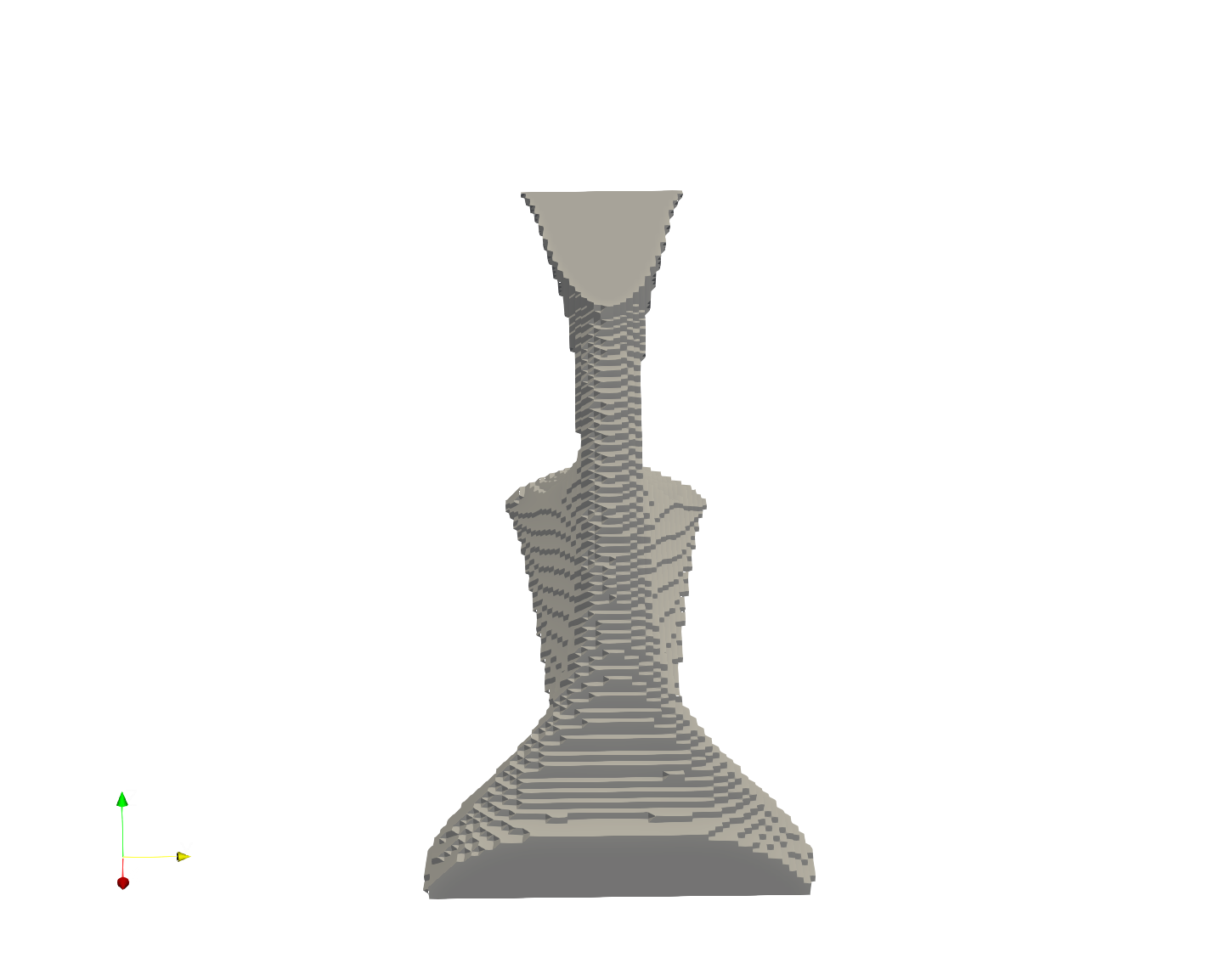}
    \caption{$\beta_1=48$, view from the front}
    \label{fig:Front48}
  \end{subfigure}
\centering
  \begin{subfigure}{0.24\textwidth}
    \centering \includegraphics[width=\linewidth,clip]{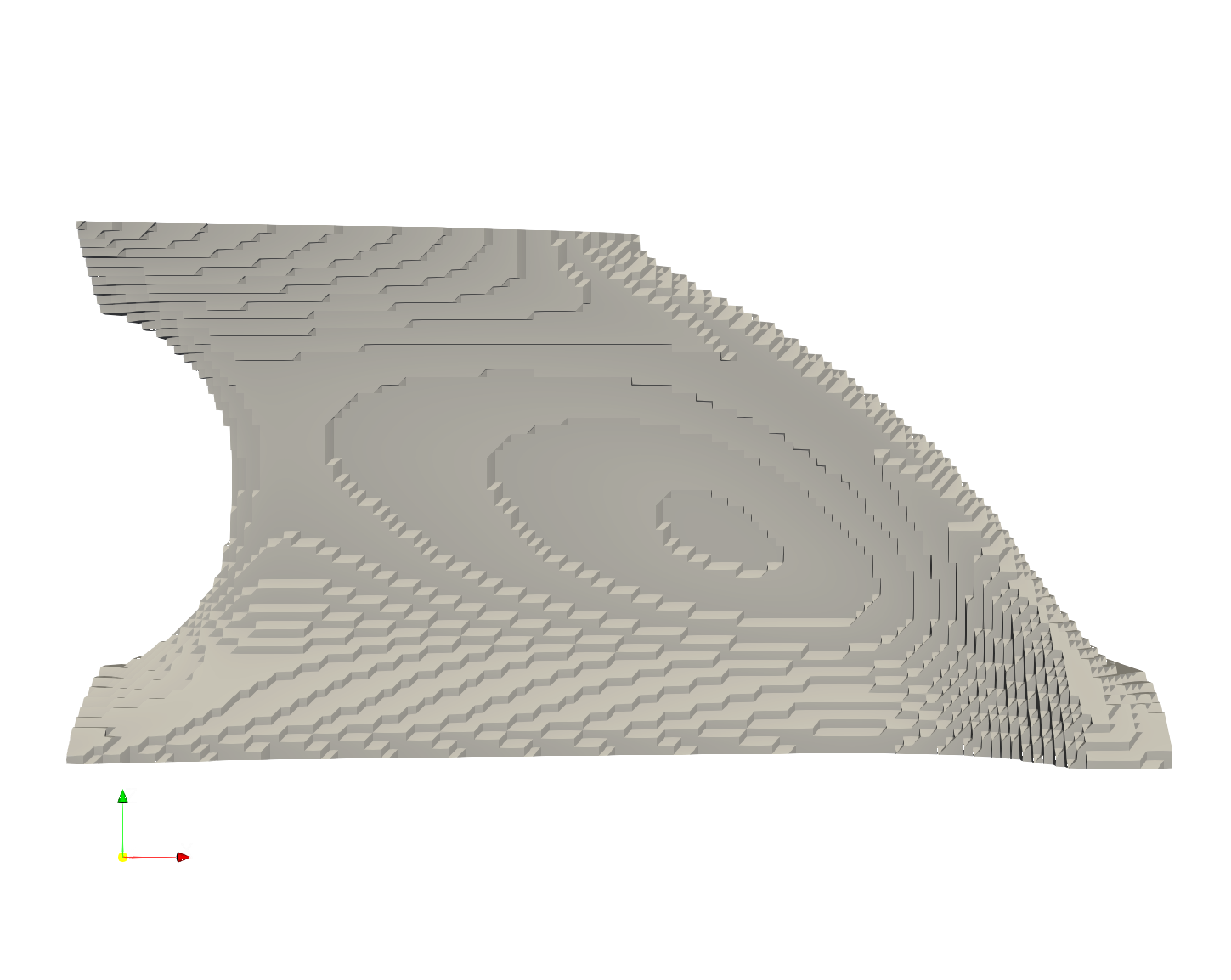}
    \caption{$\beta_1=48$, view from the side}
    \label{fig:Side48}
  \end{subfigure}
\caption{Zero level sets of the 3D cantilever examples \ref{3Dexample} with varying $\beta_1$}
\label{fig:3DSolutions}
\end{figure}
\end{center}
In Figure \ref{fig:3DSolutions}, the obtained structures without additive manufacturing
and with AM using $\beta_1=48$ are depicted. 
In the case of $\beta_1=0$, large
overhangs
in all directions emerge near $\Gamma_D$,
which are visible from the side
in Figure \ref{fig:Side0} and
from the back in Figure \ref{fig:Back0}.
These overhangs
and also the central hole vanish after increasing $\beta_1$.  
As already observed for the 2D example
\ref{testEx},
mass is relocated towards the building plate, which is evident in Figure \ref{fig:Front0} in comparison to Figure \ref{fig:Front48}.
Let us finally mention that the structures above feature no internal cavities.

%% file: KornLui.tex
For $\Gamma_B \subseteq \partial \Omega$ we define
the space
  $H^1_B(\Omega):= \{ y\in H^1(\Omega)\ |\ y|_{\Gamma_B} = 0\ a.e.
\}$.

\begin{lemma}[Poincar\'e inequality]  \label{PoinIE}
Given an open 
    $\Omega_B \subseteq \R^{d-1}$,
$ \Gamma_B=\Omega_B\times \{ 0 \}$ and
$\Omega_h=\Omega_B\times (0,h)$  for $0<h$. 
Then the  Poincar\'e inequality holds with the constant $h$, i.e. 
\begin{eqnarray} \label{eq:PoinIE}
 || y||_{L^2(\Omega_h)}
  \leq h
  || \nabla y||_{L^2(\Omega_h)}
\quad \forall y \in H^1_B(\Omega_h) \; .
\end{eqnarray}
\end{lemma}
\begin{proof}
The proof of the inequality is the same as for stripes, see e.g. in \cite{Alt2016}.\\
The inequality follows by using the fundamental theorem of calculus with $y=0$ on $\Gamma_B$ and the Cauchy-Schwarz-inequality: 
\begin{eqnarray*}
  | y(\hat x, x_d) |^2 &=&
  |\int_0^{x_d}   \partial_t y(\hat x, t) \text{ d}t|^2
  \\
  &\leq &x_d \int_0^{x_d}   | \partial_t y(\hat x, t) |^2 \text{ d}t
  \leq h \int_0^h   | \partial_t y(\hat x, t) |^2 \text{ d}t\\
  &\leq& h \int_0^h  | \nabla y(\hat x, t) |^2 \text{ d}t
  \quad \mbox{ for allmost all } (\hat x, x_d) \in \Omega_h \; .
  \\
  \Rightarrow
  || y||_{L^2(\Omega_h)}^2&\leq&
                                 h \int_{\Omega_B}  \int_0^h   \int_0^h  | \nabla y(\hat x, t) |^2 \text{ d}t   \text{ d}x_d \text{ d}\hat x \\
&=&h^2 || \nabla y ||_{L^2(\Omega_h)}^2 \; .
\end{eqnarray*}
\end{proof}

\begin{theorem}[Korn's inequality]
  \label{kornIE}
  Given 
$\Omega_B=(0,l_1)\times \ldots \times (0,l_{d-1})$ with $l_i>0$, $ \Gamma_B=\Omega_B\times \{ 0 \}$, $\Omega_h=\Omega_B\times (0,h)$  for $0<h\leq H$ and $\Omega:=\Omega_H$, then the Korn inequality holds with a constant independent of $h$, i.e.
\begin{eqnarray}
  \int_{\Omega_h} |\nabla u(x)|^2 \, \dx
     \leq
  K_{\Omega,\Gamma_B } 2^ {d-1}\int_{\Omega_{h}} \lvert \mathcal{E}(u(x)) \rvert^2 \, \dx
\quad \forall u\in H^1_B(\Omega_h)^d .
  \label{kornNew}
\end{eqnarray}
\end{theorem}
\begin{proof}
  Let
  $V := \tfrac hH \Omega
  =  \tfrac hH \Omega_B \times (0,h) $
  and
  $ G:= \left\lfloor \frac{H}{h} \right\rfloor
  \tfrac hH \bar \Omega_B \times [0,h] 
  \subseteq \bar \Omega_h$.
  By shifting $G$ in the $2^{d-1}$ corners  of
  $\Gamma_B$ with vectors
  $b_1,\ldots b_{2^{d-1}}\in   \Gamma_B$
  one can cover  $\Omega_h$.
 The set $G$ itself can be exactly covered by shifting $\bar V$
 with
  $k := \left\lfloor \frac{H}{h} \right\rfloor^{d-1}$
  vectors
  $v_i \in \Gamma_B$, such that the sets
  $v_i +V$ are disjoint. 
  Hence, with $v_{j,i}:= b_j+v_i$ it holds
$$
\overline{\Omega}_h
= \bigcup_{j=1}^{2^{d-1}} { ( G+b_j) }
= \bigcup_{j=1}^{2^{d-1}} \overline{ (\dot{\bigcup}_{i=1}^{k} \left( v_{j,i} + \frac{h}{H} \Omega \right) )}
.
$$
Here \(\dot{\bigcup}\) denotes the disjoint union.
For $y \in \Gamma_B$ it holds
$v_{j,i} + \frac{h}{H} y\in \Gamma_B$ and hence
$\tilde{u}_{j,i} \in H^1_B(\Omega)^d $ for
$
\tilde{u}_{j,i}(y) :=u\left(v_{j,i} + \frac{h}{H} y\right)$ given $u\in H^1_B(\Omega_h)^d $. Then
\begin{eqnarray}
\int_{\Omega_h} |\nabla u(x)|^2 \, \dx
  &\leq& \sum_{j=1}^{2^{d-1}} \sum_{i=1}^{k} \int_{v_{j,i} + \frac{h}{H} \Omega} |\nabla u(x)|^2 \, \dx
\nonumber \\  &=&
\sum_{j=1}^{2^{d-1}} \sum_{i=1}^{k} \left( \frac{h}{H} \right)^{d-2} \int_{\Omega} |\nabla \tilde{u}_{j,i}(y)|^2 \, \dy 
\nonumber  \\
  &\leq& \sum_{j=1}^{2^{d-1}} \sum_{i=1}^{k} \left( \frac{h}{H} \right)^{d-2} K_{\Omega,\Gamma_B } \int_{\Omega} \lvert \mathcal{E}(\tilde{u}_{j,i}(y) \rvert^2 \, \dy
\label{korn}\\
&=& K_{\Omega,\Gamma_B }\sum_{j=1}^{2^{d-1}}  \sum_{i=1}^{k}  \int_{v_{j,i} + \frac{h}{H}\Omega} \lvert \mathcal{E}(u(x)) \rvert^2 \, \dx
\nonumber \\  &\leq &
K_{\Omega,\Gamma_B } 2^ {d-1}\int_{\Omega_{h}} \lvert \mathcal{E}(u(x)) \rvert^2 \, \dx
\label{inclus}
\end{eqnarray}
where \eqref{korn} follows using Korn's inequality for functions in $H^1_B(\Omega)^d $ with a constant $K_{\Omega,\Gamma_B }$ independent of $i$ and $j$ and \eqref{inclus} follows with $\dot\bigcup_{i=1}^{k} \left( v_{j, i} + \frac{h}{H} \Omega \right )  \subseteq \Omega_{h}$.
\end{proof}

%% file: literature.bib
@book{Alt2016,
    AUTHOR = {Alt, Hans Wilhelm},
     TITLE = {Linear Functional Analysis:
      An Application-Oriented Introduction},
 PUBLISHER = {\\Springer, London},
      YEAR = {2016},
      ISBN = {9781447172796},
   MRCLASS = {46-01},
  MRNUMBER = {3497775},
       DOI = {https://doi.org/10.1007/978-1-4471-7280-2}
}

@article{CahnHilliard,
  title={Free energy of a nonuniform system. {I}. {I}nterfacial free energy},
  author={Cahn, John W and Hilliard, John E},
  journal={J. Chem. Phys.},
  volume={28},
  number={2},
  pages={258--267},
  year={1958},
  publisher={American Institute of Physics}
}

@article{AllenCahn,
title = {A microscopic theory for antiphase boundary motion and its application to antiphase domain coarsening},
journal = {Acta Metall.},
volume = {27},
number = {6},
pages = {1085-1095},
year = {1979},
issn = {0001-6160},
doi = {https://doi.org/10.1016/0001-6160(79)90196-2},
author = {Samuel M. Allen and John W. Cahn}
}

@book{Fenics,
  title={Automated Solution of Differential Equations by the Finite Element Method: The FEniCS Book},
  author={Logg, Anders and Mardal, Kent-Andre and Wells, Garth},
  volume={84},
  year={2012},
  publisher={Springer, Heidelberg},
  editors = {Anders Logg, Kent-Andre Mardal, Garth Wells}
}

@article{VMPT,
    AUTHOR = {Blank, Luise and Rupprecht, Christoph},
     TITLE = {An extension of the projected gradient method to a {B}anach
              space setting with application in structural topology
              optimization},
   JOURNAL = {SIAM J. Control Optim.},
  FJOURNAL = {SIAM Journal on Control and Optimization},
    VOLUME = {55},
      YEAR = {2017},
    NUMBER = {3},
     PAGES = {1481--1499},
       DOI = {https://doi.org/10.1137/16M1092301}
}

@article {BGSS13b,
    AUTHOR = {Blank, Luise and Garcke, Harald and Sarbu, Lavinia and Styles,
              Vanessa},
     TITLE = {Nonlocal {A}llen-{C}ahn systems: {A}nalysis and a primal-dual
              active set method},
   JOURNAL = {IMA J. Numer. Anal.},
    VOLUME = {33},
      YEAR = {2013},
    NUMBER = {4},
     PAGES = {1126--1155},
      ISSN = {0272-4979,1464-3642},
MRREVIEWER = {Daniel\ Wachsmuth},
       DOI = {https://doi.org/10.1093/imanum/drs039}
}

@article{KelleySachs1992,
    AUTHOR = {Kelley, C. T. and Sachs, E. W.},
     TITLE = {Mesh independence of the gradient projection method for
              optimal control problems},
   JOURNAL = {SIAM J. Control Optim.},
  FJOURNAL = {SIAM Journal on Control and Optimization},
    VOLUME = {30},
      YEAR = {1992},
    NUMBER = {2},
     PAGES = {477--493},
      ISSN = {0363-0129},
   MRCLASS = {49M07 (49M25 65K10)},
  MRNUMBER = {1149080},
MRREVIEWER = {Michael\ E.\ Fisher},
       DOI = {https://doi.org/10.1137/0330029}
}

@article{BFGS14,
    AUTHOR = {Blank, Luise and Garcke, Harald and Farshbaf-Shaker, M. Hassan
              and Styles, Vanessa},
     TITLE = {Relating phase field and sharp interface approaches to
              structural topology optimization},
   JOURNAL = {ESAIM Control Optim. Calc. Var.},
  FJOURNAL = {ESAIM. Control, Optimisation and Calculus of Variations},
    VOLUME = {20},
      YEAR = {2014},
    NUMBER = {4},
     PAGES = {1025--1058},
MRREVIEWER = {Vladimir\ V.\ Kobelev},
       DOI = {https://doi.org/10.1051/cocv/2014006}
}

@misc{DissRupprecht,
            year = {2016},
           title = {Projection type methods in {B}anach space with application in topology optimization},
          author = {Christoph Rupprecht},
            note = {Ph.D. thesis},
           month = {4},
             url = {https://epub.uni-regensburg.de/33715/}
}

@article{BourdinChambolle,
    AUTHOR = {Bourdin, Blaise and Chambolle, Antonin},
     TITLE = {Design-dependent loads in topology optimization},
   JOURNAL = {ESAIM Control Optim. Calc. Var.},
    VOLUME = {9},
      YEAR = {2003},
     PAGES = {19--48},
      ISSN = {1292-8119,1262-3377},
   MRCLASS = {49Q20 (49Q10 74G65 74P05 74P15)},
  MRNUMBER = {1957089},
MRREVIEWER = {Gong\ Qing\ Zhang},
       DOI = {https://doi.org/10.1051/cocv:2002070}
}

@InProceedings{BC06,
author="Bourdin, Blaise
and Chambolle, Antonin",
editor="Bends{\o}e, Martin Philip
and Olhoff, Niels
and Sigmund, Ole",
title="The phase-field method in optimal design",
booktitle="IUTAM Symposium on Topological Design Optimization of Structures, Machines and Materials",
year="2006",
publisher="Springer Netherlands, Dordrecht",
pages="207--215"
}

@article{WR12,
    AUTHOR = {Ristinmaa, Matti and Wallin, Mathias},
     TITLE = {Howard's algorithm in a phase-field topology optimization
              approach},
   JOURNAL = {Internat. J. Numer. Methods Engrg.},
  FJOURNAL = {International Journal for Numerical Methods in Engineering},
    VOLUME = {94},
      YEAR = {2013},
    NUMBER = {1},
     PAGES = {43--59},
      ISSN = {0029-5981,1097-0207},
   MRCLASS = {74P15 (49Q12)},
  MRNUMBER = {3040512},
       DOI = {https://doi.org/10.1002/nme.4434}
}

@article{PRW12,
    AUTHOR = {Penzler, Patrick and Rumpf, Martin and Wirth, Benedikt},
     TITLE = {A phase-field model for compliance shape optimization in
              nonlinear elasticity},
   JOURNAL = {ESAIM Control Optim. Calc. Var.},
    VOLUME = {18},
      YEAR = {2012},
    NUMBER = {1},
     PAGES = {229--258},
      ISSN = {1292-8119,1262-3377},
   MRCLASS = {49Q10 (74B20 74P05)},
  MRNUMBER = {2887934},
MRREVIEWER = {Shiah-Sen\ Wang},
       DOI = {https://doi.org/10.1051/cocv/2010045}
}

@incollection{BFGRS14,
    AUTHOR = {Blank, Luise and Farshbaf-Shaker, M. Hassan and Garcke, Harald and Rupprecht, Christoph and Styles, Vanessa},
     TITLE = {Multi-material phase field approach to structural topology optimization},
 BOOKTITLE = {Trends in {PDE} constrained optimization},
    SERIES = {Internat. Ser. Numer. Math.},
    VOLUME = {165},
     PAGES = {231--246},
 PUBLISHER = {Birkh\"auser/Springer, Cham},
      YEAR = {2014},
   MRCLASS = {49Q10 (65K15 74P05 74P15)},
  MRNUMBER = {3328979},
       DOI = {https://doi.org/10.1007/978-3-319-05083-6_15}
}

@article{HIK02,
    AUTHOR = {Hinterm\"uller, M. and Ito, K. and Kunisch, K.},
     TITLE = {The primal-dual active set strategy as a semismooth {N}ewton
              method},
   JOURNAL = {SIAM J. Optim.},
  FJOURNAL = {SIAM Journal on Optimization},
    VOLUME = {13},
      YEAR = {2002},
    NUMBER = {3},
     PAGES = {865--888},
MRREVIEWER = {Hou\ Duo\ Qi},
       DOI = {https://doi.org/10.1137/S1052623401383558},
}

@article{AdaptiveMeshStructOptim,
    AUTHOR = {Jin, Bangti and Li, Jing and Xu, Yifeng and Zhu, Shengfeng},
     TITLE = {An adaptive phase-field method for structural topology
              optimization},
   JOURNAL = {J. Comput. Phys.},
  FJOURNAL = {Journal of Computational Physics},
    VOLUME = {506},
      YEAR = {2024},
     PAGES = {Paper No. 112932}, 
MRREVIEWER = {Fajie\ Wang},
       DOI = {https://doi.org/10.1016/j.jcp.2024.112932}
}

@article{BSS12,
    AUTHOR = {Blank, Luise and Sarbu, Lavinia and Stoll, Martin},
     TITLE = {Preconditioning for {A}llen-{C}ahn variational inequalities
              with non-local constraints},
   JOURNAL = {J. Comput. Phys.},
  FJOURNAL = {Journal of Computational Physics},
    VOLUME = {231},
      YEAR = {2012},
    NUMBER = {16},
     PAGES = {5406--5420},
MRREVIEWER = {Messaoud\ Boulbrachene},
       DOI = {https://doi.org/10.1016/j.jcp.2012.04.035},
}

@article{BurgerStainko,
    AUTHOR = {Burger, Martin and Stainko, Roman},
     TITLE = {Phase-field relaxation of topology optimization with local
              stress constraints},
   JOURNAL = {SIAM J. Control Optim.},
    VOLUME = {45},
      YEAR = {2006},
    NUMBER = {4},
     PAGES = {1447--1466},
      ISSN = {0363-0129,1095-7138},
   MRCLASS = {74P15 (74P05 74P10 74S05 90C51)},
  MRNUMBER = {2257229},
MRREVIEWER = {Michal\ Ko\v cvara},
       DOI = {https://doi.org/10.1137/05062723X}
}

@article{Dede2012,
    AUTHOR = {Ded\`e, Luca and Borden, Micheal J. and Hughes, Thomas J. R.},
     TITLE = {Isogeometric analysis for topology optimization with a phase
              field model},
   JOURNAL = {Arch. Comput. Methods Eng.},
  FJOURNAL = {Archives of Computational Methods in Engineering. State of the
              Art Reviews},
    VOLUME = {19},
      YEAR = {2012},
    NUMBER = {3},
     PAGES = {427--465},
       DOI = {https://doi.org/10.1007/s11831-012-9075-z}
}

@article{ABCHRR,
    AUTHOR = {Auricchio, Ferdinando and Bonetti, Elena and Carraturo,
              Massimo and H\"omberg, Dietmar and Reali, Alessandro and
              Rocca, Elisabetta},
     TITLE = {A phase-field-based graded-material topology optimization with
              stress constraint},
   JOURNAL = {Math. Models Methods Appl. Sci.},
  FJOURNAL = {Mathematical Models and Methods in Applied Sciences},
    VOLUME = {30},
      YEAR = {2020},
    NUMBER = {8},
     PAGES = {1461--1483},
      ISSN = {0218-2025,1793-6314},
   MRCLASS = {74P05 (35Q74 49K20 49M05 74B99 74P15)},
  MRNUMBER = {4144362},
MRREVIEWER = {Tomasz\ Lekszycki},
       DOI = {https://doi.org/10.1142/S0218202520500281}
}

@article{gaynor2016topology,
    AUTHOR = {Gaynor, Andrew T. and Guest, James K.},
     TITLE = {Topology optimization considering overhang constraints:
              {E}liminating sacrificial support material in additive
              manufacturing through design},
   JOURNAL = {Struct. Multidiscip. Optim.},
  FJOURNAL = {Struct. Multidiscip. Optim.},
    VOLUME = {54},
      YEAR = {2016},
    NUMBER = {5},
     PAGES = {1157--1172},
      ISSN = {1615-147X,1615-1488},
   MRCLASS = {74P15},
  MRNUMBER = {3571166},
       DOI = {https://doi.org/10.1007/s00158-016-1551-x}
}

@article{PreventDripping,
  title={On preventing the dripping effect of overhang constraints in topology optimization for additive manufacturing},
  author={Garaigordobil, Alain and Ansola, Rub{\'e}n and Fernandez de Bustos, Igor},
  journal={Struct. Multidiscip. Optim.},
  fjournal = {Structural and Multidisciplinary Optimization},
  volume={64},
  number={6},
  pages={4065--4078},
  year={2021},
  publisher={Springer},
  DOI = {https://doi.org/10.1007/s00158-021-03077-w}
}

@book{BenSig03,
    AUTHOR = {Bends{\o}e, M. P. and Sigmund, O.},
     TITLE = {Topology Optimization:
      Theory, Methods and Applications},
 PUBLISHER = {Springer, Berlin},
      YEAR = {2003},
      ISBN = {3-540-42992-1},
MRREVIEWER = {Michal\ Ko\v cvara},
}

@article {bendsoe1988HomogenizationApproach,
    AUTHOR = {Bends{\o}e, Martin Philip and Kikuchi, Noboru},
     TITLE = {Generating optimal topologies in structural design using a
              homogenization method},
   JOURNAL = {Comput. Methods Appl. Mech. Engrg.},
  FJOURNAL = {Computer Methods in Applied Mechanics and Engineering},
    VOLUME = {71},
      YEAR = {1988},
    NUMBER = {2},
     PAGES = {197--224},
       DOI = {https://doi.org/10.1016/0045-7825(88)90086-2}
}

@article{RAMP,
  title={An alternative interpolation scheme for minimum compliance topology optimization},
  author={Stolpe, Mathias and Svanberg, Krister},
  journal={Struct. Multidiscip. Optim.},
  volume={22},
  number={2},
  pages={116--124},
  year={2001},
  publisher={Springer}
}

@article{SigmundReview,
  title={Topology optimization approaches: A comparative review},
  author={Sigmund, Ole and Maute, Kurt},
  journal={Struct. Multidiscip. Optim.},
  volume={48},
  number={6},
  pages={1031--1055},
  year={2013},
  publisher={Springer}
}

@misc{EL91,
  title={A generalised diffusion equation for phase separation of a multi-component mixture with interfacial free energy},
  author={Charles M. Elliott and Stefan Luckhaus},
  year={1999},
  month = {10},
  note = {Retrieved from the University Digital Conservancy (University of Minnesota), Preprint Series 887},
  url = {https://conservancy.umn.edu/items/6727e5ca-8ebd-47f3-8cd5-c3996aa47d5a},
  preprint ={Series 887}
}

@article{AllaireTopologicalDeriv,
title = {Structural optimization using sensitivity analysis and a level-set method},
journal = {J. Comput. Phys.},
volume = {194},
number = {1},
pages = {363-393},
year = {2004},
issn = {0021-9991},
doi = {https://doi.org/10.1016/j.jcp.2003.09.032},
author = {Grégoire Allaire and François Jouve and Anca-Maria Toader}
}

@article{Burger2003LevelSet,
    AUTHOR = {Burger, Martin},
     TITLE = {A framework for the construction of level set methods for
              shape optimization and reconstruction},
   JOURNAL = {Interfaces Free Bound.},
    VOLUME = {5},
      YEAR = {2003},
    NUMBER = {3},
     PAGES = {301--329},
      ISSN = {1463-9963,1463-9971},
   MRCLASS = {49Q10 (35R30)},
  MRNUMBER = {1998617},
MRREVIEWER = {Thomas\ Slawig},
       DOI = {https://doi.org/10.4171/IFB/81}
}

@article{Burger2004TopologicalDerivative,
    AUTHOR = {Burger, Martin and Hackl, Benjamin and Ring, Wolfgang},
     TITLE = {Incorporating topological derivatives into level set methods},
   JOURNAL = {J. Comput. Phys.},
    VOLUME = {194},
      YEAR = {2004},
    NUMBER = {1},
     PAGES = {344--362},
      ISSN = {0021-9991,1090-2716},
   MRCLASS = {49Q12 (65D99 65M99)},
  MRNUMBER = {2033389},
       DOI = {https://doi.org/10.1016/j.jcp.2003.09.033}
}

@article{Wang2003LevelSet,
    AUTHOR = {Wang, Michael Yu and Wang, Xiaoming and Guo, Dongming},
     TITLE = {A level set method for structural topology optimization},
   JOURNAL = {Comput. Methods Appl. Mech. Engrg.},
  FJOURNAL = {Computer Methods in Applied Mechanics and Engineering},
    VOLUME = {192},
      YEAR = {2003},
    NUMBER = {1-2},
     PAGES = {227--246},
MRREVIEWER = {Michal\ Ko\v cvara},
       DOI = {https://doi.org/10.1016/S0045-7825(02)00559-5},
}

@article{Wang2004PhaseField,
  title={Phase field: A variational method for structural topology optimization},
  author={Wang, Michael Yu and Zhou, Shiwei},
  journal={CMES},
  volume={6},
  number={6},
  pages={547--566},
  year={2004},
  publisher={Tech Science Press}
}

@article{ReviewAMNgo,
  title={Additive manufacturing (3D printing): A review of materials, methods, applications and challenges},
  author={Ngo, Tuan D and Kashani, Alireza and Imbalzano, Gabriele and Nguyen, Kate TQ and Hui, David},
  journal={Compos B Eng.},
  volume={143},
  pages={172--196},
  year={2018},
  publisher={Elsevier}
}

@article{ReviewAM,
author = {Wong, Kaufui V. and Hernandez, Aldo},
title = {A review of additive manufacturing},
journal = {ISRN Mech. Eng.},
volume = {208760},
number = {},
doi = {https://doi.org/10.5402/2012/208760},
year = {2012}
}

@article{TrendsAM,
  title={Current and future trends in topology optimization for additive manufacturing},
  author={Liu, Jikai and Gaynor, Andrew T and Chen, Shikui and Kang, Zhan and Suresh, Krishnan and Takezawa, Akihiro and Li, Lei and Kato, Junji and Tang, Jinyuan and Wang, Charlie CL and others},
  journal={Struct. Multidiscip. Optim.},
  volume={57},
  number={6},
  pages={2457--2483},
  year={2018},
  publisher={Springer}
}

@article{GLNS,
    AUTHOR = {Garcke, Harald and Lam, Kei Fong and N\"urnberg, Robert and
              Signori, Andrea},
     TITLE = {Overhang penalization in additive manufacturing via phase
              field structural topology optimization with anisotropic
              energies},
   JOURNAL = {Appl. Math. Optim.},
  FJOURNAL = {Applied Mathematics and Optimization},
    VOLUME = {87},
      YEAR = {2023},
    number = {44},
     pages = {50 pp.},
      ISSN = {0095-4616,1432-0606},
   MRCLASS = {49K20 (49J45 49J50 49K40 74P15)},
  MRNUMBER = {4565011},
MRREVIEWER = {Velusamy\ Vijayakumar},
       DOI = {https://doi.org/10.1007/s00245-022-09939-z}
}

@article{WIAS,
    AUTHOR = {Ebeling-Rump, Moritz and H\"omberg, Dietmar and Lasarzik,
              Robert and Petzold, Thomas},
     TITLE = {Topology optimization subject to additive manufacturing
              constraints},
   JOURNAL = {J. Math. Ind.},
  FJOURNAL = {Journal of Mathematics in Industry},
    VOLUME = {11},
      YEAR = {2021},
    number = {19},
      ISSN = {2190-5983},
   MRCLASS = {74Pxx (49Q10 49Q12)},
  MRNUMBER = {4336001},
MRREVIEWER = {Andrzej\ M.\ My\'sli\'nski},
       DOI = {https://doi.org/10.1186/s13362-021-00115-6}
}

@article{AllaireOverview1,
    AUTHOR = {Allaire, G. and Dapogny, C. and Estevez, R. and Faure, A. and
              Michailidis, G.},
     TITLE = {Structural optimization under overhang constraints imposed by
              additive manufacturing technologies},
   JOURNAL = {J. Comput. Phys.},
    VOLUME = {351},
      YEAR = {2017},
     PAGES = {295--328},
      ISSN = {0021-9991,1090-2716},
   MRCLASS = {74P05 (90B30)},
  MRNUMBER = {3713427},
       DOI = {https://doi.org/10.1016/j.jcp.2017.09.041}
}

@article{AllaireOverview2,
    AUTHOR = {Allaire, Gr\'egoire and Dapogny, Charles and Estevez, Rafael
              and Faure, Alexis and Michailidis, Georgios},
     TITLE = {Structural optimization under overhang constraints imposed by
              additive manufacturing processes: {A}n overview of some recent
              results},
   JOURNAL = {Appl. Math. Nonlinear Sci.},
  FJOURNAL = {Applied Mathematics and Nonlinear Sciences},
    VOLUME = {2},
      YEAR = {2017},
    NUMBER = {2},
     PAGES = {385--402},
      ISSN = {2444-8656},
   MRCLASS = {74P20 (49M30 65K10 74P05)},
  MRNUMBER = {3896926},
       DOI = {https://doi.org/10.21042/AMNS.2017.2.00031}
}

@article{Allaire2,
    AUTHOR = {Allaire, Gr\'egoire and Dapogny, Charles and Faure, Alexis and
              Michailidis, Georgios},
     TITLE = {Shape optimization of a layer by layer mechanical constraint
              for additive manufacturing},
   JOURNAL = {C. R. Math. Acad. Sci. Paris},
  FJOURNAL = {Comptes Rendus Math\'ematique. Acad\'emie des Sciences. Paris},
    VOLUME = {355},
      YEAR = {2017},
    NUMBER = {6},
     PAGES = {699--717},
      ISSN = {1631-073X,1778-3569},
   MRCLASS = {49Q10 (90B30)},
  MRNUMBER = {3661554},
       DOI = {https://doi.org/10.1016/j.crma.2017.04.008}
}

@article{AllaireBogoselSupports1,
    AUTHOR = {Allaire, Gr\'egoire and Bihr, Martin and Bogosel, Beniamin},
     TITLE = {Support optimization in additive manufacturing for geometric
              and thermo-mechanical constraints},
   JOURNAL = {Struct. Multidiscip. Optim.},
    VOLUME = {61},
      YEAR = {2020},
    NUMBER = {6},
     PAGES = {2377--2399},
      ISSN = {1615-147X,1615-1488},
   MRCLASS = {74P15 (49Q10 49Q12)},
  MRNUMBER = {4119056},
       DOI = {https://doi.org/10.1007/s00158-020-02551-1}
}

@article{AllaireBogoselSupports2,
    AUTHOR = {Allaire, Gr\'egoire and Bogosel, Beniamin},
     TITLE = {Optimizing supports for additive manufacturing},
   JOURNAL = {Struct. Multidiscip. Optim.},
    VOLUME = {58},
      YEAR = {2018},
    NUMBER = {6},
     PAGES = {2493--2515},
      ISSN = {1615-147X,1615-1488},
   MRCLASS = {74P05 (49Q12)},
  MRNUMBER = {3878711},
MRREVIEWER = {Asatur\ Zh.\ Khurshudyan},
       DOI = {https://doi.org/10.1007/s00158-018-2125-x}
}

@article{Kuo2018support,
    AUTHOR = {Kuo, Yu-Hsin and Cheng, Chih-Chun and Lin, Yang-Shan and San,
              Cheng-Hung},
     TITLE = {Support structure design in additive manufacturing based on
              topology optimization},
   JOURNAL = {Struct. Multidiscip. Optim.},
  FJOURNAL = {Structural and Multidisciplinary Optimization},
    VOLUME = {57},
      YEAR = {2018},
    NUMBER = {1},
     PAGES = {183--195},
       DOI = {https://doi.org/10.1007/s00158-017-1743-z},
}

@article{ReviewSupports,
  title={Support structures for additive manufacturing: {A} review},
  author={Jiang, Jingchao and Xu, Xun and Stringer, Jonathan},
  journal={J. Manuf. Mater. Process.},
  volume={2},
  NUMBER = {64},
  year={2018},
  publisher={MDPI}
}

@article{MIRZENDEHDEL20161,
title = {Support structure constrained topology optimization for additive manufacturing},
journal = {Comput. Aided Des.},
volume = {81},
pages = {1-13},
year = {2016},
issn = {0010-4485},
doi = {https://doi.org/10.1016/j.cad.2016.08.006},
author = {Amir M. Mirzendehdel and Krishnan Suresh}
}

@article{langelaar2018combined,
    AUTHOR = {Langelaar, Matthijs},
     TITLE = {Combined optimization of part topology, support structure
              layout and build orientation for additive manufacturing},
   JOURNAL = {Struct. Multidiscip. Optim.},
  FJOURNAL = {Structural and Multidisciplinary Optimization},
    VOLUME = {57},
      YEAR = {2018},
    NUMBER = {5},
     PAGES = {1985--2004},
       DOI = {https://doi.org/10.1007/s00158-017-1877-z},
}

@article{Takezawa2010,
  title={Shape and topology optimization based on the phase field method and sensitivity analysis},
  author={Takezawa, Akihiro and Nishiwaki, Shinji and Kitamura, Mitsuru},
  journal={J. Comput. Phys.},
  volume={229},
  number={7},
  pages={2697--2718},
  year={2010},
  publisher={Elsevier}
}

@book{Bertsekas,
    AUTHOR = {Bertsekas, Dimitri P.},
     TITLE = {Nonlinear programming},
    SERIES = {Athena Sci. Optim. Comput. Ser.},
 PUBLISHER = {Athena Scientific, Belmont, MA},
      YEAR = {2016},
MRREVIEWER = {Stephan\ Dempe}
}

@book{GruverSachs1981,
    AUTHOR = {Gruver, W. A. and Sachs, E.},
     TITLE = {Algorithmic Methods in Optimal Control},
    SERIES = {Research Notes in Mathematics 47},
 PUBLISHER = {Pitman Pub.,  Boston},
    VOLUME = {47},
      YEAR = {1981},
      ISBN = {0-273-08473-9},
  MRNUMBER = {604361},
MRREVIEWER = {Kimio\ Kanai},
}

@article{WANG2026Uncert,
    AUTHOR = {Wang, Zhuoheng and Xie, Wenxuan and Kim, Junseok and Li,
              Yibao},
     TITLE = {Efficient phase field structural design algorithm for
              reliability-based topology optimization with material
              uncertainties},
   JOURNAL = {Eng. Anal. Bound. Elem.},
  FJOURNAL = {Engineering Analysis with Boundary Elements},
    VOLUME = {183},
      YEAR = {2026},
     PAGES = {Paper No. 106618},
       DOI = {https://doi.org/10.1016/j.enganabound.2025.106618},
}

@article{CHEN2026SemiImpl,
    AUTHOR = {Chen, Huangxin and Dong, Piaopiao and Wang, Dong and Wang,
              Xiao-Ping},
     TITLE = {A robust and stable phase field method for structural topology
              optimization},
   JOURNAL = {J. Comput. Phys.},
  FJOURNAL = {Journal of Computational Physics},
    VOLUME = {547},
      YEAR = {2026},
     PAGES = {Paper No. 114531},
       DOI = {https://doi.org/10.1016/j.jcp.2025.114531}
}

@article{Garcke4DPrinting,
  title={Phase field topology optimisation for 4D printing},
  author={Garcke, Harald and Lam, Kei Fong and N{\"u}rnberg, Robert and Signori, Andrea},
  journal={ESAIM Control Optim. Calc. Var.},
  FJOURNAL = {ESAIM Control Optim. Calc. Var.},
  volume={29},
  year={2023},
  PAGES = {Paper No. 24, 46 pp.},
  publisher={EDP Sciences},
  doi = {https://doi.org/10.1051/cocv/2023012}
}
